\documentclass[11pt,leqno,oneside]{amsart}
\usepackage{amsmath,amssymb,,comment,graphicx,psfrag}
\usepackage{amssymb,version,graphicx,fancybox,mathrsfs,multirow}
\usepackage{bm,mathrsfs}
\usepackage{mathtools}
\usepackage{color}
\usepackage{enumerate}
\usepackage{pdfsync}
\usepackage{hhline}
\usepackage{float}
\usepackage[margin=0.95in]{geometry}
\usepackage{multicol}
\usepackage{multirow}

\usepackage{booktabs}
\usepackage[colorlinks,linktocpage,linkcolor=blue]{hyperref}
\usepackage{placeins}
\usepackage{subcaption} 

\usepackage{float}
\allowdisplaybreaks
\usepackage{algorithmic}
\def\baselinestretch{1.0}
  \newtheorem{assumption}{Assumption}[section]
\newcommand{\al}{\alpha}
\newcommand{\fy}{\varphi}

\def\Dal{{\partial_t^\al}}
\def\bDal{{\bar\partial_\tau^\al}}
\def\d{{\rm d}}
\def\al{\alpha}

\def\fy{\varphi}

\def\Om{\Omega}
\def\Dal{{\partial^\alpha_t}}

\def\dH#1{\dot H^{#1}(\Omega)}

\def\A{\mathcal{A}}
\def\lb{\left(}
\def\rb{\right)}
\def\lsb{\left[}
\def\rsb{\right]}
\def\lnorm{\left\|}
\def\rnorm{\right\|}

\def\II{{(\Omega)}}
\theoremstyle{plain}
\newtheorem{theorem}{Theorem}[section]

\newtheorem{remark}{Remark}[section]

\newtheorem{lemma}{Lemma}[section]

\numberwithin{equation}{section}

\usepackage{graphicx}
\usepackage{epstopdf}

\usepackage{titlesec}
\titleformat{\section}{\vskip10pt\large\bfseries}{\thesection.}{0.5em}{\centering\vspace{5pt}}
\titleformat{\subsection}{\vskip10pt\normalsize\bfseries}{\thesubsection.}{0.5em}{}

\def\d{{\rm d}}
\def\al{\alpha}

\def\Dal{{\partial^\alpha_t}}

\def\dH#1{\dot H^{#1}(\Omega)}

\def\II{(\Omega)}
\usepackage[ruled,linesnumbered]{algorithm2e}

\newcommand{\ds}{\mathrm ds}
\newcommand{\dx}{\mathrm dx}

\newcommand{\dz}{\mathrm dz}

\renewcommand{\d}{{\,\rm d}}
\def\Dal{{\partial_t^\al}}

\def\Om{\Omega}
\def\II{(\Om)}

\def\dH#1{\dot H^{#1}(\Omega)}

\def\bDal{\bar\partial_\tau^\alpha}

\def \contour{{\Gamma_{\theta,\sigma}}}
\def\L2Om{{L^2(\Omega)}}

\def\T{\mathcal{T}}

\def\Ih{\mathcal{I}_h}

\def\Ih{I_h}
\graphicspath{{./figures/}}

\begin{document}

\title[Simultaneous Reconstruction of Initial Condition and Potential]{Numerical Analysis of Simultaneous Reconstruction of Initial Condition and Potential in Subdiffusion}

\author[Xu Wu]{$\,\,$Xu Wu$\,$}
\address{Department of Applied Mathematics,
The Hong Kong Polytechnic University, Kowloon, Hong Kong}
\email {polyuxu.wu@polyu.edu.hk}

\author[Jiang Yang]{$\,\,$Jiang Yang$\,$}
\address{Department of Mathematics, SUSTech International Center for Mathematics \& National Center for Applied Mathematics Shenzhen (NCAMS), Southern University of Science and Technology, Shenzhen, China }
\email {yangj7@sustech.edu.cn}

\author[Zhi Zhou]{$\,\,$Zhi Zhou$\,$}
\address{Department of Applied Mathematics,
The Hong Kong Polytechnic University, Kowloon, Hong Kong}
\email {zhizhou@polyu.edu.hk}


\begin{abstract}
This paper investigates the simultaneous identification of a spatially dependent potential and the initial condition in a subdiffusion model based on two terminal observations. The existence, uniqueness, and conditional stability of the inverse problem are established under weak regularity assumptions through a constructive fixed-point iteration approach.
The theoretical analysis further inspires the development of an easy-to-implement iterative algorithm. A fully discrete scheme is then proposed, combining the finite element method for spatial discretization, convolution quadrature for temporal discretization, and the quasi-boundary value method to handle the ill-posedness of recovering the initial condition. Inspired by the conditional stability estimate, we demonstrate the linear convergence of the iterative algorithm and provide a detailed error analysis for the reconstructed initial condition and potential.
The derived \textsl{a priori} error estimate offers a practical guide for selecting regularization parameters
and discretization mesh sizes based on the noise level. Numerical experiments are provided to illustrate and support our theoretical findings.
\end{abstract}

\maketitle

\renewcommand{\baselinestretch}{0.95}
\setlength\abovedisplayskip{4.0pt}
\setlength\belowdisplayskip{4.0pt}

\section{Introduction}
This work aims to investigate the identification of a spatially dependent potential and the initial condition in a subdiffusion model  with a space-dependent
potential and its rigorous numerical analysis. Let $\Omega\subset\mathbb{R}^d$ ($d=1,2,3$) be a convex
polyhedral domain with a boundary $\partial\Omega$. Fixing $T>0$ as the final time, we  consider the following  initial-boundary  value problem for the  diffusion model with $\alpha\in(0,1)$:
 \begin{equation}\label{eqn:pdeo}
 \begin{cases}
  \begin{aligned}
     \Dal U(x,t) - \Delta U(x,t) +q(x) U(x,t)&=f(x), &&(x,t)\in \Omega\times(0,T],\\
    U(x,t)&=b(x),&&(x,t)\in \partial\Omega\times(0,T],\\
    U(x,0)&=v(x),&&x\in\Omega,
  \end{aligned}
  \end{cases}
 \end{equation}
where $v$ denotes the initial condition, $b$ and $f$ are space-dependent boundary data and source term, respectively.
The function $q$ refers to the radiativity or
reaction coefficient or potential, dependent on the specific applications.
Throughout, we assume that the potential $q$ is space-dependent.
The fractional order $\alpha \in (0,1)$ is fixed, and $\partial_t^{\alpha} U$ denotes the Djrbashian--Caputo fractional derivative of order $\alpha$ in time, defined by \cite[Definition~2.3]{Jin:book2021}
\begin{equation*}
    \partial_t^{\alpha} U(t)
    = \frac{1}{\Gamma(1 - \alpha)}
    \int_0^t (t-s)^{-\alpha} \partial_s U(s)\, {\rm d}s,
\end{equation*}
where $\Gamma(z) = \int_0^{\infty} s^{z-1} e^{-s}\, {\rm d}s$ for $\Re(z) > 0$ is Euler's Gamma function.
For sufficiently smooth $U$, $\partial_t^{\alpha} U$ reduces to the classical derivative $\partial_t U$ as $\alpha \to 1^-$.

 The model \eqref{eqn:pdeo} is commonly used to describe subdiffusive dynamics, where the mean squared displacement (MSD) grows slower than in normal diffusion, following a power-law $t^\alpha$ instead of linear growth. Subdiffusive processes frequently arise in biological systems, such as the motion of proteins or organelles in crowded cellular environments, in geophysics, where fluid flow in porous media deviates from normal diffusion due to complex structures, and in pollutant dispersion in natural environments, where heterogeneous landscapes or obstacles slow down the spread of contaminants. Additionally, the model is used in materials science for heat conduction in fractal or memory-dependent materials, and in finance to describe anomalous asset price dynamics. For a comprehensive overview of the derivation of these mathematical models and their diverse applications, we refer readers to \cite{MetzlerKlafter:2000,Du:book,Jin:book2021}.

In this paper, we focus on the inverse problem  associated with the subdiffusion model \eqref{eqn:pdeo}, aiming to recover the unknown potential term $q^\dag(x)\in L^\infty(\Omega)$ and initial condition $ v^\dag\in L^2\II$ 
 from two terminal observations:
\begin{equation}\label{eqn:ob_ex}
   U(x,T_1; q^\dag,v^\dag)=g_1(x),\quad  U(x,T_2; q^\dag,v^\dag)=g_2(x)\quad \mbox{in }\Omega.
\end{equation}
Here we denote the solution to \eqref{eqn:pdeo} corresponding to the potential $q$ and initial condition $v$ by $U(x,t;q,v)$.
In practice, observational data often contains noise. In this work, we consider the empirical observational data $g_1^\delta$ and $g_2^\delta$ satisfying
\begin{align*}
   \|g_1^\delta - g_1\|_{L^2\II}\le\delta \quad \text{and}\quad \|g_2^\delta - g_2\|_{ L^2\II} \le \delta , 
\end{align*}
where $\delta$ denotes the noise level.
The primary objectives of this study are to examine the solvability of the inverse problem for the identification of the initial condition and potential, to develop a robust numerical scheme for its solution, and to derive an error estimate for the numerical reconstruction of both the initial condition and potential. The derived error estimate will serve as a fundamental criterion for the selection of optimal discretization parameters, including the spatial mesh size, temporal step size, and regularization parameter, thereby enhancing the accuracy and efficiency of the numerical method.

The inverse potential problem arises in various practical contexts, where
$q^\dag$ typically represents the radiativity coefficient in heat conduction \cite{YZ:2001} or the perfusion coefficient in Pennes' bio-heat equation in human physiology \cite{Pennes:1948}.
Theoretical investigations of this problem in diffusion equations with final-time observational data have a long history; see, for instance, \cite{Isakov:1991, ChoulliYamamoto:1996, ChoulliYamamoto:1997, ChenJiangZou:2020, KlibanovLiZhang:2020} and the references therein.  
In \cite{Isakov:1991}, Isakov established the uniqueness and (conditional) existence of solutions for the parabolic inverse potential problem by developing a unique continuation principle and a constructive fixed-point iteration scheme. Zhang and Zhou adopted a similar strategy in \cite{ZhangZhou:2017} for a one-dimensional time-fractional subdiffusion model. They used the spectrum perturbation argument (\cite[Lemma 2.2]{ZhangZhou:2017} and \cite{Trubowitz:1987}) to prove that the fixed point iteration is a contraction, leading to immediate conclusions about uniqueness and existence. Their subsequent work \cite{zhang2022identification} generalized these results to higher dimensional cases, establishing conditional stability in Hilbert spaces under suitable assumptions on the problem data, and further providing an error analysis for finite element reconstructions with smooth data and noise.  
For small time $T$, the conditional stability results have been established in \cite{ChoulliYamamoto:1996,ChoulliYamamoto:1997} for normal diffusion and in  \cite{JinZhou:IP2021-a} for the case where $\alpha \in (0,1)$.
Kaltenbacher and Rundell \cite{KaltenbacherRundell:2019a} demonstrated the invertibility of the linearized map (of the direct problem) from the space $L^2(\Omega)$ to $H^2(\Omega)$ using a Paley--Wiener type argument and a strong maximum principle.
In \cite{KR:2020}, they explored the simultaneous recovery of several parameters from overposed data consisting of $u(T)$. 
Chen et al.~\cite{ChenJiangZou:2020} studied the parabolic inverse potential problem with observational data on $[T_0, T_1]\times\Omega$, proving conditional stability in negative Sobolev spaces.
More recently, Jin et al.~\cite{JLQZ:2021} considered the same data configuration and established a weighted $L^2$ stability estimate, which implies a H\"older-type stability in the standard $L^2$ norm under a suitable positivity condition.
For additional studies on the inverse potential problem for (sub)diffusion models with various types of observational data, see \cite{KianYamamoto:2019, MillerYamamoto:2013, KR:2020-b} and the references therein. 

The backward subdiffusion problem, which involves recovering the initial data from an over-posed terminal measurement, has garnered considerable attention in recent literature. For the case of $\alpha = 1$, it is well-known that the parabolic backward problem is severely ill-posed. However, in the subdiffusion case with $\alpha \in (0,1)$, the backward problem is only mildly ill-posed, as highlighted in the pioneering work of Sakamoto and Yamamoto \cite{sakamoto2011initial}. This result has inspired extensive research on the development and analysis of regularization methods for solving the backward subdiffusion problem \cite{liu2010backward, wang2013total, wei2014modified, yang2013solving, wei2019variational, wu2025numerical}.
Despite these theoretical advances, studies on numerical discretization and rigorous error analysis remain relatively limited. Zhang etal.\cite{zhang2020numerical} analyzed a fully discrete scheme for the backward problem and later extended their approach to handle time-dependent coefficients using a perturbation argument \cite{zhang2023stability}, as well as nonlinear models via a fixed-point iteration argument \cite{wu2025numerical}. Furthermore, it is noteworthy that the fractional backward model ($\alpha \in (0,1)$) can act as a regularization for the classical backward heat problem ($\alpha = 1$), as shown in \cite{KaltenbacherRundell:2019, KaltenbacherRundell:2025}.



In this study, we investigate the inverse problem of identifying the potential and initial condition in the model described by equation \eqref{eqn:pdeo} using two terminal observations. This problem is particularly challenging due to the increased complexity of the coefficients-to-state map when compared to the case of a single coefficient, and the decoupling method remains ill-defined. We propose a fixed-point iterative algorithm to recover the initial data and potential. Specifically, we define an operator \( K \) derived from the PDE \eqref{eqn:pdeo} as follows:
\[
 K\psi(x) = \frac{f(x) - \Dal U(x,T_2;\psi,B(\psi,g_1)) + \Delta g_2(x)}{g_2(x)},
\]
where \( B(\cdot, \cdot) \) denotes the backward operator, with \( B(\psi,g) \) representing the initial condition corresponding to the observational data \( g \) at \( T_1 \) with potential \( \psi \). We observe that the exact potential \( q^\dag \) is a fixed point of \( K \) for the observations described in \eqref{eqn:ob_ex}.
We demonstrate that for a fixed \( T_1 \), there exists a \( T_0 \) such that for all \( T_2 \ge T_0 \), the operator \( K \) becomes a contracting mapping, ensuring the existence of a unique fixed point. 
Moreover, this novel argument also yields a Lipschitz-type stability result in Hilbert spaces (see Theorem \ref{thm:cond-stab})
\begin{equation}\label{eqn:stab-intro}
\begin{split}
\| q_1 - q_2 \|_{L^2\II} + \| v_1- v_2\|_{L^2\II} \le C \sum_{i=1}^2 \|  U(T_i;q_1, v_1) - U(T_i;q_2, v_2) \|_{H^2\II},
\end{split}
\end{equation}
where the proof heavily relies on the smoothing properties and asymptotics of the solution operators. 
This conditional stability plays an essential role in the numerical analysis of our reconstruction algorithm with fully discretization in space and time. It is important to note that Yamamoto and Zou, in \cite{YZ:2001}, studied the simultaneous recovery of the initial data and potential using observations taken at one terminal time and a space-time subdomain, relying on the Carleman estimate. In contrast, our paper adopts entirely different approaches that are more readily adaptable to numerical analysis.

The stability estimate \eqref{eqn:stab-intro} implies that the inverse problem is ill-posed, exhibiting a loss equivalent to
a second-order derivative. Therefore, regularization techniques, particularly Tikhonov regularization, are designed to address this ill-posedness \cite{EnglKunischNeubauer:1989, YZ:2001}. In practical computations, it is necessary to discretize the continuous regularized formulation, which introduces discretization errors. However, obtaining convergence rates for discrete approximations is generally very challenging due to the nonconvexity of the regularized function, which arises from the high degree of nonlinearity in the parameter-to-state map. 
In this work, we apply the quasi-boundary
value method to handle the ill-posedness of the backward problem and propose an easy-to-implement iterative solver to recover both initial data and potential. In particular, we discretize the continuous problem \eqref{eqn:pdeo} using the Galerkin finite element method with conforming piecewise linear finite elements for space discretization, and the convolution quadrature generated by the backward Euler method (CQ-BE) for time discretization. We prove that, with properly chosen $T_1$ and $T_2$, the iteration generates a sequence that converges linearly to a fixed point $q^*$ and the numerical reconstructions satisfy the following \textsl{a priori} error estimate (cf. Theorem \ref{thm:err-fully}):
\begin{equation*}
\begin{split}
\| q^* - q^\dag \|_{L^2\II} +  & \lnorm U_{\gamma,h}^{\delta,0}(q^*)-v^\dag\rnorm_{L^2\II}\\ &\le c\left(\gamma^{-1}\delta+\gamma^{\frac{s_1}{2}}+\gamma^{-1} h^2+ \tau +\delta H^{-\frac{3}{2}} \bar H^{-\frac{1}{2}}+\bar H^{\frac{3}{2}}  H^{-\frac{3}{2}} + H^{\min\{s_1,s_2\}}\right), 
\end{split}
\end{equation*}
where $v^\dag \in H^{s_1}\II$ and $q^\dag \in H^{s_2}\II$, with constants $s_1, s_2 \in (0,2]$. Here, $\gamma$ is the regularization parameter, $h$ denotes the spatial mesh size, and $\tau$ represents the time step size. The additional spatial mesh sizes, $H$ and $\bar{H}$, arise from the two-grid algorithm used to approximate the Laplacian of noisy data. 
The \textsl{a priori} error estimate provides practical guidance for selecting regularization and discretization parameters $\gamma$, $h$, $H$, $\bar{H}$, and $\tau$ according to the noise level $\delta$. By selecting the parameters as  
\[
\gamma \sim \delta^\frac{2}{2+s_1}, \quad h \sim \delta^{\frac{1}{2}},  \quad H \sim \delta^\frac{3}{6+4\min\{s_1,s_2\}}, \quad \text{and} \quad \bar{H} \sim \delta^{\frac{1}{2}}, \quad \tau \sim \delta^\frac{s_1}{2+s_1},
\]  
one can achieve reconstructions with an optimal error of order 
$O(\delta^\frac{3\min\{s_1,s_2\}}{6+4\min\{s_1,s_2\}})$, which is fully supported by numerical experiments. The proposed numerical scheme and error analysis not only rely on the numerical analysis for the direct problem with minimal regularity \cite{JinZhou:2023book} but are also strongly motivated by the conditional stability analysis, emphasizing the connection between theoretical insights and numerical implementation. Note that the current results significantly improve upon those in earlier work, such as \cite[Theorem 4.12]{zhang2022identification}, which provided a suboptimal error estimate but required highly smooth problem data, globally continuous observational data, and uniform rectangular or triangular meshes (see Remark \ref{rem:reg-err}  and Figure \ref{rateQP1_Conti} for more discussion). 


The remainder of the paper is organized as follows. Section \ref{sec:well_pos} presents preliminary results and establishes the existence and uniqueness of the reconstruction for both the initial condition and potential.
Then in Section \ref{sec:regu}, we investigate the regularization and convergence analysis for the backward problem.
The numerical reconstruction with a fully discretization scheme is developed and analyzed in Section~\ref{sec:fully_scheme},
where we show the linear convergence of the iterative algorithm and
establish \textsl{a priori} error estimates (in terms of discretization parameters and noise level) for the reconstructed potential and initial condition.
Finally, we present some
two-dimensional numerical results to illustrate and support our theoretical findings in Section \ref{sec:num}.  

We now conclude with some useful notations.
For any $k\geq 0$ and $p\geq1$, the
space $W^{k,p}(\Omega)$ denotes the standard Sobolev spaces of the $k$-th order, and we write $H^k(\Omega)$,
when $p=2$. The notation $(\cdot,\cdot)$ denotes the $L^2(\Omega)$ inner product.
 Throughout, the notations $c$ and $C$, with or without a subscript, denote generic constants
that may change at each occurrence, but they are always independent of the regularization parameter $\gamma$, spatial mesh size $h, H, \bar{H}$, temporal step size $\tau$, and noise level $\delta$.

\section{Uniqueness and conditional stability of the inverse problem
}\label{sec:well_pos} 
The aim of this section is to establish the well-posedness of the inverse potential problem and backward problem for the subdiffusion equation \eqref{eqn:pdeo}, specifically addressing the conditional stability of the reconstruction of the initial condition and potential from two terminal observations.

\subsection{Preliminary} 
First,  by splitting the  solution to  $U(x)=u(x)+D(x)$ in the model \eqref{eqn:pdeo}, with $D(x)$ satisfying
\begin{equation}\label{eqn:pdedx}
    -\Delta  D(x)=0,\ x\in\Omega,\  D(x)=b(x), \ x\in \partial\Omega,
\end{equation}
we can observe that $u(x)$ solves:
  \begin{equation}\label{eqn:pde}
 \begin{cases}
  \begin{aligned}
     \Dal u(x,t) - \Delta u (x,t) +q(x) u(x,t)&=f(x)-q(x)D(x), &&(x,t)\in \Omega\times(0,T],\\
    u(x,t)&=0,&&(x,t)\in \partial\Omega\times(0,T],\\
    u(x,0)&=v(x)-D(x)=:\bar v(x),&&x\in\Omega.
  \end{aligned}
  \end{cases}
 \end{equation}

Next, we collect some preliminary settings for the controllable
conditions 
\begin{equation}\label{eqn:cond1}
f \in L^2\II, \quad  D\in  L^\infty \II,\quad   D(x)\ge M>0, ~~\text{ a.e. in } \Omega
\end{equation}
and the (unknown) exact potential $q^\dag\in L^\infty\II$ and (unknown) initial condition  $v^\dag\in L^2\II$. Throughout,
we assume that the exact potential  $q^\dag$  belongs to the following admissible set 
\begin{equation}\label{admissible_q}
 q^\dag \in \A=\{q\in L^\infty\II: 0\le q\le m\}.
\end{equation}
Now we recall the maximum principle of the time-fractional diffusion model~\eqref{eqn:pde}, which has been established in \cite{luchko2017maximum}.
\begin{lemma}\label{positive}
 Let   $u$  be the solution of \eqref{eqn:pde} with potential $q\in \A$ and  $\bar v,\ f\in L^2({\Omega})$ with $\bar v\ge 0,\  f-qD\ge 0,\ a.e. $ in $\Omega$. Then there holds $u(t)\ge 0$, for $t> 0$.
\end{lemma}

Now we present the solution representation of the initial-boundary value problem \eqref{eqn:pde}.
Let  $I$ be the identity operator and $A_q$ be the realization of $ -\Delta + q I$   with the homogeneous Dirichlet boundary condition, with the domain
$$ \text{Dom}(A_q) = \{ \psi \in H_0^1\II:\, A_q  \psi \in L^2\II \}=H_0^1\II\cap H^2\II .$$
$\{(\lambda_{n}(q), \fy_{n}(q))\}_{n=1}^\infty$ denote the eigenpairs of $A_q$, with $\{\fy_{n}(q)\}_{n=1}^\infty$ forms an orthonormal basis in $L^2(\Omega)$. 
 Throughout, we denote by~$\dH s$ the Hilbert space induced by the norm
$$\|\psi\|_{\dH s}^2:=\|A_q^\frac{s}{2}\psi\|_{L^2\II}^2=\sum_{n=1}^{\infty}\lambda_n^s (q)( \psi,\fy_{n}(q) )^2, \  s\ge0.$$
$\|\psi\|_{\dH 1}$ is a norm in $H_0^1(\Om)$, and $\|\psi\|_{\dH 2}=\|A_q \psi\|$ is a norm in~$H^2(\Om)\cap H^1_0(\Om)$. 
It is easy to see the space $\dH s$ is the interpolation space $(L^2\II, H^2\II\cap H_0^1\II)_{\frac{s}2}$.

If $q\in\A $, for any $ \psi \in H_0^1\II\cap H^2\II$, with constants $c_1$ and $c_2$ independent of $q$, the full elliptic regularity implies
(see e.g.  \cite[Lemma 2.1]{LiSun:2017} and \cite[Theorems 3.3 and 3.4]{GruterWidman:1982}) 
\begin{equation*}\label{eqn:equiv-n} 
c_1\|  \psi  \|_{H^2\II} \le \| A_q  \psi  \|_{L^2\II} + \|  \psi  \|_{L^2\II} \le c_2\|  \psi  \|_{H^2\II}.
\end{equation*}
 By mean of Laplace Transform,  the solution  of   $u(t)$ to problem \eqref{eqn:pde} can be represented by 
\begin{equation}\label{eq:solutionrepre}
\begin{aligned}
 u( t)&= F_q(t)\bar v +   \int_0^tE_q(s)(f- qD)\d s 
= F_q(t)(v-D)+(I-F_q(t)) A_q^{-1}(f- qD),
\end{aligned}
\end{equation}
where   the operators  $F_q(t)$ and $E_q(t)$ are,  respectively, defined by 
\begin{equation}\label{eqn:FE-MLcop}
 \begin{aligned}
 F_q(t)= \frac{1}{2\pi\mathrm{i}} \int_{\contour} e^{zt} z^{\alpha-1} (z^\alpha
 + A_q)^{-1}  \,\d z,\quad 
E_q (t)= \frac{1}{2\pi\mathrm{i}} \int_{\contour} e^{zt} (z^\alpha +A_q)^{-1}\,\d z.
  \end{aligned}
\end{equation}
 Here $\contour =\{z\in \mathbb{C}: |z| = \delta , |\arg z|\le \theta\} \cup \{ z\in \mathbb{C}: z =\rho e^{\pm i\theta},\rho\ge \sigma\}$ denotes the integral contour in the complex plane $\mathbb{C}$ 
with $\sigma\ge 0$ and $\frac{\pi}{2}<\theta< \frac{\pi}{\alpha}$,
oriented counterclockwise.
Therefore   the solution $U(t)$ to problem~\eqref{eqn:pdeo} could be represented by
\begin{align}\label{sol:origsol}
     U(t)=   F_q(t)(v-D)+(I-F_q(t)) A_q^{-1}(f- qD)+D.
\end{align}
In the following context,  we denote the solutions of $u$ and $U$ for problems~\eqref{eqn:pde} and \eqref{eqn:pdeo} corresponding to the potential and initial condition by $u(t;q, \bar v)$ and $U(t;q, v)$. There holds the relation $\bar v=v-D$.


The following lemma describes the smoothing properties and asymptotic behavior of the solution operators $F_q(t)$ and $E_q(t)$ defined in \eqref{eqn:FE-MLcop}. 
Here, $\|\cdot\|$ denotes the operator norm on $L^2(\Omega)$. 
The proof of part~(i) can be found in \cite[Theorems~6.4 and~3.2]{Jin:book2021} and \cite[Theorem~4.1]{sakamoto2011initial}. 
The proof of part~(ii) will be provided below.

\begin{lemma}\label{lem:op}
Let $F_q(t)$ and $E_q(t)$ be the solution operators defined in \eqref{eqn:FE-MLcop} for $q\in\A$.
Then  there exists a constant $c_0>0$ independent of $q$ and $t$ such that  for all $t>0$, $0\le \nu \le 1$:
\begin{itemize}
\item[$\rm(i)$] $t^{\nu\alpha}\|A_q^{\nu}F_q(t)\|+t^{1-(1-\nu)\alpha}\|A_q^{\nu}E_q(t)\|\le c_0,$  and  $ \|(A_q  F_q(t))^{-1}\| \le c_0 (1+t^{\alpha}) $;
\item[$\rm(ii)$] $ \|F_q(s) F_q(t)^{-1}\|\le c_0(1+s^{-\al}t^\al),\quad \|F'_q(s) F_q(t)^{-1}\|\le c_0s^{-1}(1+s^{-\al}t^\al) $ for $0<s\le t$.    
\end{itemize}
\end{lemma}
\begin{proof}
    For any $v\in L^2(\Omega)$, we have the following equivalence formulas of  $F_q(t)$ and $E_q(t)$
\begin{equation*}
\begin{aligned}
F_q(t) v = \sum_{n=1}^\infty E_{\alpha,1}(-\lambda_{n}(q)t^\alpha)(v,\fy_{n}(q))\fy_{n}(q), \quad
E_q(t)v = \sum_{n=1}^\infty t^{\alpha-1}E_{\alpha,\alpha}(-\lambda_{n}(q) t^\alpha)(v,\fy_{n}(q))\fy_{n}(q),
\end{aligned}
\end{equation*}
 where $E_{\alpha,\beta}(z)$ denotes the two-parameter Mittag--Leffler function.  It is well-known that, with $\alpha\in(0,1)$, there hold  \cite[Lemma 1.3]{JinZhou:2023book}  and \cite[Theorem 3.3 and Corollary 3.3]{Jin:book2021} for all $t\ge0$
\begin{equation*}
   0\le E_{\alpha,\alpha}(-t) \le \frac{c}{1+t^2}\quad \text{and}\quad\frac{1}{1+\Gamma(1-\alpha)t}\le E_{\alpha,1}(-t)\le \frac{1}{1+\Gamma(1+\alpha)^{-1}t}.
\end{equation*}
Therefore, we obtain 
\begin{align*}
    \|F_q(s) F_q(t)^{-1}v\|_{L^2\II}^2&\le c\sum_{n=1}^{\infty}\left|\frac{1+\lambda_{n}(q)t^\alpha}{1+\lambda_{n}(q)s^\alpha}\right|^2(v,\fy_{n}(q))^2
    \le c(1+s^{-\al}t^\al)^2\sum_{n=1}^{\infty}(v,\fy_{n}(q))^2.
\end{align*}
Additionally, using the identity $A_qE_q(t)=-F'_q(t)$ \cite[Lemma 6.3]{Jin:book2021} gives
\begin{align*}
    \|F'_q(s) F_q(t)^{-1}v\|_{L^2\II}^2
    \le &c\sum_{n=1}^{\infty}\left|\frac{\lambda_{n}(q)s^{\al-1}}{\sqrt{1+(\lambda_{n}(q)s^\alpha)^2}}\frac{1+\lambda_{n}(q)t^\alpha}{\sqrt{(1+(\lambda_{n}(q)s^\alpha)^2}}\right|^2(v,\fy_{n}(q))^2
    \\
    \le & 
    c(s^{-1}(1+s^{-\al}t^{\al}))^2\sum_{n=1}^{\infty}(v,\fy_{n}(q))^2.
\end{align*}
These complete the proof of the desired estimate  (ii).
\end{proof}
We now state several regularity results for $u(t)$ to the time-fractional diffusion equation~\eqref{eqn:pde}.
\begin{lemma}\label{lem:Dalu}
Let ${\bar v}, f, D(x)\in L^2\II$ and $q\in \A$, and $u(t)$ be the solution to problem \eqref{eqn:pde}. Then there holds
\begin{align*}
t^{1+\al}\|u'(t)\|_{H^2\II}+t\|   u'(t) \|_{L^2\II} +\|   u(t) \|_{L^2\II}\le c_{{\bar v},f,D} ,\quad
 \| \Dal u  (t)\|_{L^2\II} + \|  u(t) \|_{\dot H^{2}\II} \le c_{{\bar v},f,D} (1 + t^{-\alpha}),
\end{align*}
where the constant $c_{{\bar v},f,D}>0$ is  independent of $q$ and $t$,  but may depend on   ${\bar v}$, $f$ and  $D$.
\end{lemma}
\begin{proof}
    Employing the solution representation \eqref{eq:solutionrepre},  Lemma~\ref{lem:op} gives $ \|u(t)\|_{L^2\II}\le c_{{\bar v},f,D}$. 
     Moreover, the identity $A_qE_q(t)=-F'_q(t)$ gives 
    \begin{align*}
        \|u'(t)\|_{L^2\II}=&\|A_qE_q(t)({\bar v}-A_q^{-1}(f-qD))\|_{L^2\II}\le c_0t^{-1}\|{\bar v}-A_q^{-1}(f-qD)\|_{L^2\II}\le c_{{\bar v},f,D}t^{-1}.
    \end{align*}
    The rest of the proof to the estimate can be found in   \cite[Lemma 4.3]{Kian2023}.
\end{proof}
Using Sobolev embedding theorem, we can derive the following lemma.
\begin{lemma}\label{Lem:sobemb}
    For any $u, v\in L^2\II$, $q\in \A$, there exists a constant $c_1$ independent of $q$ such that 
    \begin{align*}
        \|A_q^{-1} u\|_{L^\infty\II}\le c_1\|u\|_{L^2\II},\quad \|A_q^{-1} (uv)\|_{L^2\II}\le c_1\|u\|_{L^2\II}\|v\|_{L^2\II}.
    \end{align*}
\end{lemma}
\begin{proof}
    The first estimate is obtained by Sobolev embedding theorem and elliptic regularity pickup such that
   \begin{align*}
    \|A_q^{-1} u\|_{L^\infty\II} \le C \|A_q^{-1} u\|_{H^2\II}  \le c_1 \|u\|_{L^2\II}.
    \end{align*}    
    This immediately leads to the second estimate: 
    \begin{align*}
      \|A_q^{-1} (uv)\|_{L^2\II} &=\sup_{\fy\in L^2\II}\frac{( uv,A_q^{-1}\fy)}{\|\fy\|_{L^2\II}}\le  \sup_{\fy\in L^2\II}\frac{ \|uv\|_{L^1\II}\|A_q^{-1}\fy\|_{L^\infty\II}}{\|\fy\|_{L^2\II}} 
      \\&\le c_1\|uv\|_{L^1\II}\le c_1\|u\|_{L^2\II}\|v\|_{L^2\II}.
    \end{align*}
    This completes the proof of the lemma.
\end{proof}

\subsection{Conditional stability of the inverse problem.}\label{sub2.2} 
In this part, we will show the well-posedness of the inverse problem:
look for the initial condition $v^\dag\in  L^2\II$ and 
 potential $q^\dag\in \A$, such that $U(x,t)$ satisfying
 \begin{equation}\label{eq:back_non}
 \begin{cases}
  \begin{aligned}
     \Dal U(x,t)  - \Delta U(x,t) +q^\dag(x) U(x,t)&=f(x) &&(x,t)\in \Omega\times(0,T],\\
    U(x,t)&=b(x),&&(x,t)\in \partial\Omega\times(0,T],\\
    U(x,T_1)&=g_1(x),&&x\in\Omega,\\
    U(x,T_2)&=g_2(x),&&x\in\Omega.
  \end{aligned}
  \end{cases}
 \end{equation}
Using the solution representation \eqref{sol:origsol} gives 
\begin{equation*}
 U(x,T_1) =   F_{q^\dag}(T_1)(v^\dag-D)+(I-F_{q^\dag}(T_1)) A_{q^\dag}^{-1}(f- {q^\dag}D)+D,
\end{equation*}
where $D(x)$ is the solution to \eqref{eqn:pdedx}. Therefore, we can derive that
\begin{align}\label{eq:backcontsolrep}
     v^\dag= F_{q^\dag}(T_1)^{-1}(g_1-D-(I-F_{q^\dag}(T_1)) A_{q^\dag}^{-1}(f- {q^\dag}D))+D.
\end{align}
In addition, the first  and the fourth relations in  \eqref{eq:back_non} lead to
\begin{equation}\label{eq:backcontsolrepq}
    q^\dag=\frac{f- \Dal U (T_2;q^\dag,v^\dag)  + \Delta g_2}{g_2}.
\end{equation}
We shall investigate the existence and uniqueness of $v^\dag$ and $q^\dag$ satisfying \eqref{eq:backcontsolrep} and \eqref{eq:backcontsolrepq}, which pertains to  the well-posedness of the inverse problem for identifying the initial condition and the potential 
 \eqref{eq:back_non}. 
Note that the relation \eqref{eq:backcontsolrep}--\eqref{eq:backcontsolrepq} naturally provides a fixed point iteration where the potential  $q^\dag$ is the fixed point.   

The operator 
 $K: \A \rightarrow \A$ defined by 
 \begin{equation}\label{eqn:K}
 K q :=  P_{[0,m]}\left(\frac{f- \Dal U(T_2;q,B(q,g_1)) + \Delta g_2  }{g_2 } \right),
\end{equation}
where the function $P_{[0,m]}:\mathbb{R} \rightarrow \mathbb{R}$ denotes a truncation function defined by
\begin{align}\label{eqn:P0M1}
P_{[0,m]} (a) :=   \max(\min( m, a ),0),
\end{align}
and  $ B(q,g)$ is given by 
\begin{align*}
    B(q,g)= F_{q}(T_1)^{-1}(g-D-(I-F_{q}(T_1)) A_{q}^{-1}(f- {q}D))+D.
\end{align*}
The existence and uniqueness of $q^\dag$ follow from the operator $K$ having a unique fixed point; $v^\dag$ is then uniquely solved from \eqref{eq:backcontsolrep}.


We now present the following two lemmas, which provide crucial \textsl{a priori} estimates that are extensively used in the proof of the contraction mapping.



\begin{lemma}\label{lem:stab-0}
Let $u(t;q_i,\bar v_i),\  i=1,2$  be the solution of \eqref{eqn:pde} with potential $q_i\in \A$ and  initial condition $\bar v_i\in L^2\II$.  Then it holds that for any $t > 0$
\begin{equation*}
\|\Dal u(t;q_1,\bar v_1)-\Dal u (t;q_2,\bar v_2)\|_{L^2\II} \le \tilde c_{0,\bar v_2,f,D}t^{-\alpha}\left(\|\bar v_1-\bar v_2\|_{L^2\II}+ \|q_1-q_2\|_{L^2\II}\right),
\end{equation*}
where the  constant $\tilde c_{0,\bar v_2,f,D}>0$  is independent of $q_1$, $ q_2$ and $t$, but may depend on $\bar v_2$, $f$ and  $D$.
\end{lemma}
\begin{proof}
   Let   $\omega(t)=u (t;q_1,\bar v_1)-u (t;q_2,\bar v_2)$.   Then $\omega$ solves
$$ \partial^\alpha_t \omega+A_{q_1} \omega=\bar \omega(t)  \quad \text{with}~ \omega(0)=\bar v_1-\bar v_2,$$
where $\bar\omega(t)=(q_2-q_1)(u(t;q_2,\bar v_2)+D)$.
By means of Laplace transform, 
the   solution $\omega(t)$ can be written as 
\begin{equation}\label{eq:repomega}
    \omega(t)=F_{q_1}(t)\omega(0)+\int_0^t E_{q_1}(t-s)\bar\omega(s)\ \ds.
\end{equation}
The governing equation for $\omega(t)$ and the identity $A_qE_q(t) = - F_q'(t)$  lead to
\begin{align*}
   \partial^\alpha_t  \omega(t)&=-A_{q_1} [F_{q_1}(t) \omega(0)+\int_0^t E_{q_1}(t-s)\bar\omega(s)\ \ds] +\bar\omega(t) \\
   &=-A_{q_1} F_{q_1}(t) \omega(0)+\int_0^t F'_{q_1}(t-s)\bar\omega(s)\ \ds +\bar\omega(t)
   \\
   &=A_{q_1} F_{q_1}(t) \left(-\omega(0)+ (A_{q_1} F_{q_1}(t))^{-1}\partial_t\int_0^t F_{q_1}(t-s)\bar\omega(s)\ \ds\right) .
\end{align*}
Let 
$\phi(t)=\int_0^t F_{q_1}(t-s)\bar\omega(s)\ \ds$.    
From Lemma \ref{lem:op}, there holds
\begin{align}\label{eq:acb0}
    \|\partial^\alpha_t  \omega(t)\|_{L^2\II}\le c_0t^{-\al} (\|\omega(0)\|_{L^2\II}+\| (A_{q_1} F_{q_1}(t))^{-1}\phi'(t)\|_{L^2\II}).
\end{align}
We now focus on the bound of $\| (A_{q_1} F_{q_1}(t))^{-1}\phi'(t)\|_{L^2\II}.$
Lemmas \ref{lem:op}--\ref{Lem:sobemb} imply
   \begin{align*}
&\quad\|(A_{q_1} F_{q_1}(t))^{-1}\phi(t) \|_{L^2\II}\\
\le &\int_0^t\left\| F_{q_1}(t)^{-1}F_{q_1}(t-s)A_{q_1}^{-1}\bar\omega(s)\right\|_{L^2\II}\ \ds
\le \int_0^t\left\| F_{q_1}(t)^{-1}F_{q_1}(t-s)\right\|\left\|A_{q_1}^{-1}\bar\omega(s)\right\|_{L^2\II}\ \ds\\
\le &c\|  q_1-q_2 \|_{L^2\II}\int_0^t(1+(t-s)^{-\al}t^\al) \|u(s;q_2,\bar v_2)+D\|_{L^2\II}\ds
\le c_{0,\bar v_2,f,D}t\|  q_1-q_2 \|_{L^2\II}.
\end{align*}     
Next, the identity
$t \phi(t) = \int_0^t (t-s) F_{q_1}(t-s) \bar\omega(s) \,\d s  + \int_0^t F_{q_1}(s)(t-s) \bar\omega(t-s) \,\d s$ gives
   \begin{align*}
 \partial_t(t    \phi(t))&= \int_0^t  \big[(t-s) F_{q_1}'(t-s)  +  F_{q_1}(t-s) \big] \bar\omega(s) \,\d s 
 + \int_0^t  F_{q_1}(s)   \big[\bar\omega(t-s) + (t-s) \bar\omega'(t-s) \big] \,\d s\\  &= \int_0^t  \big[s F_{q_1}'(s) +  F_{q_1}(s)\big] \bar\omega(t-s) \,\d s 
 + \int_0^t  F_{q_1}(t-s)   \big[\bar\omega(s)+ s\bar\omega'(s)\big] \,\d s =: {\rm I}_{c,1} + {\rm I}_{c,2}.
\end{align*}
Now we derive bounds for $(A_{q_1} F_{q_1}(t))^{-1}{\rm I}_{c,k}$, $k=1,2$. First, we can bound  the term $(A_{q_1} F_{q_1}(t))^{-1}{\rm I}_{c,1}$   by
Lemmas \ref{lem:op}--\ref{Lem:sobemb}:
\begin{align*}
  &\|\lb A_{q_1} F_{q_1}(t)\rb ^{-1} {\rm I}_{c,1} \|_{L^2\II}\\
  \le&  \int_0^t\left \| \big[sF_{q_1}(t)^{-1} F_{q_1}'(s) + F_{q_1}(t)^{-1} F_{q_1}(s)\big] A_{q_1} ^{-1}\bar\omega(t-s)\right\|_{L^2\II} \,\d s \\
    \le&  \int_0^t\left(\left \| sF_{q_1}(t)^{-1} F_{q_1}'(s)\right\| +\left\| F_{q_1}(t)^{-1} F_{q_1}(s)\right\|\right)\left\| A_{q_1} ^{-1}\bar\omega(t-s)\right\|_{L^2\II} \,\d s \\
  \le& c \| q_2-q_1\|_{L^2\II} \int_0^t (1+s^{-\al}t^\al)
  \| u(t-s;q_2,\bar v_2)+D\|_{L^2\II} \,\d s
 \le   c_{0,\bar v_2,f,D} t\|  q_1-q_2 \|_{L^2\II}.
\end{align*}
Similarly, using Lemmas \ref{lem:op}--\ref{Lem:sobemb}, the term $(A_{q_1} F_{q_1}(t))^{-1}{\rm I}_{c,2}$ can be bounded by
\begin{align*}
 &\|(A_{q_1} F_{q_1}(t))^{-1}{\rm I}_{c,2} \|_{L^2\II} \\\le &
  \int_0^t\left \|F_{q_1}(t)^{-1} F_{q_1}(t-s) A_{q_1}^{-1} \big[\bar\omega(s)+ s\bar\omega'(s)\big]\right\|_{L^2\II}  \,\d s \\
  &\le  
  \int_0^t\left \|F_{q_1}(t)^{-1} F_{q_1}(t-s)\right\|  \left  \|A_{q_1}^{-1} \big[\bar\omega(s)+ s\bar\omega'(s)\big]\right\|_{L^2\II}  \,\d s \\
 \le& c\|  q_1-q_2 \|_{L^2\II} \int_0^t (1+(t-s)^{-\al}t^\al)
  \left(\| u(s;q_2,\bar v_2)+D\|_{L^2\II}+\|su'(s;q_2,\bar v_2)\|_{L^2\II}\right)\,\d s \\
  \le&  c_{0,\bar v_2,f,D} t\|  q_1-q_2 \|_{L^2\II}.
\end{align*}
Then the triangle inequality yields that for any $t>0$, there holds
\begin{align*}\label{eq:estiphi}
 \| (A_{q_1} F_{q_1}(t))^{-1} \phi'(t) \|_{L^2(\Omega)} &\le t^{-1} \left(\| (A_{q_1} F_{q_1}(t))^{-1}(t\phi(t))' \|_{L^2(\Omega)} + \|(A_{q_1} F_{q_1}(t))^{-1}\phi(t)  \|_{L^2(\Omega)}\right) \\&\le  c_{0,\bar v_2,f,D}\| q_2-q_1\|_{L^2\II}.
\end{align*}
Now combining with \eqref{eq:acb0}, the desired estimate follows directly.  
\end{proof}

\begin{lemma}\label{lem:stab-00}
Let  $u (t;q_i,\bar v_i),\  i=1,2$  be the solution of \eqref{eqn:pde} with potential $q_i\in \A$ and  initial condition   $\bar v_i\in L^2\II$.   Then there holds that for any $t > 0$,
\begin{equation*}
\|\bar v_1-\bar v_2\|_{L^2\II} \le \tilde  c_{1,\bar v_2,f,D}(1+t^\al)\left( \|u (t;q_1,\bar v_1)-u (t;q_2,\bar v_2)\|_{\dot H^2\II}+ \|q_1-q_2\|_{L^2\II}\right),
\end{equation*}
where the  constant $\tilde c_{1,\bar v_2,f,D}>0$  independent of $q_1$, $ q_2$ and $t$, but may depend on $\bar v_2$, $f$ and  $D$.
\end{lemma}
\begin{proof}
    Employing the solution representation  \eqref{eq:solutionrepre} yields
    \begin{align*}
        \bar v_i= F_{q_i}(t)^{-1}\left(u (t;q_i,\bar v_i)- \int_0^tE_{q_i}(s)(f- {q_i}D)\d s \right).
    \end{align*}
   Therefore, we can find that 
    \begin{align*}
        \bar v_1-\bar v_2=&(F_{q_1}(t)^{-1}-F_{q_2}(t)^{-1})\left(u (t;q_2,\bar v_2)- \int_0^tE_{q_2}(s)(f- {q_2}D)\d s \right)\\&-F_{q_1}(t)^{-1} \left(\int_0^tE_{q_1}(s)(f- {q_1}D)\d s-\int_0^tE_{q_2}(s)(f- {q_2}D)\d s\right)\\
        &+F_{q_1}(t)^{-1}\left(u (t;q_1,\bar v_1)-u(t;q_2,\bar v_2) \right)=:{\rm II_1}(t) +{\rm II_2}(t)+{\rm II_3}(t).
    \end{align*}
     We now derive the bounds for ${\rm II_1}(t), {\rm II_2}(t)$, and ${\rm II_3}(t)$, separately. Firstly,  it holds that  
    \begin{align*}
       {\rm II_1}(t)
       &=(F_{q_2}(t)-F_{q_1}(t))F_{q_1}(t)^{-1}F_{q_2}(t)^{-1}\left(u(t;q_2,\bar v_2)- \int_0^tE_{q_2}(s)(f- {q_2}D)\d s \right)\\
       &= (F_{q_2}(t)-F_{q_1}(t))F_{q_1}(t)^{-1}\bar v_2.
    \end{align*}
    Let  ${\rm II}_1^i(s)=F_{q_i}(s)F_{q_1}(t)^{-1}\bar v_2$, $i=1,2$. 
Then ${\rm II_1}(t)={\rm II_1^2}(t)-{\rm II_1^1}(t)$ solves
    \begin{align*}
        \Dal {\rm II_1}+A_{q_2}{\rm II_1}= (q_1-q_2){\rm II_1^1}\quad \text{with }\quad  {\rm II_1}(0)=0.
    \end{align*}
Using solution representation \eqref{eq:repomega} and the  identity $A_qE_q(t)=-F_q'(t)$ gives
   \begin{align*}
    {\rm II_1}(t)
    =&\int_0^{\frac{t}{2}}E_{q_2}(t-s)(q_1-q_2)F_{q_1}(s)F_{q_1}(t)^{-1}\bar v_2\ \ds\\&
  - \int^t_{\frac{t}{2}}F'_{q_2}(t-s)A_{q_2}^{-1}(q_1-q_2)F_{q_1}(s)F_{q_1}(t)^{-1}\bar v_2\ \ds=:{\rm II_{1,1}}(t)+{\rm II_{1,2}}(t).
\end{align*} 
Applying  Lemma~\ref{lem:op} and Lemma~\ref{Lem:sobemb} to the term ${\rm II_{1,1}}(t)$ leads to
    \begin{align*}
    \lnorm{\rm II_{1,1}}(t)\rnorm _{L^2\II}&\le \int_0^{\frac{t}{2}}\lnorm A_{q_2}E_{q_2}(t-s)\rnorm\lnorm A_{q_2}^{-1}(q_1-q_2)F_{q_1}(s)F_{q_1}(t)^{-1}\bar v_2\rnorm_{L^2\II}\ds\\
    &\le c\|q_1-q_2\|_{L^2\II}\int_0^{\frac{t}{2}}(t-s)^{-1}\lnorm F_{q_1}(s)F_{q_1}(t)^{-1}\bar v_2\rnorm _{L^2\II}\ \ds\\
       &\le c\|q_1-q_2\|_{L^2\II}\int_0^{\frac{t}{2}}(t-s)^{-1}(1+s^{-\al}t^\al)\ \ds\|\bar v_2\|_{L^2\II}
        \le  c_{1,\bar v_2}\|q_1-q_2\|_{L^2\II}.
\end{align*}
Employing  integration by parts to ${\rm II_{1,2}}(t)$ yields
\begin{align*}
    {\rm II_{1,2}}(t)
    =&-A_{q_2}^{-1}(q_1-q_2)\bar v_2+F_{q_2}(\frac{t}{2})A_{q_2}^{-1}(q_1-q_2)F_{q_1}(\frac{t}{2})F_{q_1}(t)^{-1}\bar v_2\\&+\int^t_{\frac{t}{2}}F_{q_2}(t-s)A_{q_2}^{-1}(q_1-q_2)F'_{q_1}(s)F_{q_1}(t)^{-1}\bar v_2\    \ds.
\end{align*}
Therefore,  we can bound $ {\rm II_{1,2}}(t)$
by Lemma~\ref{lem:op} and Lemma~\ref{Lem:sobemb}: 
\begin{align*}
    \| {\rm II_{1,2}}(t)\|_{L^2\II}\le c\|q_1-q_2\|_{L^2\II}\left(1+\int^t_{\frac{t}{2}}(s^{-1}(1+s^{-\al}t^{\al})\ds\right) \|\bar v_2\|_{L^2\II}\le  c_{1,\bar v_2} \|q_1-q_2\|_{L^2\II}.
\end{align*}
We now analyze the second term ${\rm II_2 }(t)$.
Let ${\rm II}_2^i (t)=\int_0^tE_{q_i}(s)(f- {q_i}D)\d s$, $i=1,2$.
    We can find that ${\rm II_2^0}(t)= {\rm II_2^1}(t)- {\rm II_2^2}(t) $ solves 
   \begin{align*}
        \Dal {\rm II_2^0}+A_{q_1}{\rm II_2^0}= (q_2-q_1)({\rm II_2^2}(t)+D )\quad \text{with }\quad {\rm II_2^0}(0)=0.
    \end{align*}
  Employing solution representation \eqref{eq:repomega}, the identities  $A_qE_q(t) = - F_q'(t)$ and $ {\rm II_2}(t)=-F_{q_1}(t)^{-1}{\rm II_2^0}(t)$ gives
    \begin{align*}
       {\rm II_2}(t)=& -F_{q_1}(t)^{-1} \int_0^tE_{q_1}(t-s)(q_2-q_1)({\rm II_2^2}(s)+D )   \ \ds\\
       =&(A_{q_1}F_{q_1}(t))^{-1}\int_0^t F'_{q_1}(t-s)(q_2-q_1)({\rm II_2^2}(s)+D )\ \ds\\
      =& (A_{q_1}F_{q_1}(t))^{-1}\bigg(\partial_t\int_0^t F_{q_1}(t-s)(q_2-q_1)({\rm II_2^2}(s)+D )\ \ds-(q_2-q_1)({\rm II_2^2}(t)+D )\bigg).
    \end{align*}
    Now Lemmas~\ref{lem:op}--\ref{lem:Dalu}  imply
   \begin{equation}\label{Dlinfy}
     \begin{aligned}
        \lnorm(A_{q_1}F_{q_1}(t))^{-1}(q_2-q_1)({\rm II_2^2}(t)+D )\rnorm_{L^2\II}&\le c(1+t^\al)\|q_1-q_2\|_{L^2\II}\|{\rm II_2^2}(t)+D\|_{L^\infty \II}\\
        &\le c_{1,f,D}(1+t^\al)\|q_1-q_2\|_{L^2\II}.
    \end{aligned}  
   \end{equation} 
Additionally, from the proof of Lemma~\ref{lem:stab-0}, we can find that
\begin{align*}
      \lnorm(A_{q_1}F_{q_1}(t))^{-1}\partial_t\int_0^t F_{q_1}(t-s)(q_2-q_1)({\rm II_3^2}(s)+D )\ \ds \rnorm_{L^2\II}\le  c_{1,f,D}\|q_2-q_1\|_{L^2\II}.
\end{align*}
Last, applying Lemma~\ref{lem:op} to  the term  ${\rm II_3}(t)$ leads to
\begin{align*}
     \| {\rm II_3}(t)\|_{L^2\II}\le c(1+t^\al)\|u (t;q_1,\bar v_1)-u (t;q_2,\bar v_2)\|_{\dot H^2\II}.
\end{align*}
 The above analysis completes the proof of the lemma.
\end{proof}
Then, with the help of Lemmas 
\ref{lem:stab-0}--\ref{lem:stab-00}, we are ready to show that $K$  defined in \eqref{eqn:K} is a contraction mapping, and hence possesses a unique fixed point.

\begin{theorem}\label{lem:contK} 
Let  $K$ be the operator  defined in \eqref{eqn:K}.
For fixed $T_1$,  let $T_0$ be the solution of  
\begin{align}\label{T_0tild}
    M^{-1} \tilde  c_{0}T_0^{-\al}(\tilde  c_{1}(1+T_1^\al)+1)=\tilde c,
\end{align}
where $\tilde c $ is some constant in $(0,1)$,  $\tilde  c_{0}= \tilde c_{0,v^\dag-D,f,D}$ and $\tilde  c_{1}=\tilde c_{1,v^\dag-D,f,D}$ are the constants given by Lemmas~\ref{lem:stab-0} and \ref{lem:stab-00}.
Then there holds for $T_2>T_0$, $q_0\in \A$, the iteration 
$$q_{n+1}=Kq_n,\quad \forall\ n=0,1,\cdots,$$
linearly converges to the fixed point $q^\dag$, s.t.
$$\|q_{n+1}-q^\dag\|_{L^2\II}\le\tilde c \|q_{n}-q^\dag\|_{L^2\II}.$$
Consequently, the inverse potential and backward problem \eqref{eq:back_non} admits a unique solution $q^\dag$, and then $v^\dag$ can be uniquely solved from~\eqref{eq:backcontsolrep}.
\end{theorem}
\begin{proof} By the definition \eqref{eqn:K},  the property that $|P_{[0,m]} (a)  - P_{[0,m]} (b) | \le |a-b|$ and the fact $\Dal U (t;q,v)=\Dal u (t;q,v-D)$, there holds
  \begin{align*}
| q_{n+1}-q^\dag |=| Kq_{n}-Kq^\dag |\le \left|\frac{\Dal u (T_2;q_{n},B(q_{n},g_1(x))-D) - \Dal u (T_2;q^\dag,B(q^\dag,g_1(x))-D) }{g_2(x) } \right|.
\end{align*}
Therefore, for $T_2>T_0$ Lemmas~\ref{lem:stab-0}--\ref{lem:stab-00} give
\begin{align*}
    \|q_{n+1}-q^\dag \|_{L^2\II}&\le \left\|\frac{\Dal u (T_2;q_n,B(q_n,g_1(x))-D) - \Dal u (T_2;q^\dag,B(q^\dag,g_1(x))-D) }{g_2(x)} \right\|_{L^2\II}\\&\le M^{-1}\tilde c_0 T_2^{-\al}\left(\|B(q_n,g_1(x))-B(q^\dag,g_1(x))\|_{L^2\II}+\|q_1-q_2\|_{L^2\II}\right)\\
    &\le M^{-1}\tilde c_0 T_2^{-\al}(\tilde c_1 (1+T_1^\al) +1)\|q_n-q^\dag\|_{L^2\II}\\
    &<\tilde c\|q_n-q^\dag\|_{L^2\II},
\end{align*}  
where we use the maximum principle Lemma~\ref{positive} to get  $g_2(x)\ge M$.
The operator $K$  admits a unique fixed point. As a result, the inverse potential and backward problem \eqref{eq:back_non} admits a unique solution $q^\dag\in\A$. 
\end{proof}

The next stability estimate is the main result of this section.
\begin{theorem}\label{thm:cond-stab}
 Let   $U(t;q_i, v_i),\  i=1,2$  be the solution of \eqref{eqn:pdeo} with potential $q_i\in \A$ and  initial condition  $ v_i\in L^2\II$   respectively. 
Suppose that $\bar{v}_i = v_i - D$, $f \in L^2(\Omega)$, and $D \in L^\infty(\Omega)$ with $D(x) \geq M > 0$ almost everywhere in $\Omega$, where $\bar{v}_i \geq 0$ and $f - q_i D \geq 0$ in $\Omega$. For any  fixed $ T_1> 0$,  
then  for any  $T_2> T_0$, there holds that 
 \begin{align*}
   \| q_1 - q_2 \|_{L^2\II} +\|v_1-v_2\|_{L^2\II} \le  C 
   \sum_{i=1}^2 \|  U(T_i;q_1, v_1) - U(T_i;q_2, v_2) \|_{\dot H^2\II}.
 \end{align*}
Here $T_0$  is the solution of \eqref{T_0tild} with $\tilde  c_{0}= \tilde c_{0,\bar v_2,f,D}$ and $\tilde  c_{1}=\tilde c_{1,\bar v_2,f,D}$ given by Lemmas~\ref{lem:stab-0}--\ref{lem:stab-00}, while the constant $C$ in the estimate may depend on $T_1$ 
but is independent of $q_1$, $q_2$, and $T_2$. 
\end{theorem}
\begin{proof}
It follows from \eqref{eqn:pde}, $q_i$ could be expressed as
$q_i =(u(T_2;q_i,\bar v_i)+D)^{-1}(f-\Dal u (T_2;q_i,\bar v_i) + \Delta u(T_2;q_i,\bar v_i)).$
Then we split $q_1 - q_2$ into two parts:
\begin{align*}
 q_1 - q_2 = &\bigg\{f\frac{u(T_2;q_2,\bar v_2)-u(T_2;q_1,\bar v_1)}{(u(T_2;q_1,\bar v_1)+D)(u(T_2;q_2,\bar v_2)+D)} +\frac{  [u(T_2;q_1,\bar v_1)- u(T_2;q_2,\bar v_2)]\Dal u(T_2;q_1,\bar v_1)}{(u(T_2;q_1,\bar v_1)+D)(u(T_2;q_2,\bar v_2)+D)}  \\&\quad+ \frac{\Delta u(T_2;q_1,\bar v_1)-\Delta u(T_2;q_2,\bar v_2)}{u(T_2;q_1,\bar v_1)+D}+\frac{(u(T_2;q_2,\bar v_2)-u(T_2;q_1,\bar v_1))\Delta u(T_2;q_2,\bar v_2)}{(u(T_2;q_1,\bar v_1)+D)(u(T_2;q_2,\bar v_2)+D)}\bigg\}\\&+\frac{ \Dal u(T_2;q_2,\bar v_2)-\Dal u(T_2;q_1,\bar v_1)}{u(T_2;q_2,\bar v_2)+D}
\\=:&{\rm I}_{s,1}+{\rm I}_{s,2} .
   \end{align*}
Using the maximum principle Lemma~\ref{positive}, there holds $u(T_2)+D(x)\ge M$. Employing the standard Sobolev embedding theorem and   Lemma~\ref{lem:Dalu} gives 
\begin{align*}
   \| {\rm I}_{s,1}\|_{L^2\II} \le& M^{-2} \| f \|_{L^2\II}\| u(T_2;q_2,\bar v_2)-u(T_2;q_1,\bar v_1) \|_{L^\infty\II}\\& +M^{-2}\|\Dal u(T_2;q_2,\bar v_1)\|_{L^2\II}\|u(T_2;q_2,\bar v_2) -  u(T_2;q_1,\bar v_1)\|_{L^\infty\II}\\&+ M^{-1} \| \Delta (u(T_2;q_2,\bar v_2) -  u(T_2;q_1,\bar v_1)) \| _{L^2\II} \\&+M^{-2} \| \Delta u(T_2;q_2,\bar v_2) \|_{L^2\II} \|  u(T_2;q_2,\bar v_2) - u(T_2;q_1,\bar v_1)\|_{L^\infty\II}\\
   \le& c(1+T_2^{-\al})\|  (u(T_2;q_1,\bar v_1)-u(T_2;q_2,\bar v_2)) \|_{\dot H^2\II}.
\end{align*}
Besides,  from Lemma~\ref{lem:Dalu} and  Lemma~\ref{lem:stab-0} we arrive at 
\begin{equation*}
\begin{aligned}
\| {\rm I}_{s,2}\|_{L^2\II} 
\le M^{-1} \| \Dal (u(T_2;q_2,\bar v_2) -  u(T_2;q_1,\bar v_1)) \| _{L^2\II} 
\le M^{-1}\tilde c_0  T_2^{-\al}\left(\|\bar v_1-\bar v_2\|_{L^2\II} +\|q_2-q_1\|_{L^2\II}  \right) .
   \end{aligned}
\end{equation*}
Therefore, it holds that
\begin{equation*}
\begin{aligned}
    \| q_1 - q_2  \|_{L^2\II}  &\le   M^{-1}\tilde c_0T_2^{-\al} \left(\|\bar v_1-\bar v_2\|_{L^2\II} + \|q_2-q_1\|_{L^2\II}\right)  \\
    &\qquad +c(1+T_2^{-\al}) \|  u(T_2;q_1,\bar v_1) - u(T_2;q_2,\bar v_2)\|_{\dot H^2\II} .
\end{aligned}
\end{equation*}
Additionally,  Lemma~\ref{lem:stab-00} gives 
\begin{equation}\label{T_1}
  \| v_1- v_2\|_{L^2\II}  \le  \tilde c_1 (1+T_1^\al)\left(\|u(T_1;q_1,\bar v_1)-u(T_1;q_2,\bar v_2)\|_{\dot H^2\II}+\| q_1 - q_2  \|_{L^2\II}\right) .
\end{equation}
As a result, we derive that 
\begin{align*}
      \| q_1 - q_2  \|_{L^2\II} \le& M^{-1}\tilde  c_0 T^{-\al}_2 (\tilde c_1(1+T_1^\al)+1)\|q_2-q_1\|_{L^2\II}\\&  + c(1+T_2^{-\al})\|  u(T_2;q_1,\bar v_1) - u(T_2;q_2,\bar v_2)\|_{\dot H^2\II}\\&+M^{-1}\tilde  c_0 T^{-\al}_2(1+T_1^\al)\|  u(T_1;q_1,\bar v_1) - u(T_1;q_2,\bar v_2)\|_{\dot H^2\II}.
\end{align*}
Hence for  $T_2> T_0$, we have $M^{-1}\tilde  c_0 T^{-\al}_2 (\tilde c_1(1+T_1^\al)+1)<1$, which implies
$$ \| q_1 - q_2  \|_{L^2\II} \le  c\left( \|  u(T_1;q_1,\bar v_1) - u(T_1;q_2,\bar v_2)  \|_{\dot H^2\II}+\|  u(T_2;q_1,\bar v_1) - u(T_2;q_2,\bar v_2)  \|_{\dot H^2\II} \right). $$
Now combining with \eqref{T_1}  and the fact $U(T_j;q_i, v_i)=u(T_j;q_i, \bar v_i)$ for $i,j=1,2$, the desired inequality follows directly. This completes the proof of the theorem.
\end{proof}
\section{Regularization  and convergence analysis for backward problem}\label{sec:regu} 


Theorem~\ref{thm:cond-stab} indicates that the inverse problem is mildly ill-posed, exhibiting a loss equivalent to a second-order derivative. However, 
the observational data typically contain noise and hence not belong to the ${H}^2(\Omega)$ space.  Therefore, a suitable regularization is necessary.

Here, to overcome the ill-posedness of the backward problem, we apply the following quasi-boundary value method for regularization \cite{hao2019stability}.
Let  $U_\gamma^\delta(t;q)=u_\gamma^\delta(t;q)+D$ be the  regularized solution with  $u_\gamma^\delta(t;q)$ satisfying
 \begin{equation}\label{eq:back_non_reg:noisy}
 \begin{cases}
  \begin{aligned}
     \Dal u_\gamma^\delta(t;q)+A_q u_\gamma^\delta(t;q)&=f(x)-q(x)D(x), &&(x,t)\in \Omega\times(0,T],\\
      u_\gamma^\delta(t;q)&=0,&&(x,t)\in \partial\Omega\times(0,T],\\
       \gamma u_\gamma^\delta(0;q)+ u_\gamma^\delta(T_1;q)&=g_1^\delta(x)-D:=\bar g_1^\delta,&&x\in\Omega.
  \end{aligned}
  \end{cases}
 \end{equation}
Here $\gamma$ denotes a positive regularization parameter. 

 Then we intend to examine the error 
$U_\gamma^\delta(0;q)-U(0,q)=u_\gamma^\delta(0;q) - u(0;q)$. 
To this end, we consider an auxiliary function $u_\gamma(t;q)$ that satisfies
\begin{equation}\label{eqn:back_non_reg:}
 \begin{cases}
  \begin{aligned}
   \Dal u_\gamma(t;q) +A_q u_\gamma(t;q)&=f(x)-q(x)D(x), &&(x,t)\in \Omega\times(0,T],\\
    u_\gamma(t;q)&=0,&&(x,t)\in \partial\Omega\times(0,T],\\
       \gamma  u_\gamma(0;q)+  u_\gamma(T_1;q)&=g_1(x)-D:=\bar g_1,&&x\in\Omega.
  \end{aligned}
  \end{cases}
 \end{equation} 
 Using solution representation \eqref{eq:solutionrepre}, and the conditions $ \gamma  u_\gamma^\delta(0;q)+  u_\gamma^\delta(T_1;q)=\bar  g_1^\delta(x)$, $\gamma  u_\gamma(0;q)+  u_\gamma(T_1;q)=\bar g_1(x)$,  we can  derive  that 
 \begin{align}
      u_\gamma^\delta(0;q)&=(\gamma I+F_{q}(T_1))^{-1}( \bar g_1^\delta-(I-F_{q}(T_1)) A_q^{-1}(f-qD))\label{solu:ugamdel},\\
           u_\gamma(0;q)&=(\gamma I+F_{q}(T_1))^{-1}( \bar g_1-(I-F_{q}(T_1)) A_q^{-1}(f-qD)).\label{solu:ugam}
 \end{align}
 

The following lemma describes the smoothing properties of the solution operator $(\gamma I + F_q(T))^{-1}$. 
Since the proof is analogous to that in \cite[Lemma~15.1]{JinZhou:2023book}, it is omitted here for brevity.
\begin{lemma}\label{lemma:gammainv}
 Let $F_q(t)$ be the discrete solution operator defined in \eqref{eqn:FE-MLcop}.
For $v\in \dot H^s(\Omega)$, there holds
    \begin{align*}
        \|(\gamma I +F_q(T))^{-1}v\|_{L^2\II}&\le c\gamma^{-\frac{2-s}{2}}\|v\|_{\dot H^s(\Omega)}~~\text{and}\\
         \|F_q(t)(\gamma I +F_q(T))^{-1}v\|_{L^2\II}&\le c\min(\gamma^{-1},t^{-\al})\|  v \|_{L^2\II},
    \end{align*}
where the constant $c$  may depend on $T$ but  is independent of $\gamma$, $t$.
\end{lemma}
We now give the bound of $ u_\gamma^\delta(0;q)- u_\gamma(0;q)$ and $ u_\gamma(0;q)- u(0;q)$.
 \begin{lemma}\label{lemma:gammacon} 
Let $ u(t;q) $ be the solution of \eqref{eqn:pde} with $ u(0,q) \in \dot {H}^s\II $ for $ s \in [0,2] $ and $  u(T_1;q) = g_1-D $. Consider $ u_\gamma^\delta(0;q)  $  and $ u_\gamma(0;q)  $   satisfying \eqref{solu:ugamdel} and \eqref{solu:ugam}, respectively. Then, the following estimate holds:
\begin{align*}
    \| u_\gamma^\delta(0;q)- u_\gamma(0;q)\|_{L^2\II}\le c\gamma^{-1}\delta \quad\text{and } \quad \| u_\gamma(0;q)- u(0;q)\|_{L^2\II}\le c\gamma^\frac{s}{2}.
\end{align*}
\end{lemma}
\begin{proof}
From \eqref{solu:ugamdel},  \eqref{solu:ugam} and Lemma~\ref{lemma:gammainv}, there holds
\begin{align*}
     \| u_\gamma^\delta(0;q)- u_\gamma(0;q)\|_{L^2\II} =\|(\gamma I+F_{q}(T_1))^{-1}(g_1^\delta-g_1)\|_{L^2\II}\le c\gamma^{-1}\delta.
\end{align*}
   Furthermore, the solution representation \eqref{eq:solutionrepre} gives 
   \begin{align*}
          u(0;q)=F_{q}(T_1)^{-1}( \bar g_1-(I-F_{q}(T_1)) A_q^{-1}(f-qD)).
   \end{align*}
   Consequently,  combining with Lemma~\ref{lemma:gammainv}, we can find that
   \begin{align*}
        \|  u_\gamma(0;q)- u(0;q)\|_{L^2\II}=&\|\gamma (\gamma I+F_q(T_1))^{-1}F_{q}(T_1)^{-1}(  \bar g_1-(I-F_{q}(T_1)) A_q^{-1}(f-qD))\|_{L^2\II}\\
       =&\|\gamma (\gamma I+F_q(T_1))^{-1}u(0;q)\|_{L^2\II}\le c\gamma^\frac{s}{2}\| u(0;q)\|_{ \dot H^s\II}.
   \end{align*}
This completes the proof of the lemma.
\end{proof}
At the end of this section, we present the following regularity of $ u_\gamma(q, 0)$ and $\Dal  u_\gamma(t;q)$, which is extensively used in the numerical analysis.
\begin{lemma}\label{lem:regugamma}  
Let  $ u_\gamma(t;q) $ satisfy \eqref{eqn:back_non_reg:}. 
Then for $s\in [0,2]$, there holds
\begin{equation*}
     \| u _\gamma(0;q)\|_{ \dot H^s\II}\le c\gamma^{-\frac {s}{2}}, \quad t^{\al}\| \Dal  u_\gamma(t;q)\|_{\dot H^2\II}+t^{\al+1}\| (\Dal  u_\gamma(t;q))'\|_{\dot H^2\II}\le c\gamma^{-1}.
 \end{equation*}
\end{lemma}
\begin{proof}
    Since the exact data $g_1(x)-D\in \dot H^{2}\II$, then from  \eqref{solu:ugam}
    and   Lemma~\ref{lemma:gammainv}, we can find that
    \begin{align*}
      \| u_\gamma(0;q) \|_{L^2\II}&\le c \|( \bar g_1-(I-F_{q}(T_1)) A_q^{-1}(f-qD)) \|_{\dot H^2\II}\le c,\\
          \| u_\gamma(0;q) \|_{\dot H^2\II}&\le c\gamma^{-1} \|( \bar g_1-(I-F_{q}(T_1)) A_q^{-1}(f-qD)) \|_{\dot H^2\II}\le c \gamma^{-1}.
    \end{align*}
     Then the first estimate results with $s\in(0,2)$ followed by the interpolation.

From the solution representation \eqref{eq:solutionrepre}, and the identities $\Dal F_q(t)=A_qF_q(t)$, $ F'_q(t)=-A_qE_q(t)$, we can derive that 
\begin{align*}
     \lnorm\Dal  u_\gamma(t;q)\rnorm_{\dot H^2\II}  = & \lnorm F_q(t)( A_q  u_\gamma(0;q)-A_q^{-1}(f- qD))\rnorm_{\dot H^2\II}\\ \le& c\lnorm A_q F_q(t)( A_q  u_\gamma(0;q)-A_q^{-1}(f- qD))\rnorm_{ L^2\II}\le c\gamma^{-1}t^{-\al},\\
      \lnorm(\Dal  u_\gamma(t;q))'\rnorm_{\dot H^2\II}  = & \lnorm F'_q(t)( A_q  u_\gamma(0;q)-A_q^{-1}(f- qD))\rnorm_{\dot H^2\II}\\ \le& c\lnorm A^2_q E_q(t)( A_q  u_\gamma(0;q)-A_q^{-1}(f- qD))\rnorm_{ L^2\II}\le c\gamma^{-1}t^{-\al-1}.
\end{align*}
This completes the proof of the lemma.
\end{proof}

 \section{Fully discretization and its error analysis}\label{sec:fully_scheme}
In this section, we propose and analyze a fully discrete scheme for reconstructing the initial data and potential. We begin by investigating a semidiscrete scheme based on the finite element method for spatial discretization. While the semidiscrete scheme may not be directly practical for implementation, it serves as a critical foundation for the analysis and development of the fully discrete scheme.


\subsection{Semidiscrete scheme to the problem} 
 
Let ${\{\mathcal{T}_h\}}_{0<h<1}$ denote a family of shape-regular and quasi-uniform triangulations of the domain $\Omega$ into $d$-simplices (finite elements), 
where $h$ represents the maximum element diameter. 
The corresponding finite element spaces $X_h$ and $X_h^0$ are defined as follows:

\begin{equation*}
\begin{aligned}
 X_h  = \{ v\in H^1(\Omega): v|_K \in P_1(K) \,\, \,\,\,\forall  K\in \T_h \},~~\text{and} ~~
 X_h^0  = X_h \cap H_0^1\II.
\end{aligned}
\end{equation*}
where $P_1(K)$ denotes the space of linear polynomials on $K$.
Define the orthogonal $L_2$-projection $P_h:L^2(\Omega)\to X_h^0$ and
the Ritz projection $R_h(q):H^1_0(\Omega)\to X_h^0$ by
  \begin{align*}
    (P_h \psi,\chi_h)&=(\psi,\chi_h),&& \quad\forall \chi_h\in X_h^0,\\
    (\nabla R_h(q) \psi,\nabla\chi_h)+(qR_h(q)\psi,\chi_h)&=(\nabla \psi,\nabla\chi_h) + (q\psi,\chi_h),&& \quad \forall \chi_h\in X_h^0.
  \end{align*}
It is well known that the operators $P_h$ and $R_h(q)$ (for $q \in \A$) satisfy the following approximation property, 
sf. \cite[Lemma~1.1]{Thomee:2006} and \cite[Theorems~3.16 and~3.18]{ern-guermond}, for $s \in [1,2]$:
\begin{equation}\label{ph-bound}
  \begin{aligned}
    \|P_h \psi - \psi\|_{L^2(\Omega)} 
    + \|R_h(q)\psi - \psi\|_{L^2(\Omega)} 
    \le c h^s \|\psi\|_{\dot{H}^s(\Omega)}, 
    \quad \forall \psi \in H^s(\Omega) \cap H_0^1(\Omega).
  \end{aligned}
\end{equation}
Since $q \in \A$, the constant $c$ is independent of $q$. 
Let $\mathcal{I}_h$ denote the Lagrange interpolation operator associated with the finite element space $X_h$. 
It satisfies the following error estimates for $s = 1,2$ and $1 \le p \le \infty$ with $sp > d$ 
(see \cite[Theorem~1.103]{ern-guermond}):
\begin{align}\label{eqn:int-err}
  \|v - \mathcal{I}_h v\|_{L^p(\Omega)}  
  + h\|v - \mathcal{I}_h v\|_{W^{1,p}(\Omega)}  
  \le c h^s \|v\|_{W^{s,p}(\Omega)}, 
  \quad \forall v \in W^{s,p}(\Omega).
\end{align}
Similarly, let $\mathcal{I}_h^{\partial}$ denote the Lagrange interpolation operator on the boundary.

Let $\gamma_0$ be the trace
operator \cite[Section B.3.5]{ern-guermond}, and the set
$ X_{h}^\partial = \left\{\gamma_0(\chi_h): \  \chi_h \in X_h \right\}. $
 We solve $D(x)$ numerically by
\begin{equation*}
  (\nabla D_h , \nabla \chi_h) = 0\quad \text{for all}~~ \chi_h \in X_h^0,\quad \text{and} \quad \gamma_0(D_h) =\Ih^\partial b.
\end{equation*}



The semi-discrete scheme for \eqref{eqn:pdeo} reads: find $ U_h(t;q)\in X_h$ such that $\gamma_0( U_h(t;q)) = \Ih^\partial b$ on $\partial\Omega$  and  for all  $\chi_h \in X_h^0 $ and $t>0$, 
\begin{equation*}\label{eqn:semi-i}
\begin{aligned}
         ( \Dal  U_h(t),\chi_h)  +  (\nabla  U_h(t), \nabla \chi_h)+(q   U_h(t), \chi_h)&= (f,\chi_h)\quad\text{with }\quad  U_h(0)=P_h( v-D_h)+D_h
\end{aligned}
\end{equation*}
For $q\in \A$, we define the discrete operator $A_h(q):\, X_h^0\to X_h^0$ such that
$$(A_h(q)\xi_h,\chi_h)  = (\nabla\xi_h,\nabla\chi_h) + (q \xi_h,\chi_h)  \quad  \text{for all} ~\  \xi_h,\chi_h \in X_h^0.$$
 In particular, we denote $A_h:=A_h(0)=-\Delta_h$.
By splitting the semi-discrete solution  \eqref{eqn:semi-i} as $u_h(t;q) =  U_h(t;q) - D_h$,
we observe that $  u_h(t;q)\in X_h^0$ satisfies

\begin{equation}\label{eqn:fy_hqt}
    \begin{aligned}
        \Dal u_h(t;q)  +  A_h(q) u_h(t;q)  = P_h [f - qD_h] \quad\text{with }\quad u_h(0;q)=P_h( v-D_h).
    \end{aligned}
\end{equation}
By means of Laplace transform, 
 the semi-discrete solution in~\eqref{eqn:fy_hqt} could be written as:
\begin{equation} \label{eqn:sol-rep-semi}
\begin{aligned}
  u_h(t;q)=  F_{q,h}(t)P_h( v-D_h)+ \int_0^t   E_{q,h}(s)P_h [f - qD_h] \ \ds,
\end{aligned}
\end{equation}
where $F_{q,h}(t)$ and  $E_{q,h}(t)$  are defined by
\begin{equation}\label{sol_dis_op}
\begin{aligned}
F_{q,h}(t) &= \frac{1}{2\pi i }\int_\contour e^{zt} z^{\alpha-1}(z^\alpha+A_h(q))^{-1}\ \dz,\quad
E_{q,h}(t) &= \frac{1}{2\pi i}\int_\contour  e^{zt} (z^\alpha+A_h(q))^{-1}\  \dz.
 \end{aligned}
\end{equation}
The discrete operators $F_{q,h}(t)$ and $E_{q,h}(t)$ satisfy the following smoothing property,
whose proof is identical to that of Lemma~\ref{lem:op}.
\begin{lemma}\label{lem:op-semi}
Let $F_{q,h}(t)$ and $E_{q,h}(t)$ be the solution operators defined in \eqref{sol_dis_op} for $q\in\A$.
Then  there holds
\begin{align*}
    t^{\nu\alpha}\|A_h(q)^{\nu}F_{q,h}(t)\|+t^{1-(1-\nu)\alpha}\|A_h(q)^{\nu}E_{q,h}(t)\|\le c_0, \ \text{and}\ \|(A_h(q)  F_{q,h}(t))^{-1}\| \le  c_0 (1+t^{\alpha}),
\end{align*}
here $0\le \nu\le 1$ and $c_0>0$ is the constant  the same as  Lemma~\ref{lem:op}.
\end{lemma}



 Let $  u_{\gamma,h}(t;q)\in X_h^0$ be the semidiscrete solution of the  regularization  equation  
  \eqref{eqn:back_non_reg:}   satisfying
  \begin{equation}\label{eqn:back_non_reg:semi}
          \begin{aligned}
                 \Dal u_{\gamma,h}(t;q)  +  A_h(q) u_{\gamma,h}(t)  &= P_h [f - qD_h] \quad\text{with}\quad\gamma u_{\gamma,h}(0;q)+ u_{\gamma,h}(T_1;q)=P_h(g_1-D):=P_h\bar g_1.
          \end{aligned}
  \end{equation}
 
 Using solution representation  and the condition $\gamma u_{\gamma,h}(0;q)+ u_{\gamma,h}(T_1;q)=P_h\bar g_1$ gives  
 \begin{align}
           u_{\gamma,h}(0;q)&=\left(\gamma I+F_{q,h}(T_1)\right)^{-1}(  P_h\bar g_1-(I-F_{q,h}(T_1)) A_h(q)^{-1}P_h [f - qD_h])\label{solu:ugam_semi}.
 \end{align}
The following lemma is a discrete counterpart to \cite[Lemma 15.3]{JinZhou:2023book}, and we omit the proof. 
\begin{lemma}\label{lem:op-reg-semi}
Let $F_{q,h}(t)$ be the discrete solution operator defined in \eqref{sol_dis_op}.
Then $ v_h\in X_h^0$, there holds that
\begin{equation*}
\quad \|F_{q,h}(t)(\gamma I+F_{q,h}(T))^{-1} v_h\|_{L^2\II}\le c\min(\gamma^{-1},t^{-\al})\| v_h\|_{L^2\II},
\end{equation*}
where the constant $c$  may depend on $T$, but  is independent of $\gamma$, $h$, $t$.
\end{lemma}


We now  demonstrate the error between $ u_{\gamma,h}(0;q)$ and $ u_\gamma(0;q)$.
\begin{lemma} \label{lem:esitugh}  
Let $ u_{\gamma}(t;q)$ and $ u _{\gamma,h}(t;q)$ be the solutions to the regularized  problem \eqref{eqn:back_non_reg:} 
and  \eqref{eqn:back_non_reg:semi}, respectively. 
 Then it holds that 
    \begin{equation*}
        \|  u_{\gamma,h}(0;q)-  u_{\gamma}(0;q)\|_{L^2\II}\le c\gamma^{-1}h^2,
    \end{equation*}
    where $c$ is a constant independent on $\gamma$, $h$.
\end{lemma}
\begin{proof} 
We shall split $ u _{\gamma,h}(0;q) - u _{\gamma}(0;q)$ as:
\begin{equation*}
      \begin{aligned}
   u _{\gamma,h}(0;q) - u _{\gamma}(0;q) &=  u _{\gamma,h}(0;q)-P_h u _{\gamma}(0;q)+P_h u _{\gamma}(0;q)- u _{\gamma}(0;q)=:\zeta_h(0)+\rho.
\end{aligned} 
\end{equation*}
The projection property~\eqref{ph-bound} and Lemma \ref{lem:regugamma} indicate that
\begin{align*}
    \|P_h u _{\gamma}(0;q)- u _{\gamma}(0;q)\|_{L^2\II}\le ch^2\| u _{\gamma}(0;q)\|_{\dot H^2\II}\le c\gamma^{-1}h^2.
\end{align*}
We now turn to derive the bound of $\zeta_h(0)$. From \eqref{eqn:back_non_reg:} 
and  \eqref{eqn:back_non_reg:semi}, using the identity $A_h(q)R_h(q) \psi = P_h A_q \psi$ for $\psi\in H^2\II\cap H_0^1\II$ leads to 
\begin{equation}\label{eqn:zeta_h}
   \begin{aligned}
   \Dal \zeta_h(t)+A_h(q)\zeta_h(t)=(A_h(q)(R_h(q)-P_h) u _{\gamma}(s;q)\quad \text{with}\quad\gamma\zeta_h(0)+\zeta_h(T_1)=0.
\end{aligned} 
\end{equation}
Using the solution representation and the  condition $\gamma\zeta_h(0)+\zeta_h(T_1)=0$  yields
\begin{align*}
    \zeta_h(0)=&-\left(\gamma I+F_{q,h}(T_1)\right)^{-1}\int_0^{T_1}E_{q,h}(T_1-s)A_h(q)(R_h(q)-P_h) u _{\gamma}(s;q)\ \ds.
\end{align*}
Applying  identity $A_h(q)E_{q,h}(t)=-F'_{q,h}(t)$, Lemma \ref{lem:Dalu} and Lemma~\ref{lem:regugamma} gives 
    \begin{align*}
  &\lnorm \int_0^{T_1}E_{q,h}(T_1-s)A_h(q)(R_h(q)-P_h) u _{\gamma}(s;q)\ \ds \rnorm_{L^2\II}\\
   = &  \bigg\| \int_0^{\frac{T_1}{2}}E_{q,h}(T_1-s)A_h(q)(R_h(q)-P_h) u _{\gamma}(s;q)\ \ds +\int^{T_1}_{\frac{T_1}{2}}F_{q,h}(T_1-s)(R_h(q)-P_h) u_{\gamma}'(s;q)\ \ds 
    \\\quad&+(R_h(q)-P_h) u_{\gamma}(T_1;q) +F_{q,h}(\frac{T_1}{2})(R_h(q)-P_h) u_{\gamma}(\frac{T_1}{2};q)  \bigg\|_{L^2\II}
  \\  \le &c h^2\left(\int_0^{\frac{T_1}{2}}(T_1-s)^{-1}\| u_{\gamma}(s;q)\|_{\dot H^2\II}\ds+\int^{T_1}_{\frac{T_1}{2}}\| u_{\gamma}'(s;q)\|_{\dot H^2\II}\ds +\| u_{\gamma}(T_1;q)\|_{\dot H^2\II}+\| u_{\gamma}(\frac{T_1}{2};q)\|_{\dot H^2\II}\right)
  \\ \le& ch^2\left(\int_0^{\frac{T_1}{2}}(T_1-s)^{-1}(1+s^{-\al})\ds +\int^{T_1}_{\frac{T_1}{2}}s^{-\al-1}\ds +(1+T_1^{-\al})+ (1+(\frac{T_1}{2})^{-\al})\right)\le ch^2.
\end{align*}
The desired results
can be done from Lemma \ref{lem:op-reg-semi}.
\end{proof}
We now  demonstrate the error between $ \partial_t^\alpha  u_{\gamma,h}(t;q)$ and $\partial_t^\alpha  u_{\gamma}(t;q)$.
\begin{lemma}\label{lem:uh-err} 
Let  $ u_{\gamma}(t;q)$ and  $ u_{\gamma,h}(t;q)$ be the solutions to the regularized  problem \eqref{eqn:back_non_reg:} 
and  \eqref{eqn:back_non_reg:semi}, respectively. Then there holds
$$  \|  \partial_t^\alpha  u_{\gamma,h}(t;q)-\partial_t^\alpha  u_{\gamma}(t;q) \|_{L^2\II} \le c\gamma^{-1}h^{2}t^{-\al}$$
with the constant $c$ is independent of  $\gamma,h$ and $t$, but may depend on $T_1$.
\end{lemma}
\begin{proof}
 We shall use the splitting 
 \begin{align*}
    \partial_t^\alpha  u_{\gamma,h}(t;q)-\partial_t^\alpha  u_{\gamma}(t;q)=  \partial_t^\alpha  u_{\gamma,h}(t;q)-P_h\partial_t^\alpha  u_{\gamma}(t;q)+P_h\partial_t^\alpha  u_{\gamma}(t;q)-\partial_t^\alpha u_{\gamma}(t;q):=\rho_1^h(t)+\rho_2^h(t).
 \end{align*}
 The projection property~\eqref{ph-bound} and lemma \ref{lem:regugamma}  give 
 \begin{align*}
     \|\rho_2^h(t)\|_{L^2\II}\le c h^2 \|\partial_t^\alpha  u_{\gamma}(t;q)\|_{\dot H^2\II}\le c\gamma^{-1} h^2t^{-\al}.
 \end{align*}
 Following \eqref{eqn:zeta_h}, we can find that $\rho_1^h(t)$ solves
 \begin{align*}
   \Dal \rho_1^h(t)+A_h(q)\rho_1^h(t)=(A_h(q)(R_h(q)-P_h)\Dal u _{\gamma}(s;q),
\end{align*}
 with $\rho_1^h(0)=A_h(q)(R_h(q)u_{\gamma}(0)-u_{\gamma,h}(0)) +P_h(qD-qD_h)$.
 
Employing the Laplace transform, the solution can be given by 
 \begin{align*}
     \rho_1^h(t)=&F_{q,h}(t) \rho_1^h(0)+\int_0^tE_{q,h}(t-s) A_h(q)(R_h(q)-P_h)\Dal u _{\gamma}(s;q)\ds.
 \end{align*}
 Applying the same argument as Lemma~\ref{lem:esitugh}, with the regularity of $\Dal u _{\gamma}(s;q)$ in Lemma~\ref{lem:regugamma} gives
 \begin{align*}
    & \lnorm\int_0^tE_{q,h}(t-s) A_h(q)(R_h(q)-P_h)\Dal u _{\gamma}(s;q)\ds\rnorm_{L^2\II}\\\le &c h^2\bigg(\int_0^{\frac{t}{2}}(t-s)^{-1}\|\Dal u_{\gamma}(s;q)\|_{\dot H^2\II}\ds+\int^{t}_{\frac{t}{2}}\|(\Dal u_{\gamma}(s;q))'\|_{\dot H^2\II}\ds\\&\quad +\|\Dal u_{\gamma}(t;q)\|_{\dot H^2\II}+\|\Dal u_{\gamma}(\frac{t}{2};q)\|_{\dot H^2\II}\bigg)\\
     \le &c\gamma^{-1}h^2\left(\int_0^{\frac{t}{2}}(t-s)^{-1} s^{-\al}\ds + \int^{t}_{\frac{t}{2}}s^{-\al-1}\ds+t^{-\al}+(\frac{t}{2})^{-\al}\right)\le c\gamma^{-1}h^2t^{-\al}.
 \end{align*}
 In addition, Lemma~\ref{lem:op-semi} gives
  \begin{align*}
     \lnorm F_{q,h}(t) \rho_1^h(0)\rnorm_{L^2\II}\le& \lnorm A_h(q)F_{q,h}(t) (R_h(q)u_{\gamma}(0)-u_{\gamma,h}(0))\rnorm_{L^2\II}+ \|F_{q,h}(t) P_h(qD-qD_h)\|_{L^2\II}\\\le &ct^{-\al}(\|R_h(q)u_{\gamma}(0)-u_{\gamma,h}(0)\|_{L^2\II}+\|qD-qD_h\|_{L^2\II})\le c\gamma^{-1}h^2t^{-\al},
 \end{align*}
 where we use  $\|R_h(q)u_{\gamma}(0)-u_{\gamma,h}(0)\|_{L^2\II}\le \|R_h(q)u_{\gamma}(0)-u_{\gamma}(0)\|_{L^2\II}+\|u_{\gamma}(0)-u_{\gamma,h}(0)\|_{L^2\II} \le c\gamma^{-1}h^2$ and  $\|qD-qD_h\|_{L^2\II})\le ch^2$. 
 This completes the proof of the lemma.
\end{proof}

\subsection{Error analysis of fully discrete reconstruction}
In this section, we derive an error bound for the fully discrete scheme applied to the backward problem. To discretize in time, the interval $[0, T]$ is divided into a uniform grid with $t_n = n\tau$ for $n = 0, \ldots, N$, where $t_{N_1} = T_1$, $t_{N_2} = T_2$, and $\tau = T/N$ represents the time step size.
The fractional derivative is then approximated using the backward Euler convolution quadrature (with $\varphi^j = \varphi(t_j)$), as in \cite{lubich1988convolution}:
\begin{align*}
\bar\partial_\tau^\alpha \varphi^n = \sum_{j=0}^n \omega_{n-j}^{(\alpha)} (\varphi^{j} - \varphi^0),
\quad \text{with} \quad \omega_j^{(\alpha)} = (-1)^j \frac{\Gamma(\alpha+1)}{\Gamma(\alpha-j+1)\Gamma(j+1)}.
\end{align*}

The fully discrete scheme for \eqref{eqn:pdeo} then reads: find 
$U_h^n(q) \in X_h$ with $u_h^n(q) = U_h^n(q) - D_h \in X_h^0$ such that, for all $n \ge 1$,

\begin{equation}\label{eqn:fy_hqn}
    \begin{aligned}
         \bDal  u_h^n(q)   +  A_h(q) u_h^n(q)  = P_h [f -qD_h] \quad\text{with}\quad u_h^0(q) = P_h( v-D_h).
    \end{aligned}
\end{equation}
Then analogue to \eqref{eqn:sol-rep-semi}, the fully discrete solution in
\eqref{eqn:fy_hqn} could be written as:
\begin{equation} \label{eqn:sol-rep-fully}
\begin{aligned}
 u_h^n(q) &=  F_{h,\tau}^n(q)P_h( v-D_h)+ \tau \sum_{j=1}^{n}E_{h,\tau}^{n-j}(q)P_h [f - qD_h] \\
&=  F_{h,\tau}^n(q)P_h( v-D_h)+ (I- F_{h,\tau}^n(q)) A_h(q)^{-1}P_h [f - qD_h],
\end{aligned}
\end{equation}
where
    \begin{align}\label{eq:op-fully}
        F_{h,\tau}^n(q)=\frac{1}{2\pi i}\int_{\Gamma_{\theta,\sigma}^\tau} e^{zt_{n-1}}\delta_{\tau}(e^{-z\tau})^{\alpha-1}G_h(z)\ \dz \quad \text{and}  \quad E_{h,\tau}^n(q)=\frac{1}{2\pi i}\int_{\Gamma_{\theta,\sigma}^\tau} e^{zt_n}
        G_h(z)\ \dz 
    \end{align}
with $G_h(z)=(\delta_\tau(e^{-z\tau})^{\alpha}+A_h(q))^{-1},\ \delta_\tau(\xi)=(1-\xi)/\tau$ and the contour
$\Gamma_{\theta,\sigma}^\tau :=\{ z\in \Gamma_{\theta,\sigma}:|\Im(z)|\le {\pi}/{\tau} \}$, oriented with an increasing imaginary part, where $\theta\in(\pi/2,\pi)$ is close to $\pi$. 
Here we use the identity $$A_{h}(q) E_{h,\tau}^n(q)=-\partial_\tau F^{n+1}_{h,\tau}(q)$$ to obtain the second equality in \eqref{eqn:sol-rep-fully}, with the difference quotient $\partial_\tau f^n=\frac{f^n-f^{n-1}}{\tau}$.


The discrete operators $F_{h,\tau}^n(q)$ and $E_{h,\tau}^n(q)$ satisfy the following smoothing properties. 
The proof of part~(i) can be found in \cite[Lemma~4.6]{zhang2022identification} and \cite[Corollary~15.3]{JinZhou:2023book}. 
The proofs of parts~(ii) and~(iii) will be presented below.

\begin{lemma}\label{lem:op-fully}
Let $F_{h,\tau}^n(q)$ and $E_{h,\tau}^n(q)$ be the solution operators defined in \eqref{eq:op-fully} and $F_{h,q}(t_n)$ defined in \eqref{sol_dis_op} for $q\in\A$.
Then  there exists a $  c_{0}>0$ independent of $q$, $h$, $\tau$ and $n$ such that  for all $n>1$, $0\le \nu \le 1$:
\begin{itemize}
\item[$\rm(i)$] $t_n^{\nu\alpha}\|A_h(q)^{\nu}F_{h,\tau}^n(q)\|+t_{n+1}^{1-(1-\nu)\alpha}\|A_h(q)^{\nu}E_{h,\tau}^n(q)\|\le  c_{0}$, $ \|(A_h(q)  F_{h,\tau}^{N_1}(q))^{-1}\| \le c_{0} (1+T_1^{\alpha}) $;
\item[$\rm(ii)$] $ \| F_{h,\tau}^j(q) F_{h,\tau}^{N_1}(q)^{-1}\|\le  c_{0}(1+t_{j}^{-\al}T_1^\al),\ \|\partial_\tau F_{h,\tau}^j(q)F_{h,\tau}^{N_1}(q)^{-1}\|\le c_{0}t_j^{-1}(1+t_j^{-\al}T_1 ^\al) $  ;
\item[$\rm(iii)$]    $ \| F_{h,\tau}^j(q) F_{h,q}(t_n)^{-1}\|\le  c_{0}(1+t_{j}^{-\al}t_n^\al),\quad \|\partial_\tau F_{h,\tau}^j(q)F_{h,q}(t_n)^{-1}\|\le c_{0}t_j^{-1}(1+t_j^{-\al}t_n ^\al) $.
\end{itemize}
\end{lemma}
\begin{proof}
 For any $\bar v_h\in X_h^0$, we have  the following  equivalence formula of  $F_{h,\tau }^j(q)$   \cite[equation 15.21]{JinZhou:2023book},
\begin{align*}
F_{h,\tau}^j(q)\bar v_h=\sum_{n=1}^{N_h}F_{\tau}^j(\lambda_h^n(q))(\bar v_h, \fy_n^h(q)) \fy_n^h(q)\quad\text{with}\quad  F_{\tau}^j(\lambda)= \frac{1}{2\pi\mathrm{i}}\int_{\Gamma_{\theta,\sigma}^\tau } e^{zt_{j-1}}\delta_\tau(e^{-z\tau})^{-1}G_\tau(z,\lambda)\,\d z,
\end{align*}
where $G_\tau(z,\lambda)= \delta_\tau(e^{-z\tau})^{\al}({ \delta_\tau(e^{-z\tau})^\alpha}+\lambda)^{-1}$, and $(\lambda_h^n(q),\fy_n^h(q))_{n=1}^{N_h}$ denote the eigenpairs of $A_h(q)$, with $\{\fy_n^h(q)\}_{n=1}^{N_h}$ forms an orthonormal eigenfunctions in $L^2(\Omega)$.
Now from  \cite[Corollary 15.3]{JinZhou:2023book}, we can find that $  F^j_{\tau}(\lambda)\le c(1+\lambda t_j)^{-1},\ \forall n\ge 0$
  and $F^{N_1}_{\tau}(\lambda)\ge c(1+\lambda T_1)^{-1}$.
Therefore, it is straightforward to derive that
\begin{align}\label{iifull}
    \|F_{h,\tau}^j(q) F_{h,\tau}^{N_1}(q)^{-1} \bar v_h\|_{L^2\II}\le c\sum_{n=1}^{N_h}\left|\frac{1+\lambda_h^n(q) T_1^\al}{1+\lambda_h^n(q) t_j^\al}\right|^2 (\bar v_h, \fy_n^h(q))^2\le c(1+t_j^{-\al}T_1^\al)^2\|\bar v_h\|^2_{L^2\II}.
\end{align}
This completes the proof of the first desired estimate for (ii). 

Notice  that 
$
    \partial_\tau  F_{\tau}^{j}(\lambda)=\frac{1}{2\pi\mathrm{i}}\int_{\Gamma_{\theta,\sigma}^\tau } e^{zt_{j-1}} G_\tau(z,\lambda)\,\d z,
$
 and 
$| G_\tau(z,\lambda)|\le c\min (|z|^{-1+\al},\lambda^{-1}|z|^\al)$. 
Therefore, by setting $\sigma=t_{j}^{-1}$ in the contour $\Gamma_{\theta,\sigma}^\tau $, we can derive that (similar as  \cite[Lemma 15.6]{JinZhou:2023book})
\begin{align*}
   | \partial_\tau  F_{\tau}^{j}(\lambda)|\le   c\min(t_{j}^{-1},\lambda^{-1}t_{j}^{-1-\al})\le ct_{j}^{-1}(1+\lambda t_{j}^{\al})^{-1}.
\end{align*}
Consequently, the rest of the estimates can be proved with a similar argument as \eqref{iifull}.
\end{proof}
We now state \textsl{a priori} regularity results for the fully discrete solution $u_h^n(q)$ given in  \eqref{eqn:sol-rep-fully}.
\begin{lemma}\label{lem:Dalufully}
Let $v, f, D\in L^2\II$ and $q\in \A$, and $u_h^n(q)$ be the fully discrete solution  given in  \eqref{eqn:sol-rep-fully}. Then it holds that
\begin{align*}
t_n\|   \partial_\tau  u_h^n(q) \|_{L^2\II} +\|     u_h^n(q)\|_{L^2\II}\le \hat c_{v,f,D},
\end{align*}
where the constant $\hat c_{v,f,D}>0$ is  independent of $q$ and $t$,  but may depend on  $v$, $f$ and  $D$.
\end{lemma}
\begin{proof}
    Employing the fully discrete solution representation  \eqref{eqn:sol-rep-fully}, Lemma~\ref{lem:op-fully} gives $ \|   u_h^n(q)\|_{L^2\II}\le \hat c_{v,f,D}$.
     Using the identity $A_h(q)E_{h,\tau}^n(q)=-\partial_\tau F^{n+1}_{h,\tau}(q)$ yields 
    \begin{align*}
         \| \partial_\tau  u_h^n(q)\|_{L^2\II}=\left\|A_h(q)E_{h,\tau}^{n-1}(q)  \left(P_h( v-D_h) -A_h(q)^{-1}P_h [f - qD_h]\right)\right\|_{L^2\II}\le \hat c_{v,f,D}t_n^{-1}.
    \end{align*}
    This completes the proof of the
lemma.
    \end{proof}
    The next lemma gives a discrete version of Lemma~\ref{Lem:sobemb}.
\begin{lemma}\label{Lem:sobembdis}
    For any $u, v\in L^2\II$, $q\in \A$, there exists a constant $\hat c_{1}$,  independent of $q$ and $h$  such that 
    \begin{align*}
        \|A_h(q)^{-1}  P_hu\|_{L^\infty\II}\le \hat c_1\|u\|_{L^2\II},\quad \|A_h(q)^{-1} P_h(uv)\|_{L^2\II}\le \hat c_1\|u\|_{L^2\II}\|v\|_{L^2\II}.
    \end{align*}
\end{lemma}
 We now let 
 $  u_{\gamma,h}^{\delta,n}(q) \in X_h^0$ be the fully discrete solution of the  regularization  equation  \eqref{eq:back_non_reg:noisy}   satisfying
 \begin{equation}\label{eq:back_non_reg:noisy:fully}
         \begin{aligned}
              \bDal u^{\delta,n} _{\gamma,h}(q)  +  A_h(q) u^{\delta,n} _{\gamma,h}(q)  = P_h [f - qD_h] \quad \text{with}\quad \gamma u_{\gamma,h}^{\delta,0}(q)+ u_{\gamma,h}^{\delta,N_1}(q)=P_h( g_1^\delta-D_h).
         \end{aligned}
 \end{equation}
 Using solution representation yields  
  \begin{align}\label{eqn:udisdeelayh}
       u_{\gamma,h}^{\delta,n}(q)=  F_{h,\tau}^n(q)  u^{\delta,0} _{\gamma,h}(q) + (I- F_{h,\tau}^n(q)) A_h(q)^{-1}P_h [f - qD_h].
 \end{align}
 The condition  $\gamma u_{\gamma,h}^{\delta,0}(q)+ u_{\gamma,h}^{\delta,N_1}(q)=P_h( g_1^\delta-D_h)$  implies
 \begin{align}
      u _{\gamma,h}^{\delta,0}(q)&=(\gamma I+F_{h,\tau}^{N_1}(q))^{-1}\left( P_h( g_1^\delta-D_h)-(I-F_{h,\tau}^{N_1}(q)) A_h(q)^{-1}P_h [f - qD_h]\right).\label{solu:ugamdel_fully}
 \end{align}
 We  now intend to derive an \textsl{a priori} estimate for  $  u_{\gamma,h}(0;q) -    u_{\gamma,h}^{\delta,0}(q)$.
\begin{lemma}\label{lem:uhn-err}  
Let $ u_{\gamma,h}(0;q)$ and $ u_{\gamma,h}^{\delta,0}(q)$ satisfy \eqref{solu:ugam_semi}
and \eqref{solu:ugamdel_fully}, respectively. Then there holds
$$  \|   u_{\gamma,h}(0;q)- u_{\gamma,h}^{\delta,0}(q) \|_{L^2\II} \le c(\tau +\gamma^{-1}h^2+\gamma^{-1}\delta) $$
with the constant independent of $q,\ \gamma,\ \delta,\  \tau$ and $h$, but may depend on $T_1$,  $g_1$, $f$ and  $D$.
\end{lemma}
\begin{proof} Using the representation \eqref{solu:ugam_semi} and \eqref{solu:ugamdel_fully} leads to 
    \begin{align*}
     u_{\gamma,h}(0;q)- u_{\gamma,h}^{\delta,0}(q)
    =&\left(\left(\gamma I+F_{q,h}(T_1)\right)^{-1}(P_h-R_h(q))\bar g_1-(\gamma I+F_{h,\tau}^{N_1}(q))^{-1}(P_h-R_h(q))\bar g_1\right)\\
    &+\left(\left(\gamma I+F_{q,h}(T_1)\right)^{-1}-(\gamma I+F_{h,\tau}^{N_1}(q))^{-1}\right)R_h(q)\bar g_1
    \\&+\left(A_h(q)\left(\gamma I+F_{q,h}(T_1)\right)\right)^{-1} \left(F_{q,h}(T_1)-F_{h,\tau}^{N_1}(q)\right)P_h [f - qD_h]\\
    &+\left(\left(\gamma I+F_{q,h}(T_1)\right)^{-1}-(\gamma I+F_{h,\tau}^{N_1}(q))^{-1}\right)\left(I-F_{h,\tau}^{N_1}(q)\right) A_h(q)^{-1}P_h [f - qD_h]\\
    &+(\gamma I+F_{h,\tau}^{N_1}(q))^{-1}(P_h(g_1-D)-P_h( g_1^\delta-D_h)) =: \sum_{i=1}^5 {\rm I}_i.
\end{align*}
Similar as the arguments in \cite[Lemmas 15.7--15.8]{JinZhou:2023book}, there holds
\begin{align*}
  & \| {\rm I_{1}}\|_{L^2\II}\le  c_{T_1}\gamma^{-1}h^2\|\bar g_1\|_{\dot H^2\II}   ,\quad \| {\rm I_{5}}\|_{L^2\II}
    \le c (\gamma^{-1}\delta+\gamma^{-1}h^2),\\
    &\| {\rm I_{2}}\|_{L^2\II}+\| {\rm I_{4}}\|_{L^2\II}\le  c_{T_1}\tau(\|\bar g_1\|_{\dot H^2\II}+\|f\|_{L^2\II}+\|D\|_{L^2\II}).
\end{align*}
For the term $ {\rm I_{3}}$, using Lemma~\ref{lem:op-fully} and the direct operator estimates \cite[Lemma 4.2]{zhang2022identification} yields
\begin{align*}
     \| {\rm I_{3}}\|_{L^2\II}\le c \left\|F_{q,h}(T_1)-F_{h,\tau}^{N_1}(q)\right\|\left\|P_h [f - qD_h]\right\|_{L^2\II}\le  c_{T_1,f,D}\tau. 
\end{align*}
This completes the proof of the lemma.
\end{proof}

We now provide  an \textsl{a priori} estimate for  $ \Dal  u_{\gamma,h}(t_n;q) -   \bDal  u_{\gamma,h}^{\delta,n}(q)$.
\begin{lemma}\label{lem:daluhn-err}
Let $ u _{\gamma,h}(t;q)$ and  $ u^{\delta,n} _{\gamma,h}(q)$ satisfy \eqref{eqn:back_non_reg:semi}
and \eqref{eq:back_non_reg:noisy:fully}, respectively. Then there holds
$$  \lnorm\Dal  u_{\gamma,h}(t_n;q)-\bDal  u_{\gamma,h}^{\delta,n}(q) \rnorm_{L^2\II} \le c(\tau +\gamma^{-1}h^2)\max (t_n^{-\al},t_n^{-\al-1}) $$
with the constant independent of $q,\gamma, h, \tau$ and $t_n$ but may depend on $T_1$,  $g_1$, $f$ and  $D$.
\end{lemma}
\begin{proof}
Let $w_h(t;q)=\Dal  u_{\gamma,h}(t;q)$ and   $w_h^n(q)=\bDal  u_{\gamma,h}^{\delta,n}(q)$. From \eqref{eqn:back_non_reg:semi} and \eqref{eq:back_non_reg:noisy:fully}  there holds
    \begin{align*}
        \Dal w_h(t;q)+A_h(q)w_h(t;q)&=0,\quad \text{with } \quad w_h(0;q)=-A_h(q) u _{\gamma,h}(0;q)+P_h [f - qD_h], \\
          \bDal w_h^n(q)+A_h(q)w_h^n(q)&=0,\quad \text{with } \quad w_h^0(q)=-A_h(q) u _{\gamma,h}^{\delta,n}(q)+P_h [f - qD_h].
    \end{align*}
    By means of Laplace transform,
 the solutions    can be written as 
    \begin{align*}
        w_h(t;q)=F_{q,h}(t)w_h(0;q),\quad w_h^n(q)=F_{h,\tau}^n(q) w_h^0(q).
    \end{align*}
   Therefore,  employing Lemma~\ref{lem:esitugh}, Lemma~\ref{lem:regugamma},  Lemma~\ref{lem:op-fully}, Lemma~\ref{lem:uhn-err} and \cite[Lemma 4.2]{zhang2022identification} gives
    \begin{equation*}
    \begin{aligned}
       \| w_h(t_n;q)-w_h^n(q)\|_{L^2\II}\le& \lnorm(F_{q,h}(t_n)-F_{h,\tau}^n(q))w_h(0;q)\rnorm_{L^2\II}+\lnorm F_{h,\tau}^n(q)(w_h(0;q)-w_h^0(q))\rnorm_{L^2\II}\\
       \le & \lnorm A_h(q)(F_{q,h}(t_n)-F_{h,\tau}^n(q)) u _{\gamma,h}(0;q)\rnorm_{L^2\II}\\&+\lnorm(F_{q,h}(t_n)-F_{h,\tau}^n(q))P_h [f - qD_h]\rnorm_{L^2\II}
       \\&+\lnorm A_h(q)F_{h,\tau}^n(q)( u _{\gamma,h}(0;q)- u _{\gamma,h}^{\delta,0}(q))\rnorm_{L^2\II}\\
       \le &c \tau t_n^{-\al-1}(1+\gamma^{-1}h^2)+c\tau t_n^{-\al}+ c(\tau +\gamma^{-1}h^2)t_n^{-\al}.
    \end{aligned}
    \end{equation*}
    This completes the proof of the lemma.
\end{proof}


%
For the potential $q\in\A$ and initial condition $v\in X_h^0$, we denote the fully discrete solution $ u_h^n$ to problem~\eqref{eqn:fy_hqn} by $ u_h^n(q,v)$. 
  The following two lemmas, which are discrete counterparts of Lemmas~\ref{lem:stab-0}--\ref{lem:stab-00}, play a crucial  role in establishing the contraction of the operator 
 $K_{h,\tau}$. Due to the length of the proof, it is deferred to the Appendix.
 \begin{lemma}\label{lem:Dal-uhn}
For $i=1,2$, let $ u_h^n(q_i,\bar v_i)$ be the solution to the fully discrete scheme \eqref{eqn:fy_hqn}, with potential $q_i\in \A$ and initial condition  $\bar v_i\in X_h^0$.
Then there holds 
\begin{equation*}
\|  \bar \partial_\tau^\alpha ( u_h^n(q_1,\bar v_1) -  u_h^n(q_2,\bar v_2)) \|_{L^2\II} \le \bar  c_{0,\bar v_2,f,D} t_n^{-\alpha}\left(\|\bar v_1-\bar v_2\|_{L^2\II}+\| q_1 - q_2 \|_{L^2\II}\right),
\end{equation*}
where the constant $\bar  c_{0,\bar v_2,f,D}>0$ is independent of $h$, $\tau$, $q_1$, $q_2$ and $t_n$, but may depend on   $\bar v_2$, $f$ and  $D$.
\end{lemma}
\begin{lemma}\label{lem:stab-0_reg_nois}
Let  potentials $q_i\in \A, \ i=1,2$ and    $ u_{\gamma,h}^{\delta,0}(q_i)$ satisfy \eqref{solu:ugamdel_fully}. Additionally, suppose that  $\gamma^{-1}(\delta+h^2)\le c$ for some constant $c$.  Then it holds that
\begin{equation*}
\left\| u_{\gamma,h}^{\delta,0}(q_1) - u_{\gamma,h}^{\delta,0}(q_2)\right \|_{L^2\II}\le \bar  c_{1,g_1,b,f}(1+T_1^\al)\|q_1-q_2\|_{L^2\II},
\end{equation*}
where the constant $\bar   c_{1,g_1,b,f}>0$ independent of $q_1$, $ q_2$ and $T_1$, but may depend on   $g_1$, $f$ and  $D$.
\end{lemma}

\subsection{Numerical reconstruction and error estimates}

This section is devoted to the construction of a robust fully discrete scheme for the simultaneous reconstruction of the initial condition and the potential. 
The subsequent analysis is carried out under the following assumption.
\begin{assumption}\label{assump:numerics}
We assume that the exact potential $q^\dag$, the exact initial condition $v^\dag$,  the source term $f$ satisfy the following conditions:
\begin{itemize}
 \item[(i)] $v^\dag-D \in \dot H^{s_1}\II$,  $f-q^\dag D\in H^{s_2}\II$, $q^\dag\in\A$;
\item[(ii)] $ q^\dag|_{\partial\Omega}$ is  \textsl{a priori} known, $v^\dag-D \ge 0$ a.e. in $\Omega$, $f-q^\dag D \ge 0$, a.e. in $\Omega$.
\end{itemize}

\end{assumption}


\begin{remark}\label{rk:41}
From the maximum principle Lemma~\ref{positive}, it holds that $g_2\ge M$  based on the above Assumption~\ref{assump:numerics}. It is reasonable to assume that the noisy data $g_2^\delta \geq M$. Otherwise, we may revise the observational data by $ \tilde g_2^\delta =\max(g_2^\delta(x),M),\ \forall x\in \Omega.$
There holds  $ \| \tilde g_2^\delta-g_2\|_{ L^2\II}\le  \|  g_2^\delta-g_2\|_{ L^2\II}\le\delta.$
Then $\tilde  g_2^\delta$ can be used as the observational data in computation, with $ \tilde g_2^\delta \ge M$.
\end{remark}
Due to the roughness of the noise, $\Delta g_2^\delta$ may not be well-defined in $L^2\II$, necessitating a numerical approximation of the unknown function $\Delta g_2$. To address this, we employ a two-grid method, where $H$ and $\bar{H}$ represent the spatial mesh sizes with $X_H \subset X_{\bar{H}}$. The function is projected onto the fine grid with mesh size $\bar{H}$, while the Laplacian is approximated on the coarse grid with mesh size $H$.

We define a function $\psi_H \in X_H$ such that
\begin{equation}\label{eqn:psih}
\begin{aligned}
 \gamma(\psi_H)=P_{H}^{\partial}(f-q^\dag b),\quad
(\psi_H,\phi_H)=-(\nabla P_{X_{\bar H}} g_2^\delta, \nabla \phi_H)\quad~~\text{for all}~~ \phi_H \in X_H^0, 
\end{aligned}
\end{equation}
where  the $L^2$ orthogonal projection $P_{X_H}:L^2(\Omega)\rightarrow X_H$ and $P_H^\partial:L^2(\partial\Omega)\rightarrow X_H^\partial$ are defined by
\begin{align*}
\int_{\Omega} (P_{X_H} \psi) \ \chi_H\ \dx = \int_{\Omega}  \psi \ \chi_H\ \dx~~\text{and}~~ 
 \int_{\partial\Omega} (P_H^\partial \psi) \ \gamma_0(\chi_H)\ \ds = \int_{\partial\Omega}  \psi \ \gamma_0(\chi_H)\ \ds \quad\forall \chi_H\in X_H,
\end{align*}

Then the next lemma shows that $\psi_H$ is an approximation to $ \Delta g$ with appropriate choices of $H$ and $\bar H$.

\begin{lemma}\label{lem:reg-err}
Suppose that Assumption \ref{assump:numerics} is valid.
Let $\psi_H \in X_H$ be the function defined in \eqref{eqn:psih}.
Then there holds
\begin{equation*}
\|  \psi_H- \Delta g_2 \|_{L^2\II} \le c \Big(\delta \bar H^{-\frac{1}{2}} H^{-\frac{3}{2}}+\bar H^{\frac{3}{2}}H^{-\frac{3}{2}}+ H^{\min\{s_1,s_2\}}\Big),
\end{equation*}
where the constant $c$ is independent of $H, \bar H$ and $\delta$.
\end{lemma}

\begin{proof}
Let $D^0(x)$ be the solution of 
\begin{equation*}
    -\Delta D^0(x)=0 \quad \text{in} ~\Omega \qquad \text{with}\quad D^0(x)=f-q^\dag b\quad \text{on}~\partial\Omega.
\end{equation*}
Let $D^0_H$ be the numerical solution of $D^0$  with  $\gamma(D^0_H)=P_{H}^{\partial}(f-q^\dag b) $. 
Then there holds the estimate \cite[Theorem 1]{BrambleKing:1994} and \cite[Theorem 5.5]{berggren2004approximations} 
\begin{equation}\label{eqn:app-D0}
\|D^0_H-D^0\|_{L^2\II}\le cH^{\min \{s_1,s_2\}}.
\end{equation}
Moreover, according to Assumption \ref{assump:numerics}, we conclude that 
$\Delta g_2-D^0\in  \dH {\min \{s_1,s_2\}}$. 
The approximation property of $L^2$ orthogonal projection $P_H$ in \eqref{ph-bound} and  \cite[eq. (2.11)]{steinbach2001stability} 
leads to 
$$\|P_H(\Delta g_2-D^0)-(\Delta g_2-D^0)\|_{L^2\II}\le  cH^{\min \{s_1,s_2\}}.$$
This together with the estimate \eqref{eqn:app-D0} implies
\begin{equation*}
 \begin{aligned}
&  \| \psi_H - \Delta g_2  \|_{L^2\II} \\
\le &
\| (\psi_H-D^0_H) - P_H(\Delta g_2-D^0)  \|_{L^2\II}+
\|P_H(\Delta g_2-D^0)- (\Delta g_2-D^0)\|_{L^2\II} +  \|D_H^0 - D^0 \|_{L^2\II}\\
\le &\| (\psi_H-D^0_H) - P_H(\Delta g_2-D^0)  \|_{L^2\II} + cH^{\min \{s_1,s_2\}}.
 \end{aligned}
\end{equation*}
Then we apply the estimate \eqref{eqn:app-D0} 
and the definition of $\psi_h$ in \eqref{eqn:psih} to obtain that
 \begin{align*}
     \|(\psi_H-D^0_H) - P_H(\Delta g_2-D^0)\|_{L^2\II}=&\sup_{\phi_H \in X_H^0}\frac{\big((\psi_H-D^0_H) - (\Delta g_2-D^0), \phi_H\big)}{\| \phi_H \|_{L^2\II}}\\=&\sup_{\phi_H \in X_H^0}\frac{(\nabla (g_2-P_{X_{\bar H}} g_2^\delta), \nabla \phi_H)+(D^0 - D^0_H,\phi_H)}{\| \phi_H \|_{L^2\II}}\\
     \le &\sup_{\phi_H \in X_H^0} \frac{(\nabla (g_2 - P_{X_{\bar H}} g_2^\delta), \nabla \phi_H)}{\| \phi_H \|_{L^2\II}} + c H^{\min \{s_1,s_2\}}.\\
 \end{align*}
Next, we aim to prove the claim that 
\begin{equation}\label{eqn:app-g2}
\sup_{\phi_H \in X_H^0}\frac{(\nabla (g_2- P_{X_{\bar H}} g_2^\delta), \nabla \phi_H)}{\| \phi_H \|_{L^2\II}}
\le c(\delta \bar H^{-\frac{1}{2}} H^{-\frac{3}{2}}+\bar H^{\frac{3}{2}}H^{-\frac{3}{2}} ).
\end{equation}
Let $\{ K_j \}_{j=1}^{N_H} = \T_H$ and  $\{ \bar K_j \}_{j=1}^{N_{\bar H}} = \T_{\bar H}$, with $\T_{\bar H}\subset \T_{ H}$. We can derive 
\begin{align*}
& |(\nabla (P_{X_{\bar H}} g_2^\delta-g_2), \nabla \phi_H)|
\le \sum_{j=1}^{N_H}  
\Big|\int_{K_j} \nabla  (P_{X_{\bar H}} g_2^\delta-g_2) \cdot \nabla \phi_H \, \d x \Big|=\sum_{j=1}^{N_H}\Big|\int_{\partial K_j}  \frac{\partial \phi_H }{\partial n_j} (P_{X_{\bar H}} g_2^\delta-g_2) \d s\Big|\\
\le &\sum_{j=1}^{N_H} \Big\|\frac{\partial \phi_H }{\partial n_j}\Big\|_{L^2(\partial K_j)}\left( \|P_{X_{\bar H}} (g_2^\delta-g_2)\|_{L^2(\partial K_j)}+ \|P_{X_{\bar H}} g_2-g_2\|_{L^2(\partial K_j)}\right)\\
    \le &\sqrt{\sum_{j=1}^{N_H}  \Big\|\frac{\partial \phi_H }{\partial n_j} \Big\|^2_{L^2(\partial K_j)}} \left(\sqrt{\sum_{j=1}^{N_H}  \|P_{X_{\bar H}} (g_2^\delta-g_2)\|^2_{L^2(\partial K_j)}}+\sqrt{\sum_{j=1}^{N_H}  \|P_{X_{\bar H}} g_2-g_2\|^2_{L^2(\partial K_j)}}\right)
    \\
    \le &\sqrt{\sum_{j=1}^{N_H} \Big\|\frac{\partial \phi_H }{\partial n_j}\Big\|^2_{L^2(\partial K_j)}} \left(\sqrt{ \sum_{i=1}^{N_{\bar H}}  \|P_{X_{\bar H}} (g_2^\delta-g_2)\|^2_{L^2(\partial \bar K_i)}}+\sqrt{\sum_{i=1}^{N_{\bar H}} \|P_{X_{\bar H}} g_2-g_2\|^2_{L^2(\partial \bar K_i)}}\right)
    \\\le& cH^{-\frac{1}{2}}\|\nabla \phi_H\|_{L^2\II}(\delta\bar H^{-\frac{1}{2}}  +\bar H^{\frac{3}{2}})\le  c(\delta \bar H^{-\frac{1}{2}} H^{-\frac{3}{2}}+\bar H^{\frac{3}{2}}H^{-\frac{3}{2}} )\| \phi_H\|_{L^2\II},
\end{align*}
where we use the trace inequalities (see e.g. \cite[Lemma A.3]{wang2014weak} and  \cite[Lemma B.1]{Monk1998}) to obtain that
\begin{align*}
&\sum_{j=1}^{N_H} \Big\|\frac{\partial \phi_H }{\partial n_j}\Big\|^2_{L^2(\partial K_j)} \le   
c H^{-1}\sum_{j=1}^{N_H} \|\nabla \phi_H\|^2_{L^2( K_j)}   
\le cH^{-1}\|\nabla \phi_H\|_{L^2\II}^2\le cH^{-3}\| \phi_H\|_{L^2\II}^2,\\
& \sum_{i=1}^{N_{\bar H}}  \|P_{X_{\bar H}} (g_2^\delta-g_2)\|^2_{L^2(\partial \bar K_i)}\le   c\bar H^{-1}\sum_{i=1}^{N_{\bar H}}  \|P_{X_{\bar H}} (g_2^\delta-g_2)\|^2_{L^2( \bar K_i)}\le c\bar H^{-1} \|P_{X_{\bar H}} (g_2^\delta-g_2)\|_{L^2\II}^2 \le c\delta^2 \bar H^{-1},\\
 &  \sum_{i=1}^{N_{\bar H}} \|P_{X_{\bar H}} g_2-g_2\|^2_{L^2(\partial \bar K_i)}
 \le  c\sum_{i=1}^{N_{\bar H}}\left(  \bar H^{-1}\|P_{X_{\bar H}} g_2-g_2\|^2_{L^2(\bar K_i)}+\bar H \|\nabla(P_{X_{\bar H}} g_2-g_2)\|^2_{L^2(\bar K_i)}\right)\le c\bar H^{3} \|g_2\|_{H^2\II}^2.
\end{align*}
This shows \eqref{eqn:app-g2} and hence completes the proof of this lemma.

\end{proof}

\begin{remark}\label{rem:reg-err}
Compared to the results in \cite[Lemma 4.10 and Theorem 4.12]{zhang2022identification}, Lemma \ref{lem:reg-err} offers an improved error estimate with reduced smoothness requirements for both the problem data and observational data. Additionally, the estimate in \cite[Lemma 4.10]{zhang2022identification} heavily depends on superconvergence results, which are only applicable to uniform meshes. This restrictive condition has been removed in the present work. The improvement is further demonstrated through numerical experiments; see Figure \ref{rateQP1_Conti} for a comparison of the two approaches.
\end{remark}

We now define the operator $K_{h,\tau}: \A \rightarrow \A$ such that
\begin{equation}\label{eqn:Kh}
 K_{h,\tau} q (x) :=  P_{[0,m]}\left(\frac{f(x)-\bar \partial_\tau^\alpha u_h^{N_2}(x;q,u_{\gamma,h}^{\delta,0}(q)) + \psi_H(x)  }{g_2^\delta(x)}\right),
\end{equation}
where the function $P_{[0,m]}:\mathbb{R} \rightarrow \mathbb{R}$ denotes a truncation function given in \eqref{eqn:P0M1}. The next lemma shows a contraction property of the operator $K_{h,\tau}$.

\begin{lemma}\label{lem:Dal-uhn-2}
Let $q_1,q_2 \in  \A$. Then for fixed $T_1$, $g_1$, $f$ and  $b$,   let $\bar T_0$ be the solution of  
\begin{align*}
   M^{-1} \bar  c_{0}\bar T_0^{-\al}(\bar  c_{1}(1+T_1^\al)+1)=\bar c,
\end{align*}
where $\bar c $ is some constant in $(0,1)$,  $\bar  c_{0}$ and $\bar  c_{1}$ are the constants given by Lemmas~\ref{lem:Dal-uhn}--\ref{lem:stab-0_reg_nois}.
It holds that any $T_2\ge \bar T_0$
$$\| K_{h,\tau} q_1 - K_{h,\tau} q_2 \|_{L^2\II} \le    \bar  c_{0}T_2^{-\al}(\bar  c_{1}(1+T_1^\al)+1)\|  q_1 -  q_2 \|_{L^2\II}\le \bar c\|  q_1 -  q_2 \|_{L^2\II}, $$
\end{lemma}
\begin{proof} 
Based on  Remark~\ref{rk:41} and  Lemmas~\ref{lem:Dal-uhn}--\ref{lem:stab-0_reg_nois}, similar as the proof in Theorem~\ref{lem:contK},  we can derive that 
\begin{align*}
\|K_{h,\tau} q_1 - K_{h,\tau} q_2\|_{L^2\II} &\le M^{-1} \lnorm \bar \partial_\tau^\alpha \left(u_h^{N_2}(q_2,u_{\gamma,h}^{\delta,0}(q_2)) - u_h^{N_2}(q_1,u_{\gamma,h}^{\delta,0}(q_1)) \right) \rnorm_{L^2\II}\\
&\le
M^{-1} \bar c _{0}T_2^{-\al}\left(\lnorm u_{\gamma,h}^{\delta,0}(q_2)-u_{\gamma,h}^{\delta,0}(q_1)\rnorm_{L^2\II}+\lnorm q_1 - q_2 \rnorm_{L^2\II}\right)\\
&\le M^{-1}  \bar c _{0}T_2^{-\al}( \bar c _{1}(1+T_1^\al)+1)\lnorm q_1 - q_2 \rnorm_{L^2\II}.
\end{align*}
This completes the proof of the lemma.
\end{proof}

Now we are ready to present the main theorem of this section.

\begin{theorem}\label{thm:err-fully}
Suppose that  Assumption   \ref{assump:numerics} is valid.
Let $K_{h,\tau}$ be the operator defined in \eqref{eqn:Kh}.  For fixed $T_1$, $g_1$, $f$ and  $b$,   let $\bar T_0$ defined in Lemma~\ref{lem:Dal-uhn-2}, then there holds
for any  $T_2\ge \bar T_0,\ q_0 \in  \A$, the iteration
\begin{align}\label{eqn:iter-fully}
 q_{n+1} = K_{h,\tau} q_n,\qquad \forall~~ n=0,1,\ldots,
\end{align}
linearly converges to a unique fixed point $q^* \in L^\infty\II$  of $K_{h,\tau}$  with $0\le q^*\le m$ s.t.
\begin{align*}
\|   q^* - q_{n+1}  \|_{L^2\II}
\le  cT_2^{{-\alpha}}
\| q^* - q_{n}  \|_{L^2\II}\qquad \text{for}~~ n\ge 0.  \end{align*}
Moreover, there holds  
\begin{align*}
\| q^* - q^\dag \|_{L^2\II} +  &\lnorm  U_{\gamma,h}^{\delta,0}(q^*)-v^\dag\rnorm_{L^2\II} \\ 
&\le  c\left(\gamma^{-1}\delta+\gamma^{\frac{s_1}{2}}+\gamma^{-1} h^2+ \tau +\delta \bar H^{-\frac{1}{2}} H^{-\frac{3}{2}}+\bar H^{\frac{3}{2}}H^{-\frac{3}{2}}+ H^{\min\{s_1,s_2\}}\right), 
\end{align*}
where \(q^\dag\) denotes the exact potential, \(v^\dag\) represents the exact initial condition, and \(U_{\gamma,h}^{\delta,0}(q^*) = u_{\gamma,h}^{\delta,0}(q^*) + D_h\). The constant \(c\) is independent of \(\tau\), \(h\), \(H\), \(\bar H\), and \(\delta\).

\end{theorem}
\begin{proof}
    Choosing an arbitrary random initial  $q_0 \in \A$, the contraction mapping theorem and Lemma \ref{lem:Dal-uhn-2} (with   time $T_2\ge  \bar  T_0$) imply that
the iteration \eqref{eqn:iter-fully} generates a Cauchy sequence $\{q_n\}_{n=1}^\infty$ in $L^2\II$ sense. Therefore the sequence $\{ q_n \}$ converges to a fixed point of  $K_{h,\tau}$  as $n\rightarrow\infty$, denoted by
$q^*\in L^2\II$. Then the use of the box restriction $P_{[0,m]}$ indicates  $0 \le q^* \le m$.

Next, we show the error estimate between $q^*$ and $q^\dag$, $u_{\gamma,h}^{\delta,0}(q^*)$ and $\bar  v^\dag$ .  It holds that
\begin{align*}
\| q^\dag - q^* \|_{L^2\II} 
\le& \lnorm\frac{f - \partial_t^\alpha u( T_2;q^\dag, \bar v^\dag) + \Delta g_2}{g_2}  - \frac{f - \partial_t^\alpha u(T_2;q^\dag, \bar v^\dag) + \Delta g_2}{g_2^\delta}\rnorm_{L^2\II}\\
& + \lnorm\frac{f  + \Delta g_2}{g_2^\delta} - \frac{f +\psi_H  }{g_2^\delta}\rnorm_{L^2\II}+ \lnorm \frac{\partial_t^\alpha u(T_2;q^\dag, \bar  v^\dag) }{g_2^\delta} - \frac{\partial_t^\alpha u(T_2;q^\dag, u_\gamma(q^\dag,0)) }{g_2^\delta}\rnorm_{L^2\II}\\
&+
 \lnorm \frac{\partial_t^\alpha u(T_2;q^\dag, u_\gamma(q^\dag,0)) }{g_2^\delta} - \frac{\bar \partial_\tau^\alpha u_h^{N_2}(q^\dag,u_{\gamma,h}^{\delta,0}(q^\dag)) }{g_2^\delta}\rnorm_{L^2\II}\\
 &+ \lnorm \frac{\bar \partial_\tau^\alpha u_h^{N_2}(q^\dag,u_{\gamma,h}^{\delta,0}(q^\dag))}{g_2^\delta}- \frac{\bar \partial_\tau^\alpha u_h^{N_2}(q^*,u_{\gamma,h}^{\delta,0}(q^*)) }{g_2^\delta}\rnorm_{L^2\II}
 =: \sum_{i=1}^5 {\rm J}_i. 
\end{align*}
Due to the fact that 
 $f(x)-\partial_t^\alpha u(t;q^\dag, \bar  v^\dag) +\Delta g_2= q^\dag g_2\in L^\infty\II$, it is straightforward to see that the first term is bounded by
$$ {\rm J_1}\le c\|g_2^\delta-g_2\|_{L^2\II}\le c \delta .$$
Meanwhile, Lemma~\ref{lem:reg-err} implies the second term bounded by 
$${\rm J_2} \le  c \Big(\delta \bar H^{-\frac{1}{2}} H^{-\frac{3}{2}}+\bar H^{\frac{3}{2}}H^{-\frac{3}{2}}+ H^{\min\{s_1,s_2\}}\Big).$$
Moreover, it gives from Remark~\ref{rk:41}, Lemma~\ref{lem:stab-0} and Lemma~\ref{lemma:gammacon}  that 
 \begin{align*}
 {\rm J_3}&\le M^{-1} \tilde c_0T_2^{-\alpha}\| \bar  v^\dag-u_\gamma(q^\dag,0)\|_{L^2\II}\le c\gamma^{\frac {s_1}{2}}.
\end{align*}
 Remark~\ref{rk:41}, Lemma~\ref{lem:uh-err} and Lemma~\ref{lem:daluhn-err} indicate 
\begin{align*}
    {\rm J_4}\le  &\,M^{-1}\lnorm \Dal \big(u(T_2;q^\dag, u_\gamma(q^\dag,0))-u_h(T_2;q^\dag, u_\gamma(q^\dag,0))\big)\rnorm_{L^2\II} \\&+M^{-1}\lnorm \Dal u_h(T_2;q^\dag, u_\gamma(q^\dag,0))- \bDal u_h^{N_2}(q^\dag,u_{\gamma,h}^{\delta,0}(q^\dag)) \rnorm_{L^2\II}\\
    \le &\,c (\gamma^{-1} h^2+ \tau + \gamma^{-1}\delta).
\end{align*}
Applying Lemma~\ref{lem:Dal-uhn-2} to $  {\rm J_5}$ gives 
\begin{align*}
    {\rm J_5} \le M^{-1}  \bar c _{0}T_2^{-\al}( \bar c _{1}(1+T_1^\al)+1)\| q^\dag - q^* \|_{L^2\II}.
\end{align*}
Therefore, if $T_2\ge  \bar  T_0$,    it holds that 
\begin{align*}
    \| q^\dag - q^* \|_{L^2\II}\le c\left(\gamma^{-1}\delta+\gamma^{\frac{s_1}{2}}+\gamma^{-1} h^2+ \tau +\delta \bar H^{-\frac{1}{2}} H^{-\frac{3}{2}}+\bar H^{\frac{3}{2}}H^{-\frac{3}{2}}+ H^{\min\{s_1,s_2\}}\right).
\end{align*}
Next we turn to the estimate of $u_{\gamma,h}^{\delta,0}(q^*)-\bar  v^\dag$. Using Lemma~\ref{lemma:gammacon},  Lemma~\ref{lem:esitugh}, Lemma~\ref{lem:uhn-err} and Lemma~\ref{lem:stab-0_reg_nois}  leads to
\begin{align*}
   \lnorm u_{\gamma,h}^{\delta,0}(q^*)-\bar  v^\dag\rnorm_{L^2\II}&\le \lnorm u_{\gamma,h}^{\delta,0}(q^*)-u_{\gamma,h}^{\delta,0}(q^\dag)\rnorm_{L^2\II}+\lnorm u_{\gamma,h}^{\delta,0}(q^\dag)-\bar v^\dag\rnorm_{L^2\II}\\
  & \le\tilde c_{1,\tau}(1+T_1^{\al})\lnorm q^*-q^\dag\rnorm_{L^2\II}+ c(\gamma^{-1}\delta+\gamma^{-1} h^2+\tau+\gamma^{\frac{s_1}{2}})\\
  &\le c\left(\gamma^{-1}\delta+\gamma^{\frac{s_1}{2}}+\gamma^{-1} h^2+ \tau +\delta \bar H^{-\frac{1}{2}} H^{-\frac{3}{2}}+\bar H^{\frac{3}{2}}H^{-\frac{3}{2}}+ H^{\min\{s_1,s_2\}}\right).
\end{align*}
 We then complete the proof with the relations that $U_{\gamma,h}^{\delta,0}(q^*)=u_{\gamma,h}^{\delta,0}(q^*)+D_h$ and $ v^\dag=\bar  v^\dag+D$.
\end{proof} 

\begin{remark}\label{rem:iter}
The contraction property of \(K_{h,\tau}\), established in Theorem~\ref{thm:err-fully}, naturally motivates the development of an iterative algorithm to compute \(q^*\) and 
\(U_{\gamma,h}^{\delta,0}(q^*) = u_{\gamma,h}^{\delta,0}(q^*) + D_h\) as defined in \eqref{eqn:Kh}. 
Each iteration requires solving a linear backward problem, which can be efficiently handled using the conjugate gradient method 
\cite{zhang2020numerical,zhang2023stability}. 
The complete procedure is summarized in Algorithm~\ref{alg}. 
The contraction property from Theorem~\ref{thm:err-fully} guarantees the linear convergence of the iterative scheme.

\end{remark}

\begin{algorithm}
\SetKwInOut{Input}{input}\SetKwInOut{Output}{output}
\caption{Recover the potential $q^*$  and  initial condition $U_{\gamma,h}^{\delta,0}(q^*)$ from $g_1^\delta$ and $g_2^\delta$. }\label{alg}
 \KwData{Order $\alpha$, terminal time $T_1$ and $T_2$, source term $f$, boundary data $b$,
 noisy observation $g_1^\delta$, $g_2^\delta$ upper bound constant $m$, parameters \(\gamma\), \(h\), \(H\), \(\bar H\) and \(\tau\);}
\KwResult{ Approximate  potential $q^*$ and initial condition $ U_{\gamma,h}^{\delta,0}(q^*)$.}
Compute $\psi_H$ by \eqref{eqn:psih} and  initialize $0\le q_{0}\le m$ for random $q_{0}$, set $e^0=1$, $k=0$\;
\While{$e^k>\text{tol}=10^{-10}$}{Compute  $u_{\gamma,h}^{\delta,0}(q_k)$ from \eqref{solu:ugamdel_fully} using the Conjugate Gradient method\; Set $u_{\gamma,h}^{\delta,0}(q_k)=\max(u_{\gamma,h}^{\delta,0}(q_k),0)$ and Compute  $u_{\gamma,h}^{\delta,n}(q_k)$ from  \eqref{eqn:udisdeelayh} with potential $q_k$\;
Update the potential by
$$
q_{k+1}= K_{h,\tau} q_k :=  P_{[0,m]}\left(\frac{f-\bar \partial_\tau^\alpha u_h^{N_2}(q_k,u_{\gamma,h}^{\delta,0}(q_k)) + \psi_H  }{g_2^\delta} \right);$$\\
Compute  error
$$
e^{k+1} = \|q_{k+1} - q_k\|_{L^2\II};
$$\\
$k\leftarrow k+1$\;
}
$q^*\leftarrow q_k$, $ U_{\gamma,h}^{\delta,0}(q^*)\leftarrow  u_{\gamma,h}^{\delta,0}(q_k)+D_h$\;
\Output{The approximated potential $q^*$ and the initial condition $ U_{\gamma,h}^{\delta,0}(q^*)$.}
\end{algorithm}


\begin{remark}\label{pra_rmk}
The error estimate presented in Theorem \ref{thm:err-fully} offers valuable guidance for selecting the regularization and discretization parameters \(\gamma\), \(h\), \(H\), \(\bar H\) and \(\tau\) based on the known noise level \(\delta\). For instance, the choice that 
\[
\gamma \sim \delta^\frac{2}{2+s_1}, \quad h \sim \delta^{\frac{1}{2}},  \quad H \sim \delta^\frac{3}{6+4\min\{s_1,s_2\}}, \quad \bar{H} \sim \delta^{\frac{1}{2}}, \quad \tau \sim \delta^\frac{s_1}{2+s_1}
\]
results in the optimal convergence rate of \(O(\delta^\frac{3\min\{s_1,s_2\}}{6+4\min\{s_1,s_2\}})\). This finding is well-supported by our numerical results shown in Figure \ref{rateP_1_ritz} of Section \ref{sec:num}.
\end{remark}


\vskip20pt
\section{Numerical Experiments}\label{sec:num}
In this section, we present some two-dimensional numerical tests to illustrate our theoretical analysis. We consider the two-dimensional subdiffusion model  \eqref{eqn:pdeo} with the square domain \(\Omega = (0,3)^2\). 
To compute the exact solutions \(U(T_1)\) and \(U(T_2)\) as reference data, we solve the direct problem on very fine meshes. 
Noisy data \(g_1^\delta\) and \(g_2^\delta\) are then generated as:
\begin{align*}
   g_1^\delta(x) = U(x,T_1) + \delta \zeta(x)\sup_{x\in\Omega} U(x,T_1), \quad g_2^\delta(x) = U(x,T_2) + \delta \zeta(x)\sup_{x\in\Omega} U(x,T_2),
\end{align*}
where \(\zeta\) is generated from  the standard Gaussian distribution, and $\delta$ represents the associated noise level. Then we employ Algorithm \ref{alg} to compute the numerical reconstructions of the potential \(q^*\) and the initial condition \(U_{\gamma,h}^{\delta,0}(q^*)\). All simulationas are performed on a desktop computer using MATLAB 2023.
The quasi-uniform meshes using the optimal delaunay triangulations \cite{ChenXu2004} are generated using the IFEM package \cite{Chen:2008ifem}.

In the experiments, the source and boundary data are chosen as
\begin{equation*}
     f(x,y) = 10, \quad b(x,y) = \frac{x(3-x)}{4} + 1.
\end{equation*}
We consider following three exact potential functions:
\begin{itemize}
    \item[(i)] Smooth potential in $H^2(\Omega)$:
    \[
    q_1^\dagger(x, y) = 3 - \cos(\pi x) \cos(\pi y).
    \]
    \item[(ii)] Piecewise smooth potential in $H^{\frac32-\epsilon}(\Omega)$, defined as a pyramid-shaped function:
    \[
    q_2^\dagger(x, y) = 3 + 1.5 \times (-1)^{j+k} \psi(x - j, y - k), \quad (x, y) \in [j, j+1] \times [k, k+1], \quad j, k = 0, 1, 2,
    \]
    where \(\psi(x, y)\) is defined by:
    \[
    \psi(x, y) = \begin{cases}
      2y, & x \ge y, \, x + y > 1, \, y < 0.5, \\
      2x, & x < y, \, x + y \leq 1, \, x < 0.5, \\
      2(1 - y), & x < y, \, x + y > 1, \, y > 0.5, \\
      2(1 - x), & x \ge y, \, x + y \ge 1, \, x \ge 0.5.
    \end{cases}
    \]
    \item[(iii)] Discontinuous potential in $H^{\frac12-\epsilon}(\Omega)$, defined as a step function where
    \[
    q_3^\dagger(x, y) = 3 + (-1)^{j+k}, \quad (x, y) \in [j, j+1] \times [k, k+1], \quad j, k = 0, 1, 2.
    \]
\end{itemize}

We also consider the following three initial conditions:  
\begin{itemize}
    \item[(i)] Smooth initial condition in $H^2(\Omega)$:
    \begin{equation*}
        v_1^\dag(x,y) = x(3-x)\left(\frac{1}{4} + \frac{y(3-y)}{2}\right)  + \exp\left(\frac{q_1^\dag(x,y)}{4}\right)\sin(\pi x)^2 \sin(\pi y)^2+ 1.
    \end{equation*}
    \item[(ii)] Piecewise smooth initial condition in $H^{\frac32-\epsilon}(\Omega)$:
    \[
    v_2^\dag(x,y) = x(3-x)\left(\frac{1}{4} + \frac{y(3-y)}{2}\right) + q_2^\dag(x,y) - 2.
    \]
    \item[(ii)] Discontinuous initial condition in $H^{\frac12-\epsilon}(\Omega)$: 
    \[
     v_2^\dag(x,y) = x(3-x)\left(\frac{1}{4} + \frac{y(3-y)}{2}\right) + q_3^\dag(x,y).
    \]
\end{itemize}

We shall test the numerical  recovery of the pairs of initial condition  and   potential  \((v_i^\dag, q_j^\dag)\) for \(i, j = 1, 2, 3\), with all choices satisfying Assumption \ref{assump:numerics}.
Given the known noise level \( \delta \), we select the discretization parameters based on Remark \ref{pra_rmk} with default values:
for $(v_1^\dag,q_1^\dag)$
\begin{align*}
 \gamma = \delta^{\frac{1}{2}}/7.25, \quad h =\frac{3}{\lceil1.25\delta^{-\frac12}\rceil}, \quad H = \frac{3}{\lceil2.85\delta^{-\frac{3}{14}}\rceil}, \quad \bar H=\frac{3}{\lceil2.85\delta^{-\frac{3}{14}}\rceil\rceil\lceil\delta^{-\frac27}\rceil}, \quad \tau = \frac{T_2}{\lceil15\delta^{-\frac12}T_2\rceil};
\end{align*}
for $(v_1^\dag, q_2^\dag)$, 
\begin{align*}
 \gamma = \delta^{\frac{1}{2}}/5.75,  \quad h =\frac{3}{\lceil1.25\delta^{-\frac12}\rceil}, \quad H = \frac{3}{\lceil2\delta^{-\frac14}\rceil}, \quad \bar H=\frac{3}{\lceil2\delta^{-\frac14}\rceil\lceil1.25\delta^{-\frac14}\rceil}, \quad \tau = \frac{T_2}{\lceil15\delta^{-\frac12}T_2\rceil};
\end{align*}
 for $(v_2^\dag, q_1^\dag)$, $(v_2^\dag, q_2^\dag)$ 
\begin{align*}
 \gamma = \delta^{\frac{4}{7}}/1.85, \quad h =\frac{3}{\lceil1.25\delta^{-\frac12}\rceil}, \quad H = \frac{3}{\lceil2\delta^{-\frac14}\rceil}, \quad \bar H=\frac{3}{\lceil2\delta^{-\frac14}\rceil\lceil1.25\delta^{-\frac14}\rceil}, \quad \tau = \frac{T_2}{\lceil15\delta^{-\frac37}T_2\rceil};
\end{align*}
 for $(v_2^\dag,q_3^\dag)$
\begin{align*}
 \gamma = \delta^{\frac{4}{7}}/1.85,  \quad h =\frac{3}{\lceil1.15\delta^{-\frac12}\rceil}, \quad H = \frac{3}{\lceil1.65\delta^{-\frac38}\rceil}, \quad \bar H=\frac{3}{\lceil1.65\delta^{-\frac38}\rceil\rceil\lceil1.75\delta^{-\frac18}\rceil},  \quad \tau = \frac{T_2}{\lceil10\delta^{-\frac37}T_2\rceil};
\end{align*}
and for $(v_3^\dag,q_3^\dag)$
\begin{align*}
 \gamma = \delta^{\frac{4}{5}}/.95, \quad h =\frac{3}{\lceil1.15\delta^{-\frac12}\rceil}, \quad H = \frac{3}{\lceil1.65\delta^{-\frac38}\rceil}, \quad \bar H=\frac{3}{\lceil1.65\delta^{-\frac38}\rceil\rceil\lceil1.75\delta^{-\frac18}\rceil}, \quad \tau = \frac{T_2}{\lceil17.5\delta^{-\frac15}T_2\rceil}.
\end{align*}

We first examine the convergence rate of the numerical reconstruction. Figure~\ref{rateP_1_ritz} shows the relative errors \( e_v \) and \( e_q \) against \( \delta \) for terminal times \( T_1 = 0.1 \), \( T_2 = 1 \), and different values of \( \alpha \). Here, \( e_v \) and \( e_q \) are defined as:
\begin{equation*}\label{eqn:eqv}
e_v = \| U_{\gamma,h}^{\delta,0}(q^*) - v^\dag \|_{L^2\II}/\|v^\dag\|_{L^2\II}, \quad e_q = \|q^* - q^\dag \|_{L^2\II}/\|q^\dag\|_{L^2\II}.
\end{equation*}
The numerical results show that, for all pairs \( (v_i^\dagger, q_j^\dagger) \) with \( i,j = 1,2,3 \), the convergence rates are \(O(\delta^\frac{3\min\{s_1,s_2\}}{6+4\min\{s_1,s_2\}})\), which is consistent with the theoretical finding in Theorem~\ref{thm:err-fully}.
\begin{figure}[H]
\vspace{0.35in}
		\begin{tabular}{ccc}
		\centering
\includegraphics[trim={0.63in 0.15in 0.5in 0.5in},clip,width=0.25\textwidth]{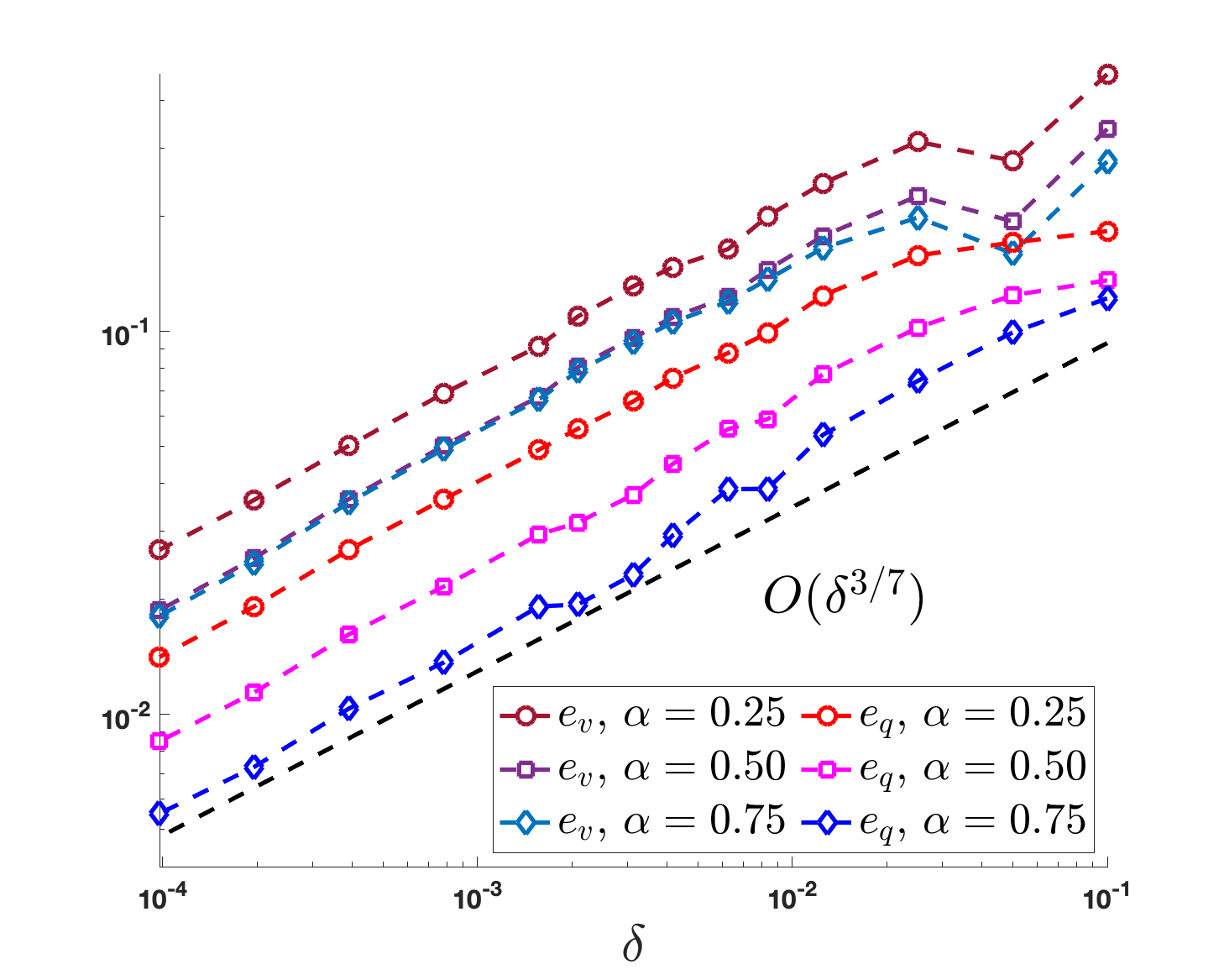}&
	\includegraphics[trim={0.63in 0.15in 0.5in 0.5in},clip,width=0.25\textwidth]{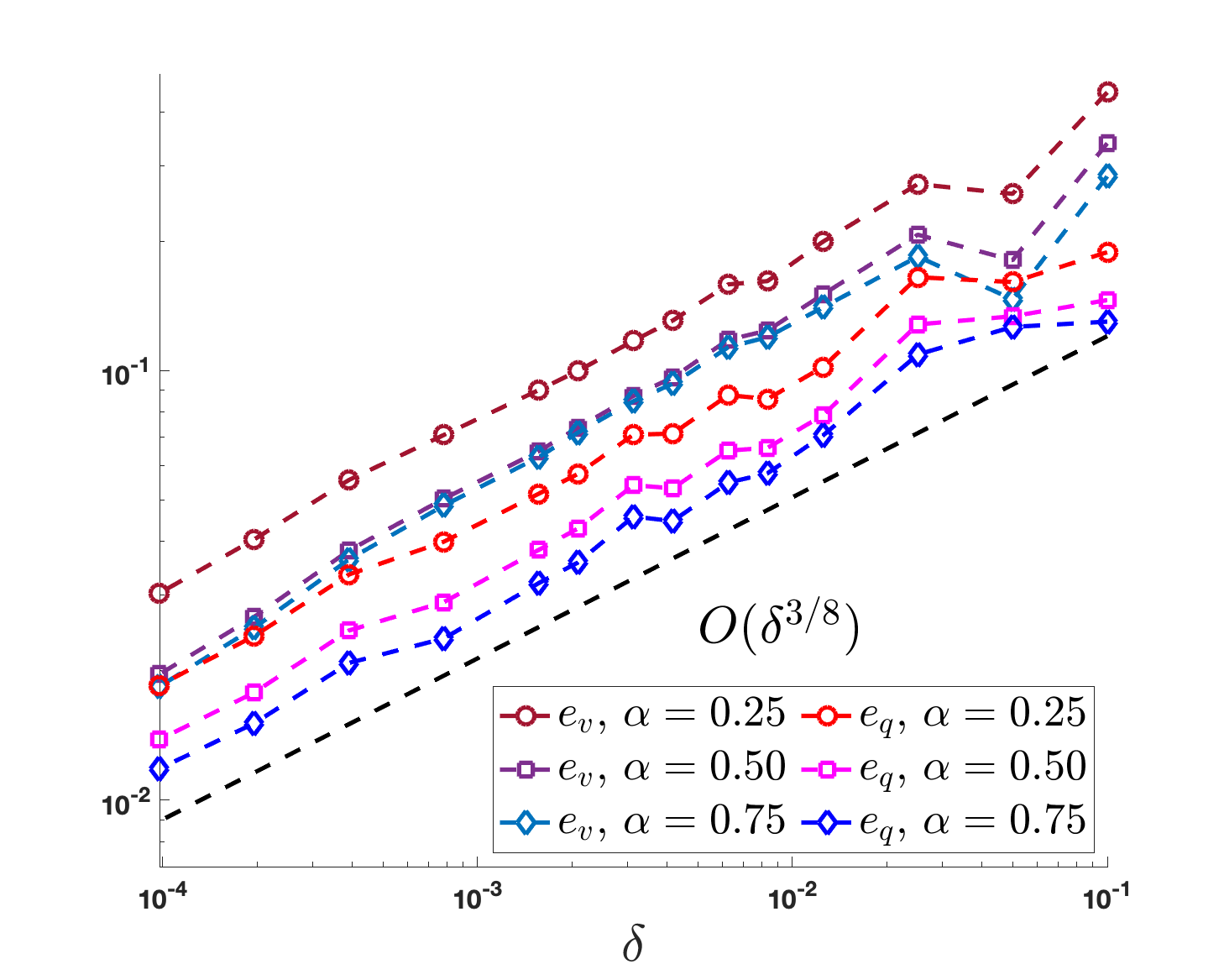}&
\includegraphics[trim={0.63in 0.15in 0.5in 0.5in},clip,width=0.25\textwidth]{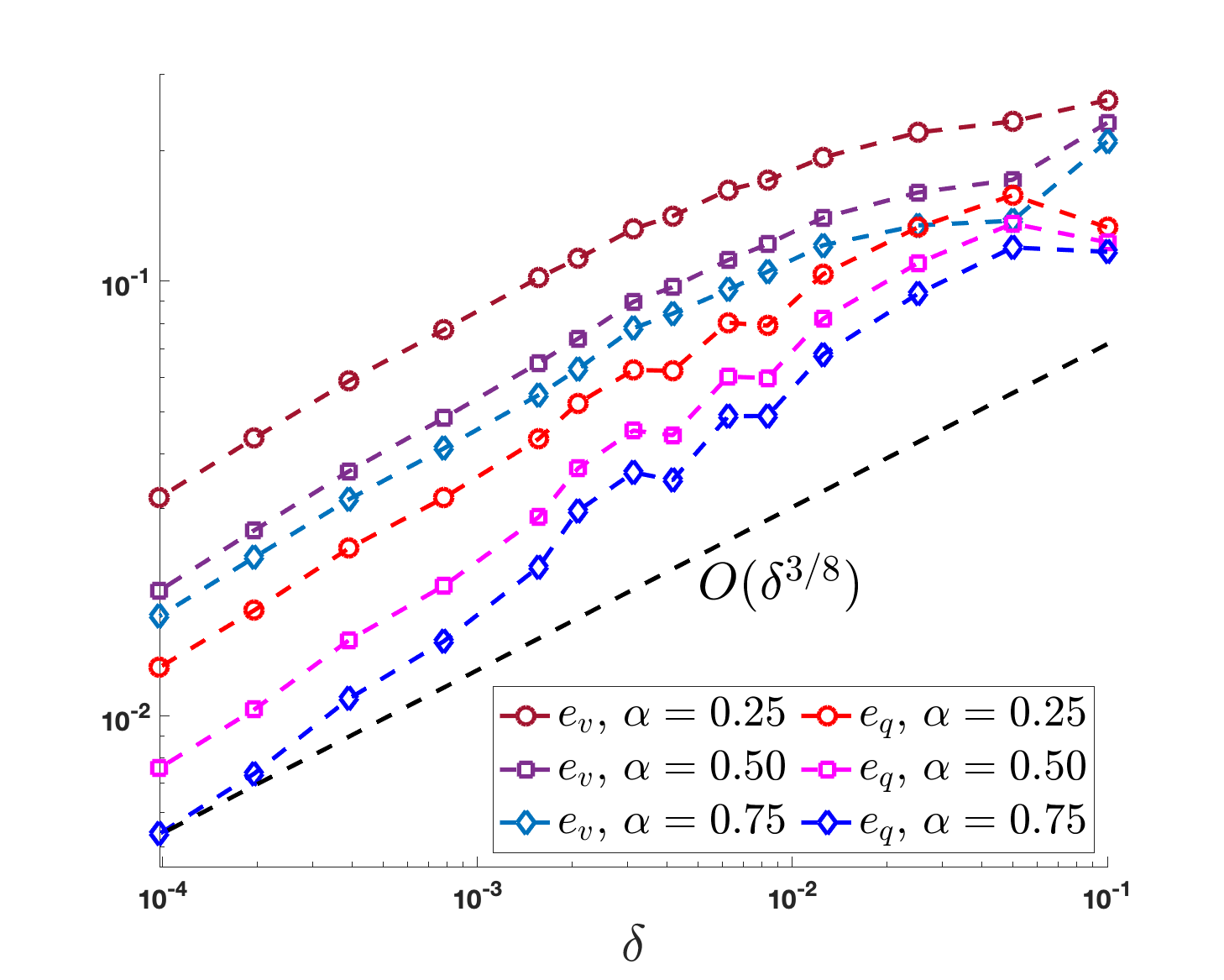}\\
  Rate for  $(v_1^\dag,q_1^\dag)$ &    Rate for  $(v_1^\dag,q_2^\dag)$ &
   Rate for  $(v_2^\dag,q_1^\dag)$\\
   \includegraphics[trim={0.63in 0.15in 0.5in 0.5in},clip,width=0.25\textwidth]{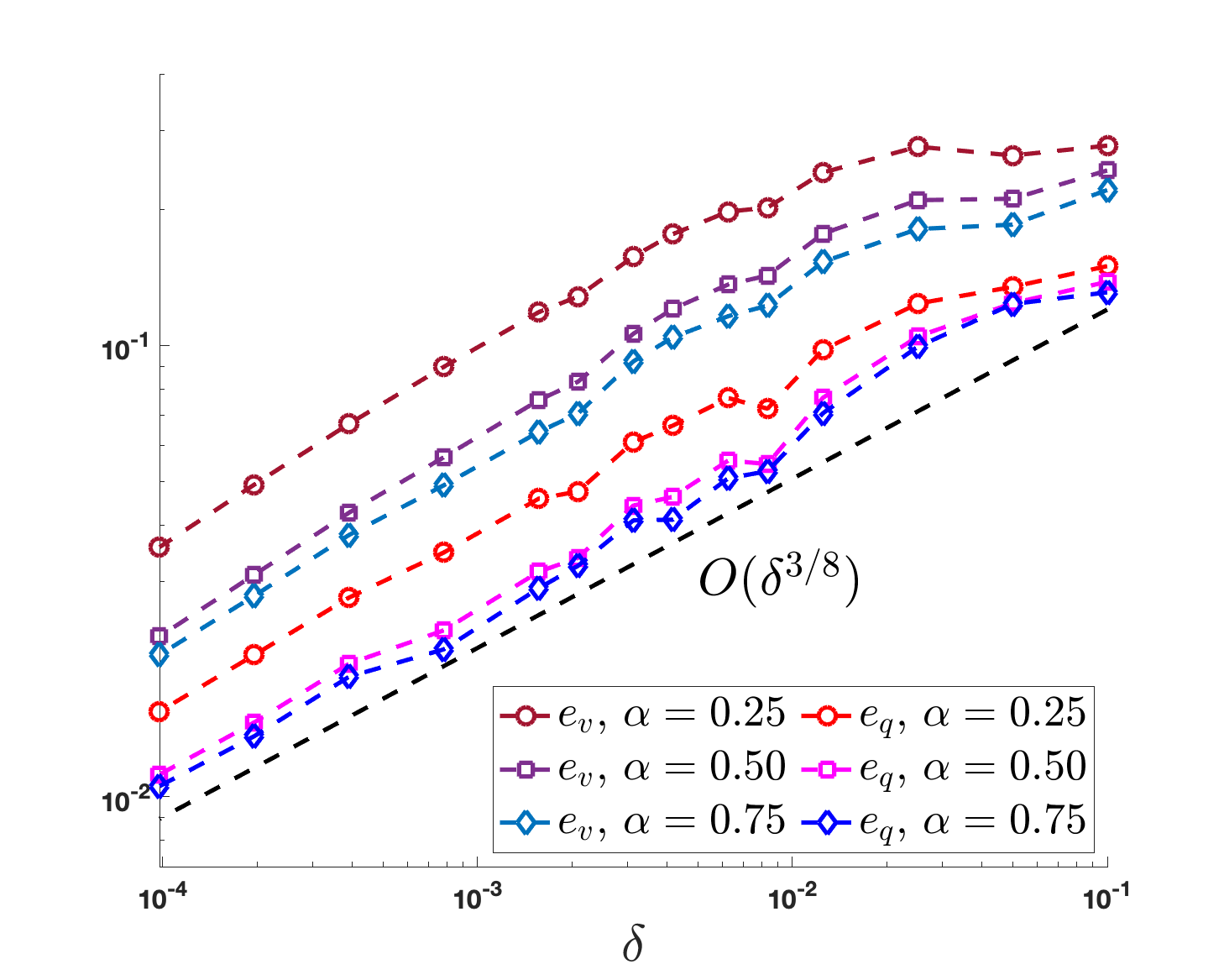}&
	\includegraphics[trim={0.63in 0.15in 0.5in 0.5in},clip,width=0.25\textwidth]{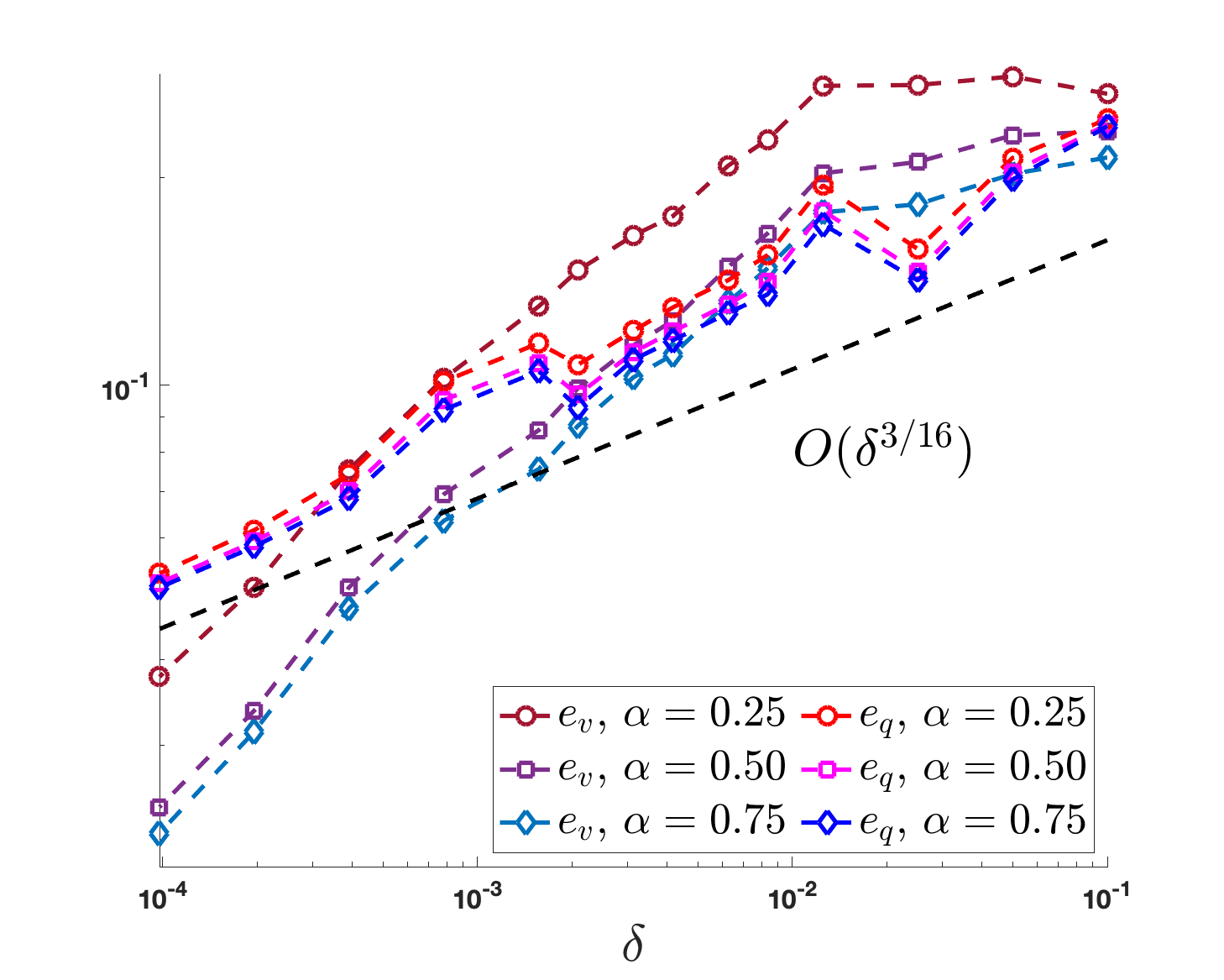}&
\includegraphics[trim={0.63in 0.15in 0.5in 0.5in},clip,width=0.25\textwidth]{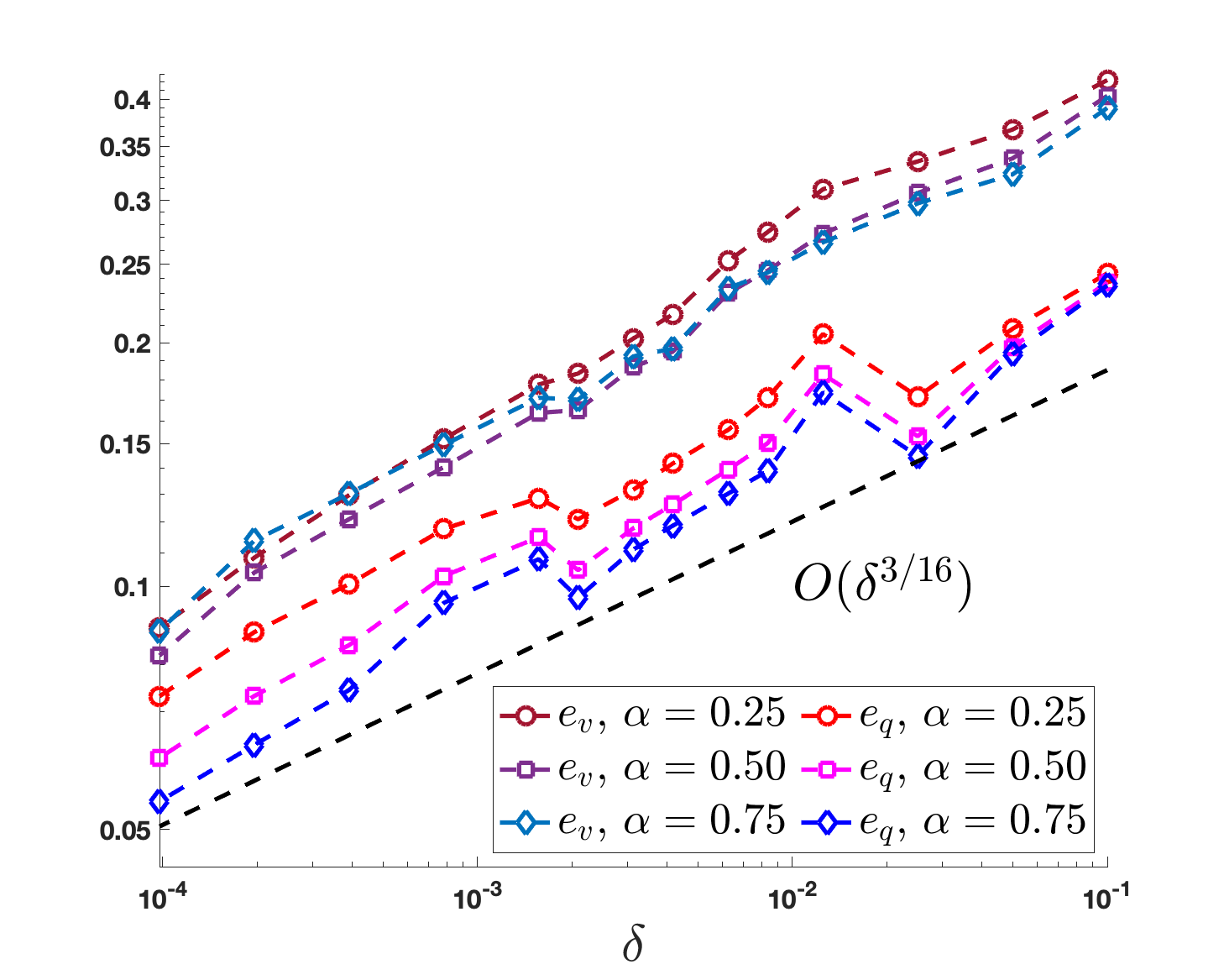}\\
  Rate for  $(v_2^\dag,q_2^\dag)$ &    Rate for  $(v_2^\dag,q_3^\dag)$ &
   Rate for  $(v_3^\dag,q_3^\dag)$
	\end{tabular}
\caption{Convergence rates of \( e_v \) and \( e_q \) for recovering \( (v_i^\dag, q_j^\dag) \) against \( \delta \).}
\label{rateP_1_ritz}
\end{figure}

Next, we replace $P_{X_{\bar H}}$ in equation \eqref{eqn:psih} with Lagrange interpolation, following the approach outlined in \cite[eq. (4.21)]{zhang2022identification} to approximate $\Delta g_2$. When the noise is uniformly bounded, satisfying  
$\|g_2^\delta - g_2\|_{C(\overline{\Omega})} = \delta$,  
the optimal recovery rate for the potential $q$ is $O(\delta^{1/3})$, as demonstrated in Figure \ref{rateQP1_Conti}(a).  
However, for non-uniform meshes, no clear convergence rate can be observed, as illustrated in Figure \ref{rateQP1_Conti}(b). This lack of convergence arises because the optimal error estimate in this scheme relies heavily on the superconvergence property of P1 element interpolation (see \cite[Lemma 4.10]{zhang2022identification}), which holds only for uniform meshes.  
These results highlight the significant improvement achieved by the current method.

 \begin{figure}[H]
\vspace{0.35in}
		\begin{tabular}{ccc}
		\centering
  \includegraphics[trim={0.5in 0.15in 0.6in 0.35in},clip,width=0.3\textwidth]{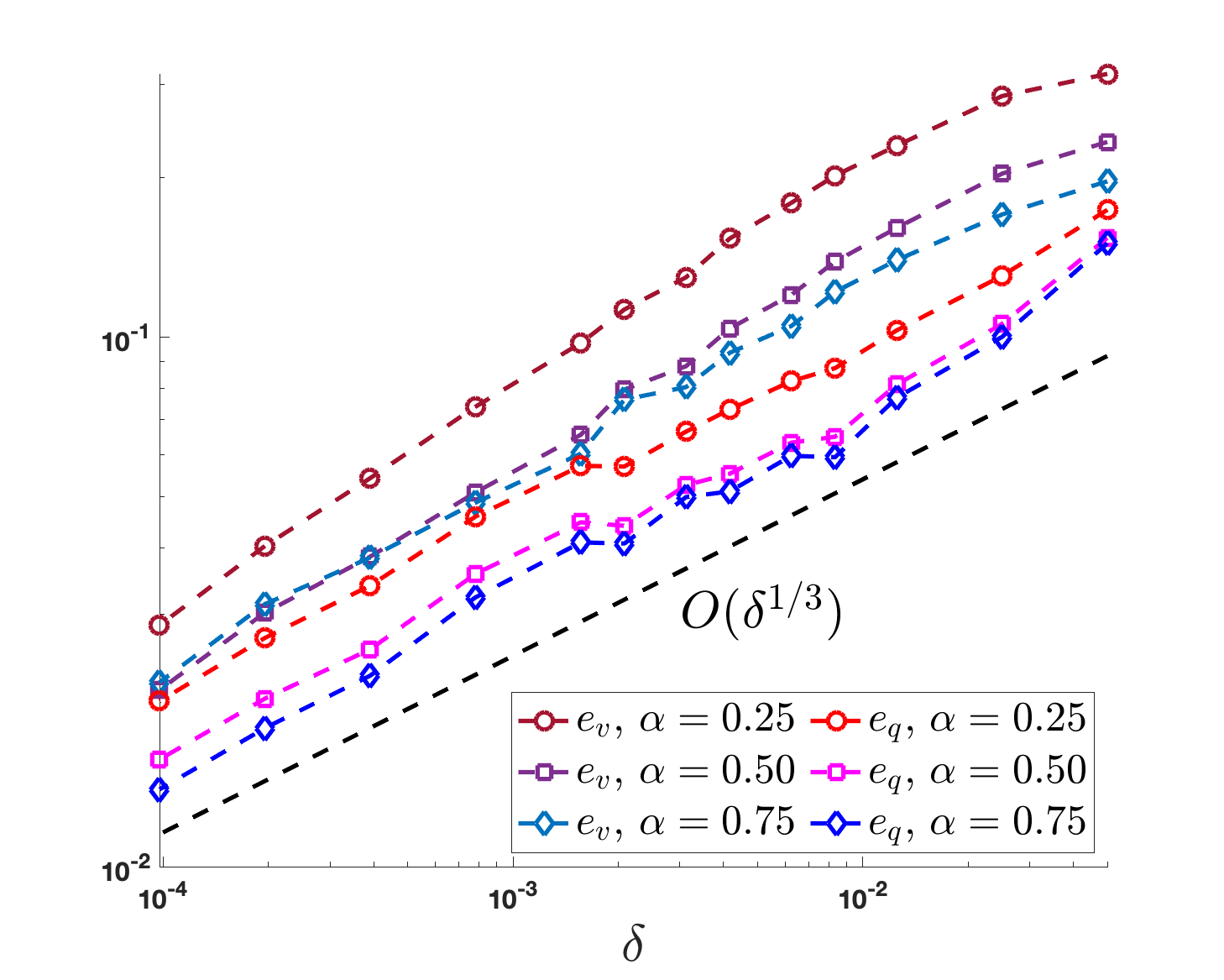}&
	\includegraphics[trim={0.5in 0.15in 0.6in 0.35in},clip,width=0.3\textwidth]
    {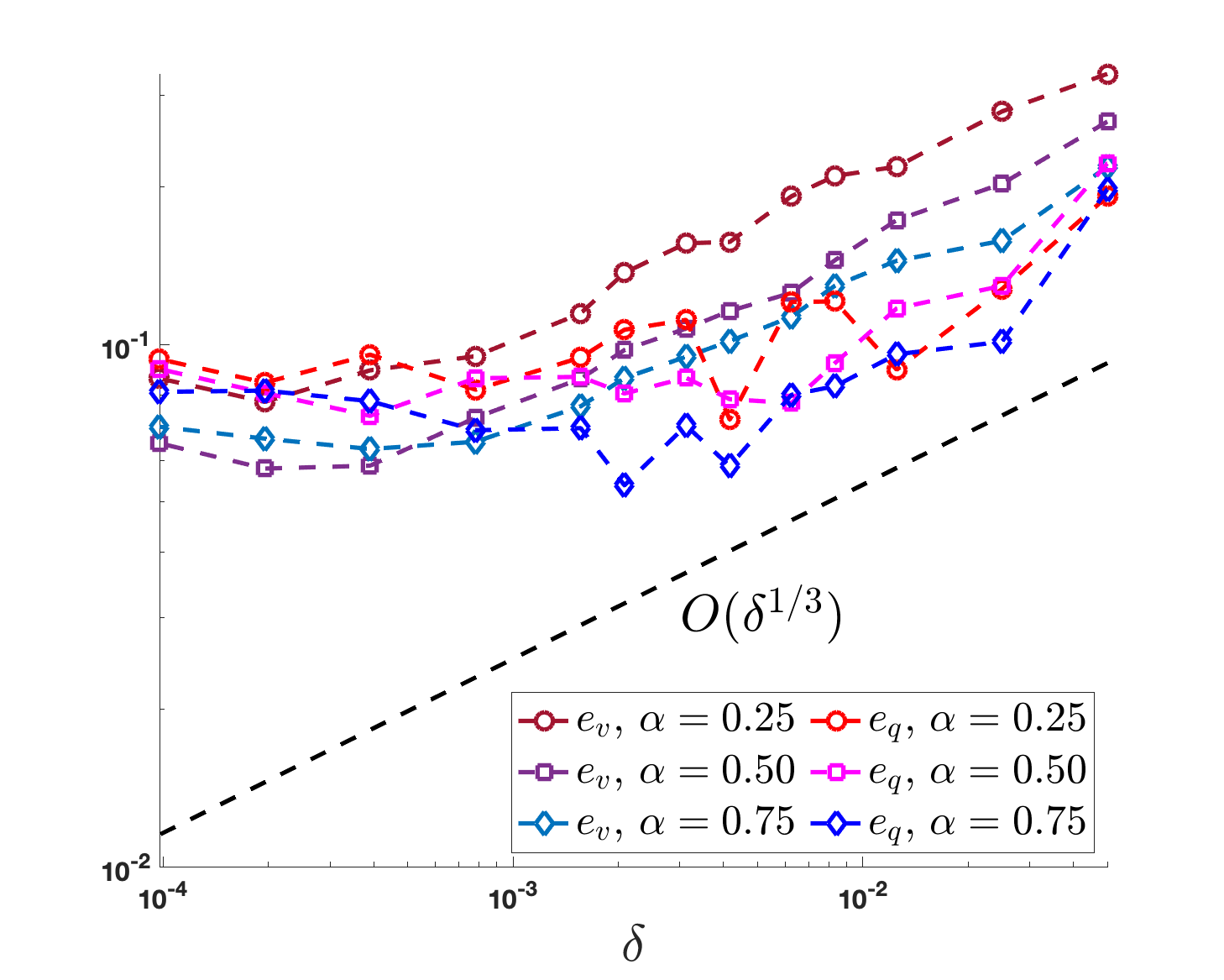}\\
 (a)  uniform meshes &  (b)   non-uniform meshes
	\end{tabular}
\caption{Convergence rates for recovering \( (v_2^\dag, q_2^\dag) \) using the scheme   \cite[(4.21)]{zhang2022identification}.}\label{rateQP1_Conti} 
\end{figure}

In addition, we explore the convergence of the iterative scheme and the effect of parameters \( T_1 \) and \( T_2 \). Let \( q_k \) represent the solution after the \( k \)-th iteration of Algorithm~\ref{alg}, with the relative error defined by:
\[
e_k = \|q_k - q^\dag\|_{L^2\II}/\|q^\dag\|_{L^2\II}\qquad \text{for all} \quad k \ge 0.
\]
Figures~\ref{con_iter} (a) and (b) display convergence histories for fixed \( T_1 = 0.1 \) with varying \( \alpha \) or \( T_2 \) values, demonstrating linear convergence with an increased rate as \( T_2 \) grows. Furthermore, Figure~\ref{con_iter} (c) shows convergence behavior across different \( T_1 \) and \( T_2 \) values with \( \alpha = 0.5 \), indicating that smaller \( T_1 \) values necessitate proportionally smaller \( T_2 \) for effective iterative convergence.


Finally, we examine the profile for both small and large values of \( T_2 \) with a fixed \( T_1 \). In Figure~\ref{Profile:diver}, we display the numerical reconstructions for \( T_2 = 0.01 \) and \( T_2 = 1 \), with \( T_1 = 0.1,\ \al=0.5,\  \tau =0.0025\). In this case, we set the tolerance \( \text{tol} = 10^{-8} \). The numerical results suggest that Algorithm~\ref{alg} provides a highly accurate reconstruction when \( T_2 = 1 \), but the reconstruction becomes inaccurate for smaller values of \( T_2 \). This finding underscores the importance of the assumption in Theorems~\ref{thm:cond-stab} and \ref{thm:err-fully} that the terminal time \( T_2 \) should be sufficiently large.

\begin{figure}
\vspace{0.35in}
 \centering
	\begin{tabular}{ccc}
	\includegraphics[trim={.5in 0.5in 0.85in 0.5in},clip,width=0.25\textwidth]{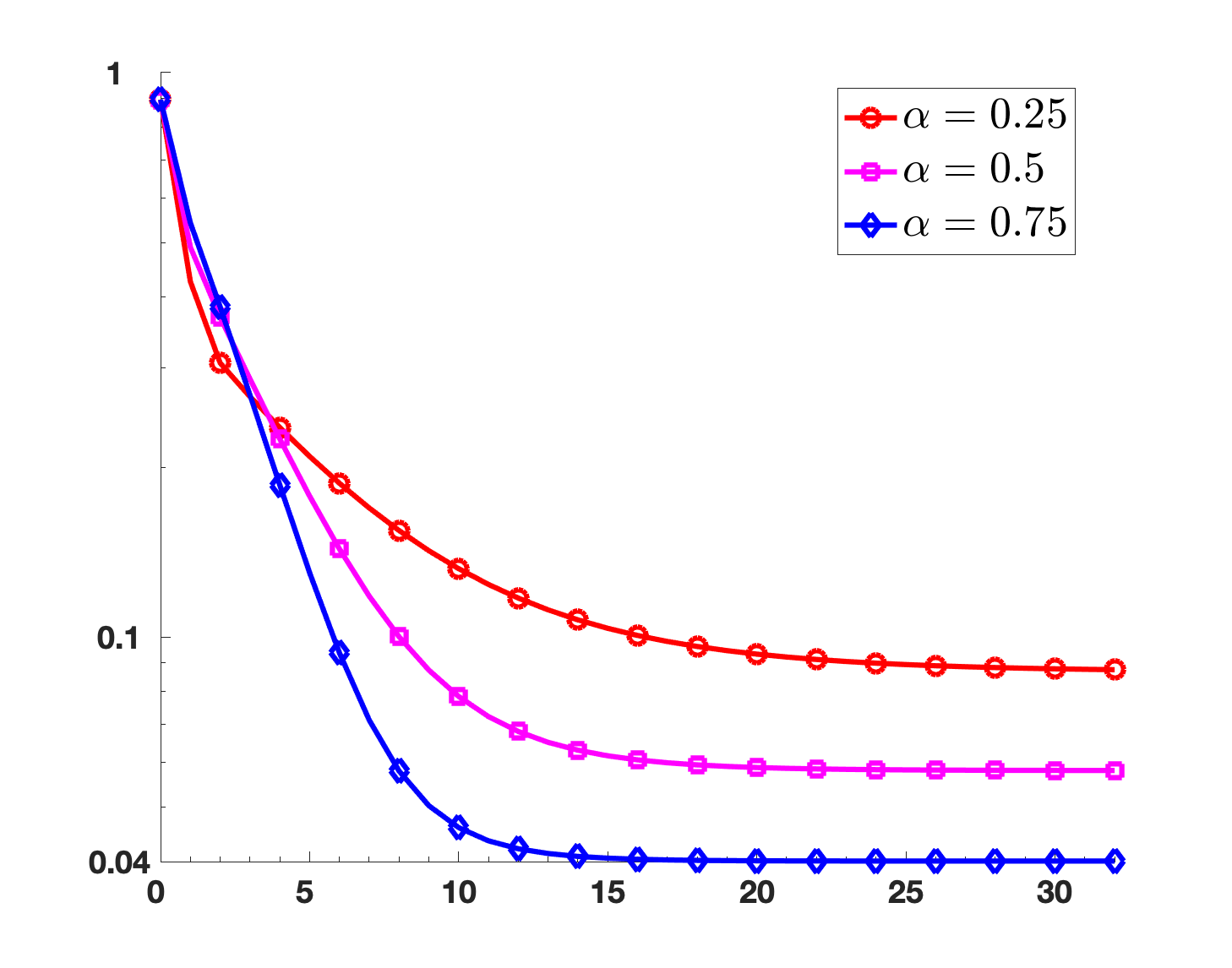}&
	\includegraphics[trim={.5in 0.5in 0.85in 0.5in},clip,width=0.25\textwidth]{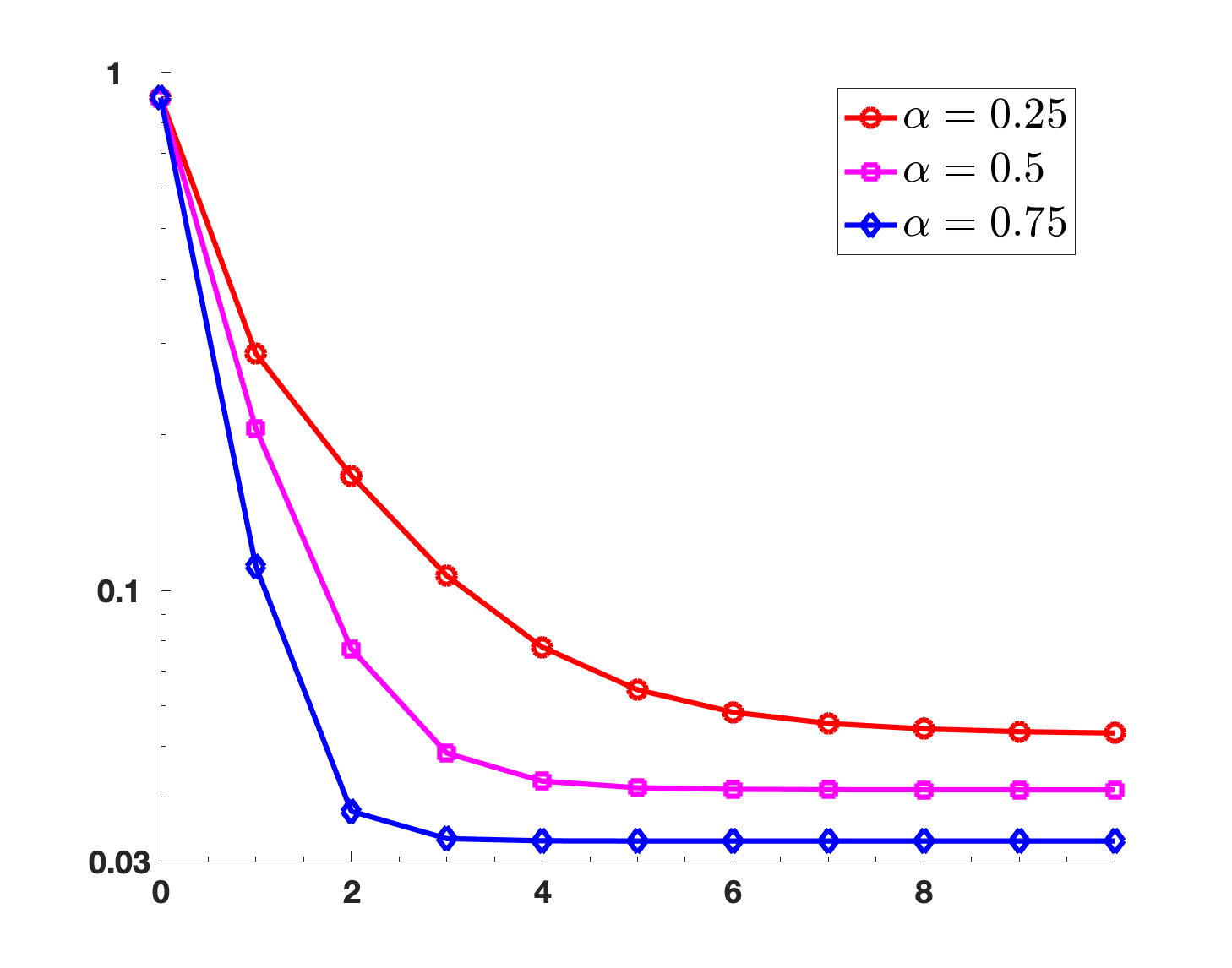}&
\includegraphics[trim={.5in 0.5in 0.85in 0.5in},clip,width=0.25\textwidth]{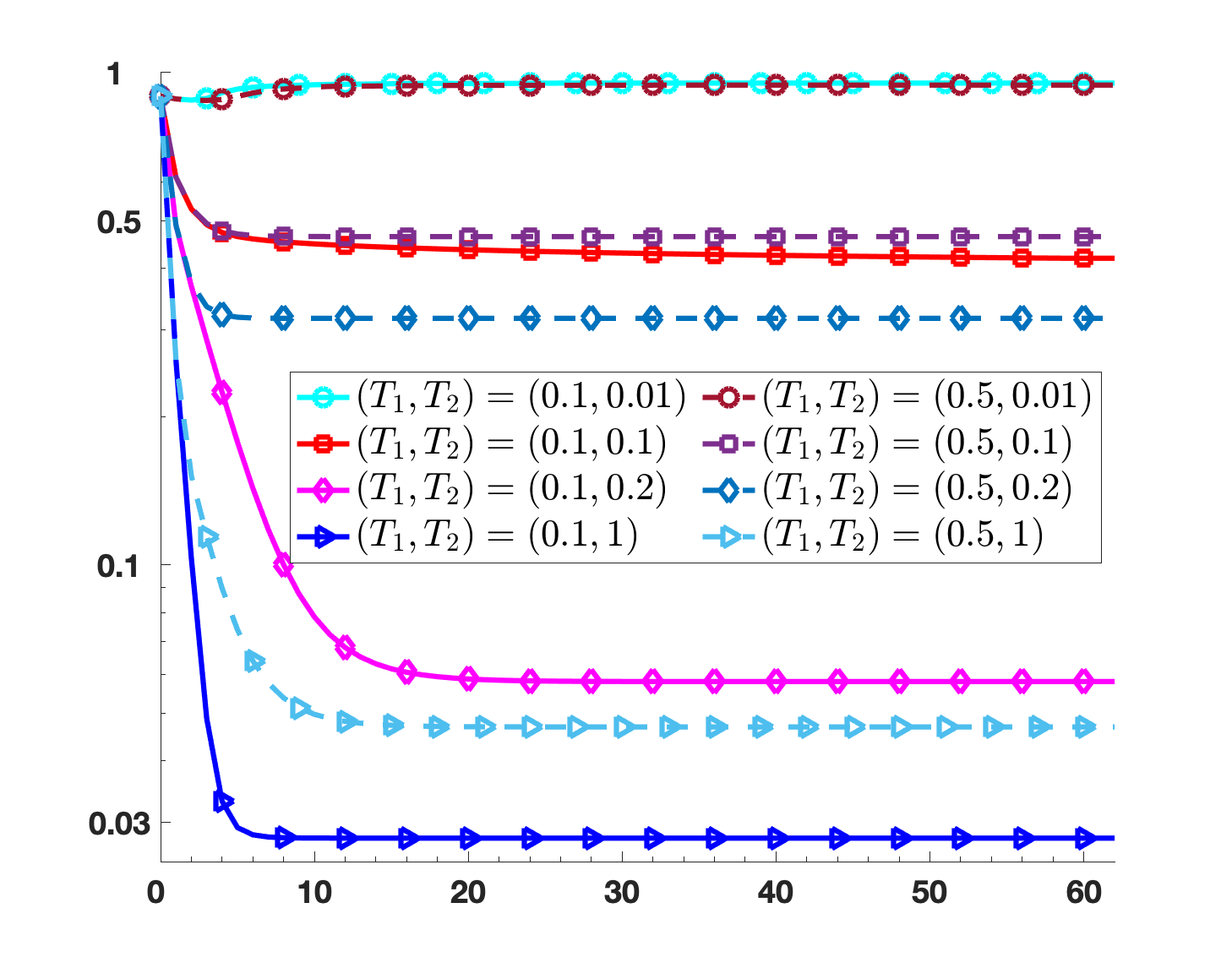}\\
		 (a) $T_1=0.1$, $T_2=0.2$&  (b)  $T_1=0.1$, $T_2=2$ &  (c) $\al=0.5$
	\end{tabular}
\caption{Convergence histories of Algorithm~\ref{alg}  to recover $( v^\dag_2,q^\dag_2)$ for  with different $\alpha$ or different $T_1$, $T_2$, where $\delta=10^{-3}$,     $\tau=0.005$.  } \label{con_iter}
\end{figure}

\begin{figure}
\vspace{0.35in}
 \centering
	\begin{tabular}{ccc}
    \includegraphics[trim={.1in 0.1in 0.6in 0.2in},clip,width=0.25\textwidth]{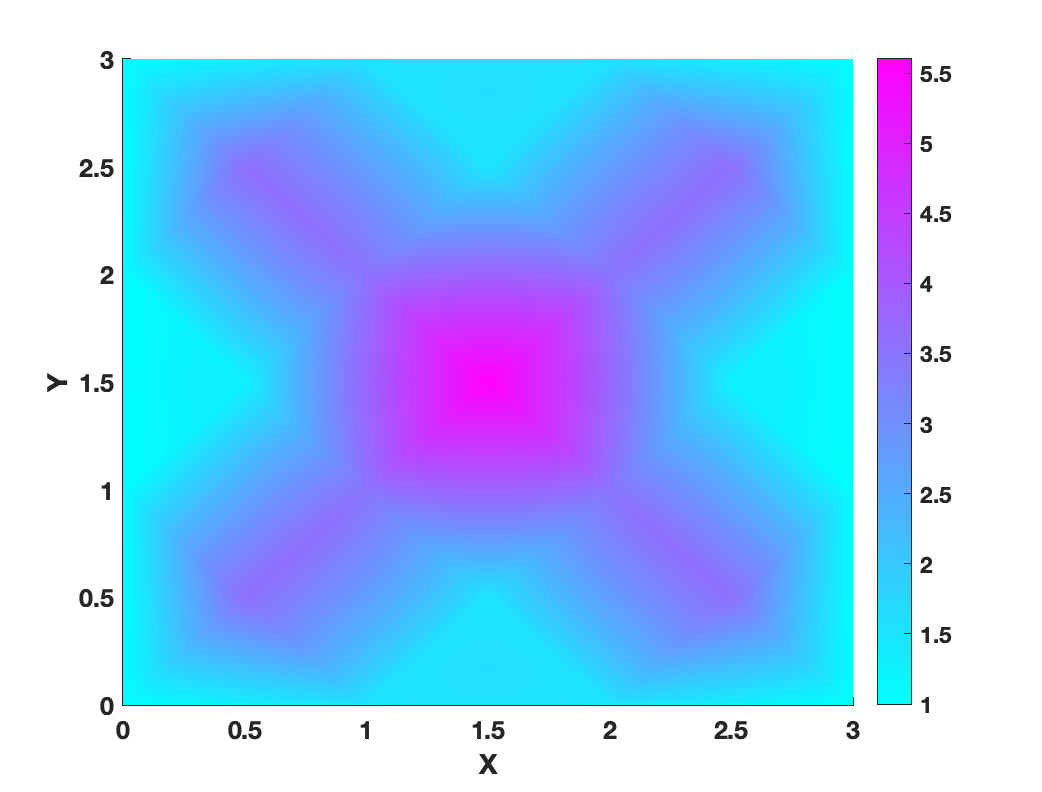}&
	  \includegraphics[trim={.1in 0.1in 0.6in 0.2in},clip,width=0.25\textwidth]{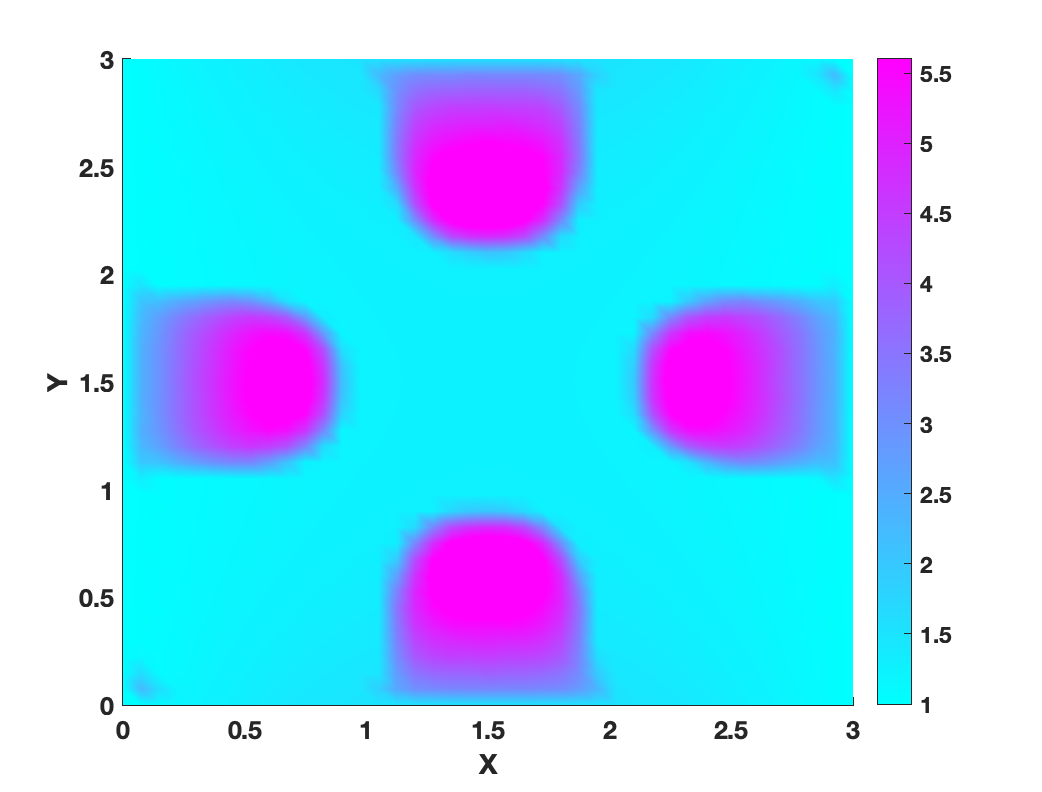}&
  \includegraphics[trim={.1in 0.1in 0.6in 0.2in},clip,width=0.25\textwidth]{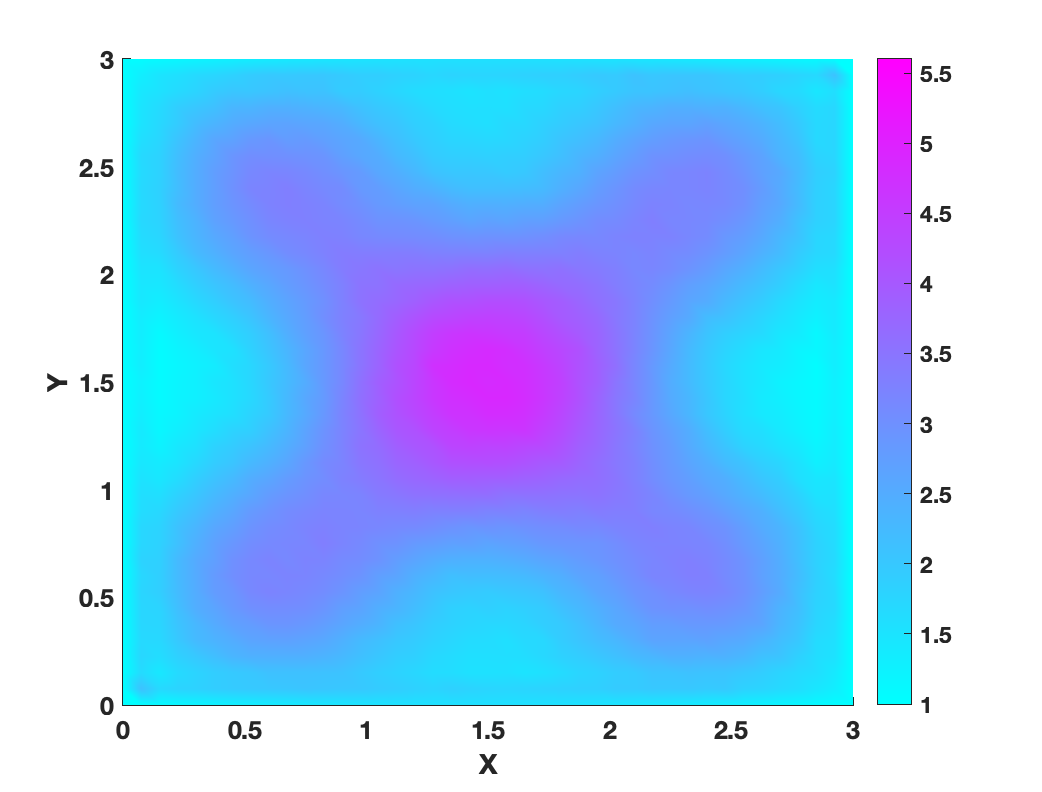}\\
	  \includegraphics[trim={.1in 0.1in 0.6in 0.2in},clip,width=0.25\textwidth]{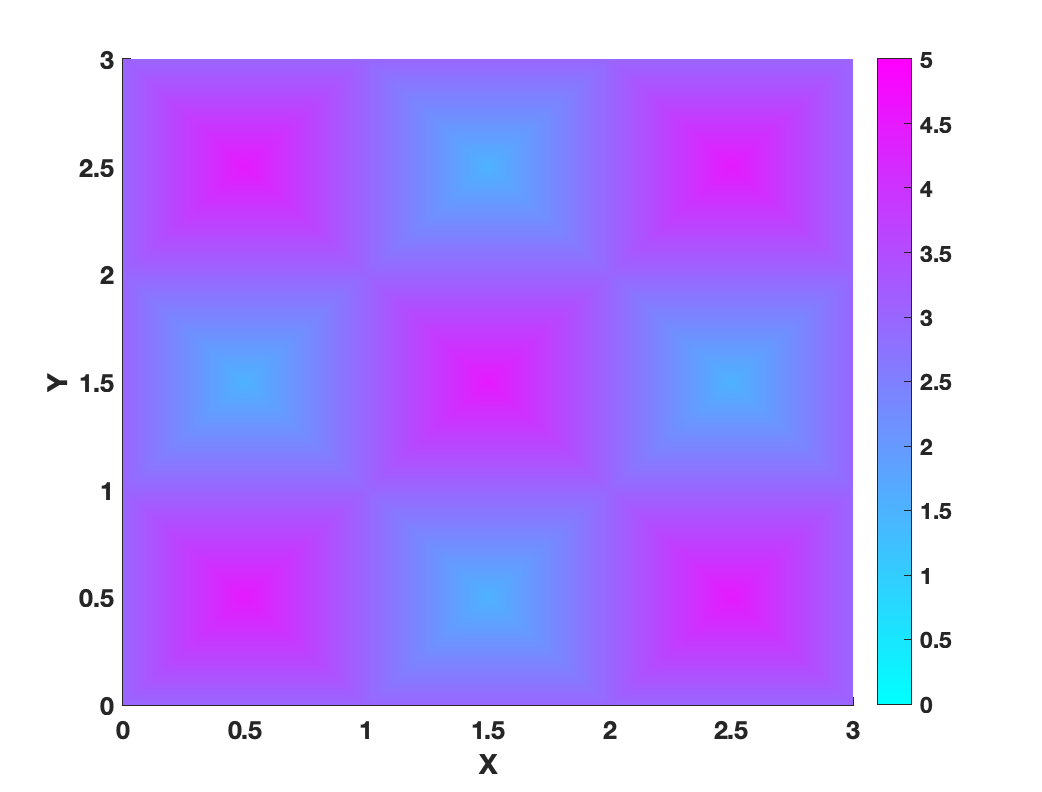}&
	  \includegraphics[trim={.1in 0.1in 0.6in 0.2in},clip,width=0.25\textwidth]{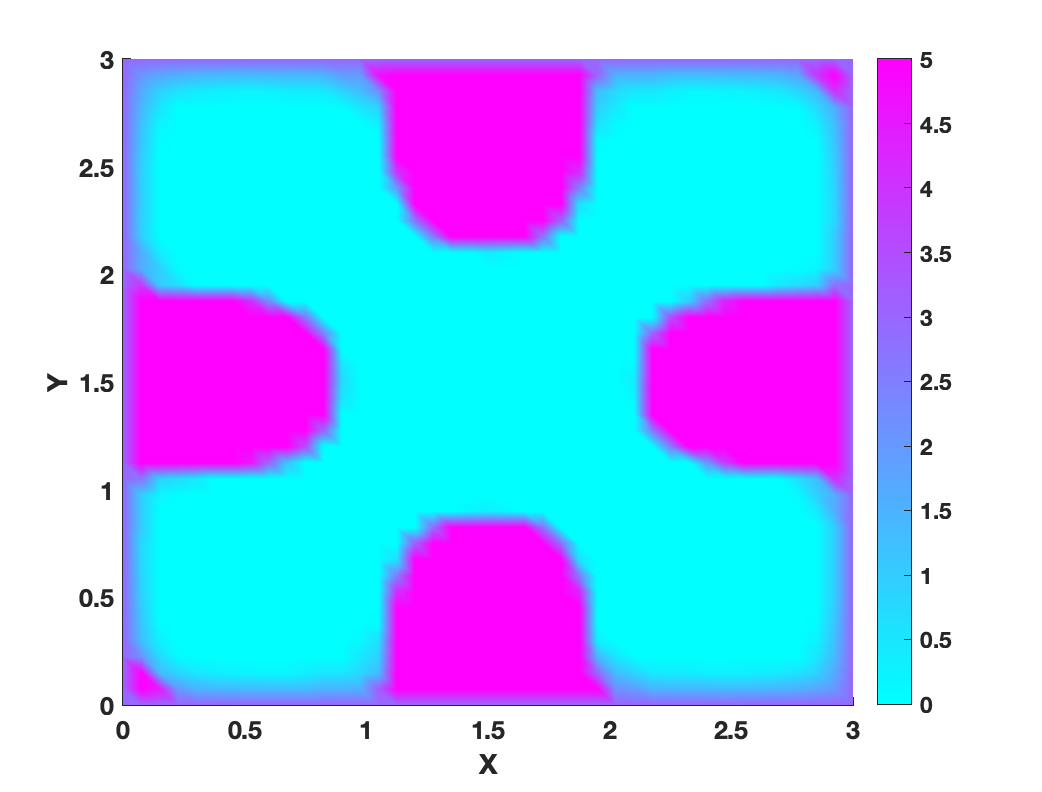}&
  \includegraphics[trim={.1in 0.1in 0.6in 0.2in},clip,width=0.25\textwidth]{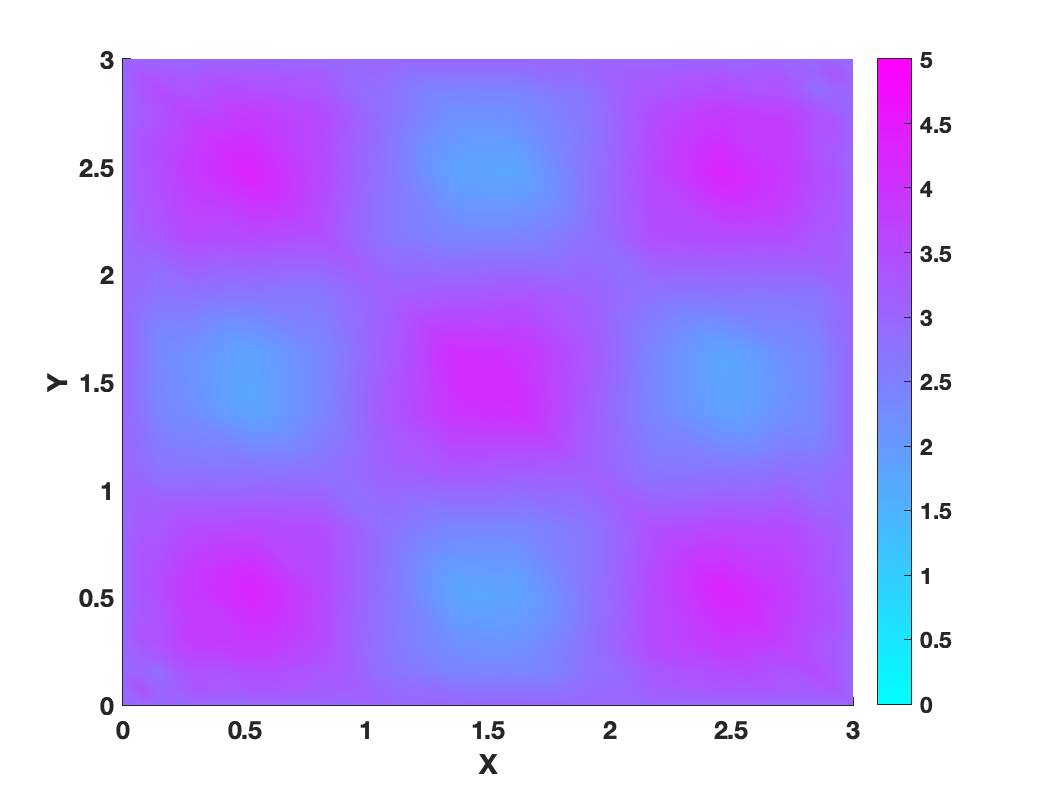}\\
		 (a) Exact $(v_2^\dag,q_2^\dag)$&  (b)  $T_1=0.1$, $T_2=0.01$ &  (c) $T_1=0.1$, $T_2=1$ 
	\end{tabular}
\caption{Profiles of numerical reconstruction with $\al=0.5$, $\tau =0.0025$, $\delta=10^{-3}$. (a) exact initial condition and potential $(v_2^\dag,q_2^\dag)$; (b)  353 iterations, $\|q^{353}-q^{352}\|_{L^2\II}\le 10^{-8}$; (c)  21 iterations, $\|q^{21}-q^{20}\|_{L^2\II}\le 10^{-8}$. } \label{Profile:diver}
\end{figure}

\section{Conclusion}\label{sec:conclusion}
In this paper, we investigated the inverse problem of simultaneously identifying a spatially dependent potential and the initial condition in a subdiffusion model using two terminal observations. The existence, uniqueness, and conditional stability of the solutions were established under weak regularity assumptions by employing the contraction mapping theorem. To address the mildly ill-posed nature of recovering the initial condition, the quasi-boundary value method was employed as a regularization technique.
A fully discrete numerical scheme was proposed, combining the finite element method for spatial discretization and convolution quadrature for temporal discretization. The iterative reconstruction of the initial condition and potential was achieved through a fixed-point algorithm, whose linear convergence was rigorously proven. Additionally, an \textit{a priori} error estimate was derived, offering practical guidance for selecting regularization parameters and discretization mesh sizes based on the noise level.
The theoretical results were grounded in novel error estimates for the direct and backward problems, as well as the established conditional stability. Numerical experiments were conducted to validate the theoretical findings, demonstrating the effectiveness and robustness of the proposed method.
For future work, we aim to explore more challenging scenarios involving boundary measurements or sparse internal measurements.

\section*{Acknowledgements}
The work of J. Yang is supported by the National Science Foundation of China (No.12271240, 12426312), the fund of the Guangdong Provincial Key Laboratory of Computational Science and Material Design, China (No.2019B030301001), and the Shenzhen Natural Science Fund (RCJC 20210609103819018).
The work of Z. Zhou is supported by by National Natural Science Foundation of China (Projects 12422117 and 12426312), Hong Kong Research Grants Council (15303122) and an internal grant of Hong Kong Polytechnic University (Project ID: P0038888, Work Programme: ZVX3). The work of Z. Zhou and X. Wu is also supported by the CAS AMSS-PolyU Joint Laboratory of Applied Mathematics.

\section*{Appendix}\label{apped}
\noindent\textbf{A. Proof of Lemma \ref{lem:Dal-uhn}} 
 \begin{proof}
    Let   $\omega_h^n= u_h^n(q_1,\bar v_1)- u_h^n(q_2,\bar v_2)$.   Then with $\bar \omega^n=(q_2-q_1)( u_h^n(q_2,\bar v_2)+D_h)$, $\omega_h^n\in X_h^0$ solves
    \begin{equation*}
\begin{aligned}
   \bDal \omega_h^n+A_{h}(q_1) \omega_h^n&=P_h  \bar \omega^n\quad \text{with}\quad \omega_h^0=\bar v_1-\bar v_2.
\end{aligned}
\end{equation*}
By means of Laplace transform, 
the fully discrete solution $\omega_h^n$ can be written as 
\begin{equation}\label{eq:repomega-fully}
  \omega_h^n=F_{h,\tau}^n(q_1)\omega_h^0+\tau \sum_{j=1}^n E_{h,\tau}^{n-j}(q_1)P_h\bar \omega^j.
\end{equation}
The governing equation of  $w_h^n$ and the identity $A_h(q)E_{h,\tau}^{n}(q) = - \partial_\tau F_{h,\tau}^{n+1}(q)$ lead to
\begin{align*}
     \bDal \omega_h^n&=-A_h(q_1) \left(F_{h,\tau}^n(q_1)\omega_h(0)+\tau \sum_{j=1}^n E_{h,\tau}^{n-j}(q_1)P_h\bar \omega^j\right)+P_h\bar \omega^n
   =-A_h(q_1)F_{h,\tau}^n(q_1) \omega_h(0)+\partial_\tau \phi_h^n,
\end{align*}
where  $\phi_h^n=\tau \sum_{j=1}^n F_{h,\tau}^{n+1-j}(q_1)P_h\bar \omega^j$.
The first term can be bounded by  Lemma~\ref{lem:op-fully}
\begin{align*}\left\|A_h(q_1)F_{h,\tau}^n(q_1)\omega_h(0)\right\|_{L^2\II}\le ct_n^{-\al}\|\omega_h(0)\|_{L^2\II}\le ct_n^{-\al}\|\bar v_1-\bar v_2\|_{L^2\II}.
\end{align*}
We now  analyze the bound of $(A_h(q_1)F_{h,q_1}(t_n)) ^{-1} \partial_\tau \phi_h^n $.
Lemmas \ref{lem:op-fully}--\ref{Lem:sobembdis} imply
   \begin{align*}
&\quad\left\| (A_h(q_1)F_{h,q_1}(t_n)) ^{-1} \phi_h^n \right\|_{L^2\II}\\&\le \tau \sum_{j=1}^n \left\|F_{h,q_1}(t_n)^{-1}F_{h,\tau}^{n+1-j}(q_1)A_h(q_1)^{-1}P_h\bar \omega^j\right\|_{L^2\II}\\
&\le \tau \sum_{j=1}^n \left\|F_{h,q_1}(t_n)^{-1}F_{h,\tau}^{n+1-j}(q_1)\right\|\left\|A_h(q_1)^{-1}P_h\bar \omega^j\right\|_{L^2\II}\\
&\le  c\|  q_1-q_2 \|_{L^2\II}\tau \sum_{j=1}^n\left( 1+t_{n+1-j}^{-\al}t_n^{\al} \right) \left \|   u_h^j(q_2,\bar v_2)+D_h\right \|_{L^2\II}
\le \bar c_{0,\bar v_2,f,D} t_n \|  q_1-q_2 \|_{L^2\II} .
\end{align*}     
Next, the identity
$t_{n+1} \phi_h^n =\tau \sum_{j=1}^n t_{n+1-j}F_{h,\tau}^{n+1-j}(q_1)(q_1)P_h\bar\omega^j +\tau \sum_{j=1}^n F_{h,\tau}^{j}(q_1)P_h[t_{n+1-j}\bar\omega^{n+1-j}]$ gives
\begin{align*}
   \partial_\tau ( t_{n+1} \phi_h^n)= &\tau \sum_{j=1}^n \partial_\tau \left(t_{n+1-j}F_{h,\tau}^{n+1-j}(q_1)\right)P_h\bar\omega^j+ \tau \sum_{j=1}^n  F_{h,\tau}^{j}(q_1)  P_h\partial_\tau(t_{n+1-j}\bar\omega^{n+1-j})\\= &\tau \sum_{j=1}^n   \left(t_{n-j} \partial_\tau F_{h,\tau}^{n+1-j}(q_1)+F_{h,\tau}^{n+1-j}(q_1)\right)P_h\bar\omega^j+ \tau \sum_{j=1}^n F_{h,\tau}^{j}(q_1)  P_h  \left[t_{n-j}\partial_\tau\bar\omega^{n+1-j}+ \bar\omega^{n+1-j}\right]\\=:&{\rm I}_{h,\tau}^1+ {\rm I}_{h,\tau}^2.
\end{align*}
We shall derive the bounds of $(A_h(q_1)F_{h,q_1}(t_n))^{-1}{\rm I}_{h,\tau}^1$ and $(A_h(q_1)F_{h,q_1}(t_n))^{-1}{\rm I}_{h,\tau}^2 $. First, we bound the term $(A_h(q_1)F_{h,\tau}^n(q_1))^{-1}{\rm I_{h,\tau}^1}$ by Lemmas \ref{lem:op-fully}--\ref{Lem:sobembdis}
\begin{align*}
    &\left\|(A_h(q_1)F_{h,q_1}(t_n))^{-1}{\rm I_{h,\tau}^1}\right\|_{L^2\II}\\
    \le& \tau \sum_{j=1}^n \left \|F_{h,q_1}(t_{n})^{-1} \left(t_{n-j} \partial_\tau F_{h,\tau}^{n+1-j}(q_1)+F_{h,\tau}^{n+1-j}(q_1)\right)\right\|\left\|A_h(q_1)^{-1}P_h\bar\omega^j\right\|_{L^2\II}
       \\\le& c\|q_2-q_1\|_{L^2\II}\tau \sum_{j=1}^n \left(1+t_{n+1-j}^{-\al}t_n^{\al}\right)\left \|   u_h^j(q_2,\bar v_2)+D_h\right \|_{L^2\II}\le\bar c_{0,\bar v_2,f,D}t_n\|q_2-q_1\|_{L^2\II}.
\end{align*}
Meanwhile, applying Lemmas \ref{lem:op-fully}--\ref{Lem:sobembdis} to the term $(A_h(q_1)F_{h,q_1}(t_{n}))^{-1}{\rm I}_{h,\tau}^2$  leads to
\begin{align*}
    &\left\|(A_h(q_1)F_{h,q_1}(t_n))^{-1}{\rm I}_{h,\tau}^2\right\|_{L^2\II}\\
    \le& \tau \sum_{j=1}^n  \left\|F_{h,q_1}(t_{n})^{-1} F_{h,\tau}^{j}(q_1)\right\|\left\|A_h(q_1)^{-1} P_h\left[t_{n-j}\partial_\tau\bar\omega^{n+1-j}+ \bar\omega^{n+1-j}\right]\right\|_{L^2\II}
       \\\le& c\|q_2-q_1\|_{L^2\II}\tau \sum_{j=1}^n \left(1+t_{j}^{-\al}t_n^{\al}\right)\left\|t_{n-j}\partial_\tau  u_h^{n+1-j}(q_2,\bar v_2)+  u_h^{n+1-j}(q_2,\bar v_2)+D_h\right\|_{L^2\II}\\ \le &\bar c_{0,\bar v_2,f,D}t_n\|q_2-q_1\|_{L^2\II}.
\end{align*}
Hence  for any $n>0$,  using the identity $\partial_\tau  (t_{n+1}\phi_h^n)=\phi^n_h+t_n\partial_\tau \phi_h^n $ along with the triangle inequality yields
\begin{align*}
    &t_n\left\| (A_h(q_1)F_{h,q_1}(t_n))^{-1} \partial_\tau  \phi_h^n\right \|_{L^2(\Omega)} \\
    & \le\left\|( A_h(q_1)F_{h,q_1}(t_n))^{-1}\partial_\tau(t_{n+1} \phi_h^n ) \right\|_{L^2(\Omega)} +\left \|\left (A_h(q_1)F_{h,q_1}(t_n)\right)^{-1} \phi_h^n \right\|_{L^2(\Omega)} \\
    &\le \bar c_{0,\bar v_2,f,D}t_n\| q_2-q_1\|_{L^2\II}.
\end{align*}
Therefore, we can derive the estimate of $ \partial_\tau  \phi_h^n$ as
\begin{align*}
  \left \| \partial_\tau  \phi_h^n \right\|_{L^2\II}
 & \le \left\|A_h(q_1)F_{h,q_1}(t_n)\right\|\left \| (A_h(q_1)F_{h,q_1}(t_n))^{-1} \partial_\tau  \phi_h^n \right\|_{L^2(\Omega)}\\
 & \le  \bar c_{0,\bar v_2,f,D}t_n^{-\al}\| q_2-q_1\|_{L^2\II}.
\end{align*}
 This completes the proof of the lemma.
\end{proof}

\noindent\textbf{B. Proof of Lemma \ref{lem:stab-0_reg_nois}} 
\begin{proof} From the representation \eqref{solu:ugamdel_fully}, there holds 
   \begin{align*}
     u_{\gamma,h}^{\delta,0}(q_1) - u_{\gamma,h}^{\delta,0}(q_2)
    = &  \left(F_{h,\tau}^{N_1}(q_2)-F_{h,\tau}^{N_1}(q_1)\right)\left(\gamma I+F_{h,\tau}^{N_1}(q_1)\right)^{-1} u_{\gamma,h}^{\delta,0}(q_2)
   \\& +\left(\gamma I+F_{h,\tau}^{N_1}(q_1)\right)^{-1}\bigg\{ \left(I- F_{h,\tau}^{N_1}(q_2)\right) A_h(q_2)^{-1} P_h \left[f -q_2D_h\right]\\&\quad\quad-\left(I- F_{h,\tau}^{N_1}(q_1)\right) A_h(q_1)^{-1}P_h \left[f -q_1D_h\right]\bigg\}=:{\rm II}_{h,\tau,1}^{N_1}+{\rm II}_{h,\tau,2}^{N_1}.
   \end{align*} 
 Let  ${\rm II}_{h,\tau,1}^{n,i}=F_{h,\tau}^{n}(q_i)\left(\gamma I+F_{h,\tau}^{N_1}(q_1)\right)^{-1} u_{\gamma,h}^{\delta,0}(q_2)$, $i=1,2$.
Then  ${\rm II}_{h,\tau,1}^{N_1}$ solves
 \begin{align*}
      \bDal {\rm II}_{h,\tau,1}^{n}+ A_h(q_2) {\rm II}_{h,\tau,1}^{n}=P_h\lsb (q_1-q_2){\rm II}_{h,\tau,1}^{n,1} \rsb\quad \text{with}\quad  {\rm II}_{h,\tau,1}^{0}=0.
 \end{align*}
 Employing solution representation   \eqref{eq:repomega-fully}  gives 
 \begin{align*}
     {\rm II}_{h,\tau,1}^{N_1}
      =&\tau \sum_{j=1}^{\lceil\frac{N_1}{2}\rceil-1}A_h(q_2)E_{h,\tau}^{N_1-j}(q_2)A_h(q_2)^{-1}P_h\left[(q_1-q_2){\rm II}_{h,\tau,1}^{j,1}\right ]\\&+\tau \sum_{j=\lceil\frac{N_1}{2}\rceil}^{N_1}A_h(q_2)E_{h,\tau}^{N_1-j}(q_2)A_h(q_2)^{-1}P_h\left[(q_1-q_2){\rm II}_{h,\tau,1}^{j,1}\right ]=:  {\rm II}_{h,\tau,1,1}^{N_1}+{\rm II}_{h,\tau,1,2}^{N_1}.
 \end{align*}
 Lemma~\ref{lem:op-fully} and Lemma~\ref{Lem:sobembdis} imply
 \begin{align*}
     \lnorm {\rm II}_{h,\tau,1,1}^{N_1}\rnorm_{L^2\II}\le&\tau \sum_{j=1}^{\lceil\frac{N_1}{2}\rceil-1} \left\|A_h(q_2)E_{h,\tau}^{N_1-j}(q_2)\right\|\left\|A_h(q_2)^{-1}P_h\left[(q_1-q_2){\rm II}_{h,\tau,1}^{j,1}\right ]\right\|_{L^2\II}\\
     \le &c\|q_1-q_2\|_{L^2\II}\tau \sum_{j=1}^{\lceil\frac{N_1}{2}\rceil-1}t_{N_1-j+1}^{-1}\left\|{\rm II}_{h,\tau,1}^{j,1}\right\|_{L^2\II}\\
      \le& c\|q_1-q_2\|_{L^2\II}\tau \sum_{j=1}^{\lceil\frac{N_1}{2}\rceil-1}t_{N_1-j+1}^{-1}(1+t_j^{-\al}T_1^\al)\left\| u_{\gamma,h}^{\delta,0}(q_2)\right\|_{L^2\II}\\
     \le &c\left\| u_{\gamma,h}^{\delta,0}(q_2)\right\|_{L^2\II}\|q_1-q_2\|_{L^2\II}.
 \end{align*}
 Applying the identity  $A_h(q)E_{h,\tau}^n(q)=-\partial_\tau F_{h,\tau}^{n+1}(q)$ to $ {\rm II}_{h,\tau,1,2}^{N_1}$ yields
 \begin{align*}
     {\rm II}_{h,\tau,1,2}^{N_1}
     =&\tau \sum_{j=\lceil\frac{N_1}{2}\rceil}^{N_1} F_{h,\tau}^{N_1-j+1}(q_2)A_h(q_2)^{-1}P_h\left[(q_1-q_2)\partial_\tau {\rm II}_{h,\tau,1}^{j,1}\right ]\\&+ F_{h,\tau}^{\lfloor\frac{N_1}{2}\rfloor}(q_2)A_h(q_2)^{-1}P_h\left[(q_1-q_2){\rm II}_{h,\tau,1}^{{\lceil\frac{N_1}{2}\rceil},1}\right ]-A_h(q_2)^{-1}P_h\left[(q_1-q_2) {\rm II}_{h,\tau,1}^{N_1,1}\right ].
 \end{align*}
 Therefore, from  Lemma~\ref{lem:op-fully}, we can derive that  
 \begin{align*}
      \lnorm{\rm II}_{h,\tau,1,2}^{N_1}\rnorm_{L^2\II}\le& c\|q_1-q_2\|_{L^2\II}(1+\tau \sum_{j=\lceil\frac{N_1}{2}\rceil}^{N_1}t_j^{-1}(1+t_j^{-\al}T_1^{\al}))\left\| u_{\gamma,h}^{\delta,0}(q_2)\right\|_{L^2\II}\\\le&  c\left\| u_{\gamma,h}^{\delta,0}(q_2)\right\|_{L^2\II}\|q_1-q_2\|_{L^2\II}.
 \end{align*}
 Similarly, let $ {\rm   II}_{h,\tau,2}^{n,i}=(I- F_{h,\tau}^{n}(q_i)) A_h(q_i)^{-1}  P_h [f -q_iD_h]$, $i=1,2$. Then   ${\rm   \bar {II}}_{h,\tau,2}^{N_1}={\rm   II}_{h,\tau,2}^{N_1,1}-{\rm   II}_{h,\tau,2}^{N_1,2}$ solves
 \begin{align*}
     \bDal {\rm   \bar {II}}_{h,\tau,2}^{n}+A_h(q_1) {\rm   \bar {II}}_{h,\tau,2}^{n}= P_h[ (q_2-q_1)({\rm   II}_{h,\tau,2}^{n,2}+D_h)]\quad \text{with}\quad   {\rm   \bar {II}}_{h,\tau,2}^{0}=0.
 \end{align*}
Using solution representation and  the identities  $  {\rm   II}_{h,\tau,2}^{N_1}=-\lb \gamma I+F_{h,\tau}^{N_1}(q_1)\rb^{-1} {\rm   \bar {II}}_{h,\tau,2}^{N_1}$ and $A_h(q)E_{h,\tau}^{n}(q) = - \partial_\tau F_{h,\tau}^{n+1}(q)$   leads to 
 \begin{align*}
    {\rm   II}_{h,\tau,2}^{N_1}
    =&\left(A_h(q_1)\lb \gamma I+F_{h,\tau}^{N_1}(q_1)\rb\right)^{-1}\tau \sum_{j=1}^{N_1} \left(\partial_\tau F_{h,\tau}^{N_1-j+1}(q_1)P_h {\rm \bar {II}}^j\right)\\=&\left(A_h(q_1)\lb \gamma I+F_{h,\tau}^{N_1}(q_1)\rb\right)^{-1}\partial_\tau\left(\tau \sum_{j=1}^{N_1}  F_{h,\tau}^{N_1-j+1}(q_1)P_h {\rm \bar {II}}^j\right)-\left(A_h(q_1)\lb \gamma I+F_{h,\tau}^{N_1}(q_1)\rb\right)^{-1}P_h {\rm \bar {II}}^{N_1}\\=:& {\rm   II}_{h,\tau,2,1}^{N_1}+ {\rm   II}_{h,\tau,2,2}^{N_1}, 
 \end{align*}
 where $ {\rm \bar {II}}^j=(q_2-q_1)( {\rm   II}_{h,\tau,2}^{j,2}+D_h)$. 
 From Lemma~\ref{lem:Dal-uhn}, it holds that
 \begin{align*}
     \lnorm{\rm   II}_{h,\tau,2,1}^{N_1}\rnorm_{L^2\II}&\le\left\|\left(A_h(q_1)F_{h,\tau}^{N_1}(q_1)\right)^{-1}\partial_\tau\left(\tau \sum_{j=1}^{N_1}  F_{h,\tau}^{N_1-j+1}(q_1)P_h{\rm \bar {II}}^j\right)\right\|_{L^2\II}\le  c_{f,D}\|q_1-q_2\|_{L^2\II}.
 \end{align*}
 Additionally,      Lemma~\ref{lem:op-fully} and Lemma~\ref{Lem:sobembdis} give
 \begin{align*}
    &\quad\lnorm{\rm   II}_{h,\tau,2,2}^{N_1}\rnorm_{L^2\II}\\&\le \left\|\left(A_h(q_1) F_{h,\tau}^{N_1}(q_1)\right)^{-1}P_h{\rm \bar {II}}^{N_1}\right\|_{L^2\II}\le c(1+T_1^{\al})\|q_2-q_1\|_{L^2\II}\left(\left\|{\rm   II}_{h,\tau,2}^{N_1,2}\right\|_{L^\infty\II}+\left\|D_h\right\|_{L^\infty\II}\right)
\\&\le c(1+T_1^{\al})\|q_2-q_1\|_{L^2\II}\left(\left\|A_h(q_2){\rm  II}_{h,\tau,2}^{N_1,2}\right\|_{L^2\II}+\left\|D_h\right\|_{L^\infty\II}\right)\le  c_{f,D}(1+T_1^{\al})\|q_2-q_1\|_{L^2\II}.
 \end{align*}
 This completes the proof of the lemma from the   following \textsl{a priori} estimate  of  $\left\| u_{\gamma,h}^{\delta,0}(q_2)\right\|_{L^2\II}$ in Lemma~\ref{lem:pfopri}.
\end{proof}
 \begin{lemma}\label{lem:pfopri}  
For  given $f, q, b$, $g_1$ and any $q\in \A$, suppose that  $\gamma^{-1}(\delta+h^2)\le c$ for some constant $c$.  Then it holds that $$ \lnorm u _{\gamma,h}^{\delta,0}(q)\rnorm_{L^2\II}\le c_{g_1,f,D}(1+T_1^\al).$$
\end{lemma}
\begin{proof} Form equation \eqref{eqn:udisdeelayh}, there holds
   \begin{equation*}
   \begin{aligned}
  \lnorm u _{\gamma,h}^{\delta,0}(q)\rnorm_{L^2\II}\le &\lnorm(\gamma I+F_{h,\tau}^{N_1}(q))^{-1}(P_h( g_1^\delta-D_h)-P_h (g_1-D))\rnorm_{L^2\II}  \\& +\lnorm(\gamma I+F_{h,\tau}^{N_1}(q))^{-1}(P_h-R_h(q))\bar g_1\rnorm_{L^2\II} + \lnorm\big(A_h(q)(\gamma I+F_{h,\tau}^{N_1}(q))\big)^{-1}A_h(q)R_h(q)\bar g_1\rnorm_{L^2\II}\\&+ \lnorm(A_h(q)(\gamma I+F_{h,\tau}^{N_1}(q)))^{-1}(I-F_{h,\tau}^{N_1}(q)) P_h [f - qD_h]\rnorm_{L^2\II}\\\le&c (\gamma^{-1}\delta+\gamma^{-1}h^2)+c\gamma^{-1}h^2\|\bar g_1\|_{\dot H^2\II}+c(1+T_1^\al)(\| \bar  g_1\|_{\dot H^2\II}+\|f\|_{L^2\II}+\|D\|_{L^2\II})\\
  \le& c_{g_1,f,D}(1+T_1^\al),
\end{aligned}   
\end{equation*}
where we use the fact that
$$ \|P_h( g_1^\delta-D_h)-P_h (g_1-D) \|_{L^2\II}\le \|g_1^\delta-g\|_{L^2\II}+\|D_h-D\|_{L^2\II}\le \delta+ch^2.$$
\end{proof}


\bibliographystyle{abbrv}
\bibliography{bibfile}
\end{document}